\documentclass{amsproc}

\newtheorem{theorem}{Theorem}

\newtheorem{proposition}{Proposition}[section]

\newtheorem{lemma}[proposition]{Lemma}
\newtheorem{corollary}[proposition]{Corollary}

\theoremstyle{definition}
\newtheorem{definition}[proposition]{Definition}
\newtheorem{example}[proposition]{Example}

\theoremstyle{remark}
\newtheorem{remark}[proposition]{Remark}

\numberwithin{equation}{section}

\usepackage[stable]{footmisc}
\usepackage[dvipsnames]{xcolor}

\DeclareSymbolFontAlphabet{\mathcal}{symbols}

 \usepackage{tikz}
\usepackage{mathrsfs}

\usepackage{dsfont}
\usepackage{eufrak}

\usepackage[all]{xy}
\usepackage{framed}
\usepackage{hyperref}

\usepackage{mathabx}
\usepackage{amsthm}

\usepackage{enumerate}
\usepackage{graphicx}

\newcommand{\bl}{\begin{lemma}}
\newcommand{\bp}{\begin{proposition}}
\newcommand{\bt}{\begin{theorem}}
\newcommand{\bc}{\begin{corollary}}
\newcommand{\be}{\begin{equation}}
\newcommand{\bee}{\begin{equation*}}
\newcommand{\bd}{\begin{definition}}
\newcommand{\bex}{\begin{example}}
\newcommand{\bpr}{\begin{proof}}

\newcommand{\el}{\end{lemma}}
\newcommand{\ep}{\end{proposition}}
\newcommand{\et}{\end{theorem}}
\newcommand{\ec}{\end{corollary}}
\newcommand{\ee}{\end{equation}}
\newcommand{\eee}{\end{equation*}}
\newcommand{\ed}{\end{definition}}
\newcommand{\eex}{\end{example}}
\newcommand{\epr}{\end{proof}}

\newcommand{\secref}[1]{Section~\ref{#1}}

\newcommand{\thmref}[1]{Theorem~\ref{#1}}
\newcommand{\propref}[1]{Proposition~\ref{#1}}
\newcommand{\lemref}[1]{Lemma~\ref{#1}}
\newcommand{\corref}[1]{Corollary~\ref{#1}}

\newcommand{\remref}[1]{Remark~\ref{#1}}

\newcommand{\defref}[1]{Definition~\ref{#1}}

\newcommand{\partref}[1]{Part~\ref{#1}}
\newcommand{\pagref}[1]{page~\pageref{#1}}
\def\R{{\mathbb R}}

\def\ov{\overline}

\newcommand{\cal}[1]{\mathcal #1}

\def\tcL{{\widetilde{\mathcal L}}}

\def\1{{\boldsymbol 1}}
\def\gd{{\mathfrak{d}}}
\def\gC{{\mathfrak{C}}}
\def\gH{{\mathfrak{H}}}

\def\gT{{\mathfrak{T}}}

\def\tc{{\mathtt c}}
\def\tp{{\mathtt p}}

\def\tv{{\mathtt v}}
\def\tw{{\mathtt w}}
\def\N{\mathbb{N}}
\def\Q{\mathbb{Q}}
\def\R{\mathbb{R}}
\def\Z{\mathbb{Z}}

\def\im{{\rm Im\,}}

\def\Hom{{\rm Hom}}

\def\id{{\rm id}}

\def\sub{{\rm Sub}}
\def\wsub{{\widetilde{{\rm Sub}\,}}}
\def\codim{{\rm codim\,}}
\def\pr{{\rm pr}}
\def\tN{{\widetilde{N}}}
\def\tT{{\widetilde{\mathrm{T}}}}
\def\ext{{\rm Ext}}

\def\der{{\mathtt{der}\,}}
\def\rc{{\mathring{\tc}}}
\def\fc{{\boldsymbol all}}

 \newcommand{\menos}{\backslash} 
      
\def\epsi{{\varepsilon}}

 \usepackage[applemac]{inputenc}
    
\newcommand{\bi}[2]{{#1}^{^{#2}}}

\newcommand{\tres}[3]{{#1}^{^{ #2 }}_{_{#3}}}

\newcommand{\Hiru}[3]{{#1}^{^{#2}}{( #3 )}}
\newcommand{\hiru}[3]{{#1}_{_{#2}}{( #3 )}}
\newcommand{\lau}[4]{{#1}^{^{#2}}_{_{#3}}{( #4 )}}

\newcommand{\parr}[1]{ \medskip \begin{center} { \bf   \  #1.} \end{center}}

\newcommand{\IH}{\mathscr H}

\def\tDelta{{\widetilde{\Delta}}}
\def\tdelta{{\widetilde{\delta}}}
\def\tcap{\,{\widetilde{\cap}\;}}
\def\tg{{\tilde{g}}}
\def\fa{{\vartriangleleft}}
\def\si{{\mathtt{Simp}}}

\title{Blown-up  intersection cohomology}

\author{David Chataur}
\address{Lafma\\
Universit\'e de Picardie Jules Verne\\
33, rue Saint-Leu\\
80039 Amiens Cedex~1\\
         France}
\email{David.Chataur@u-picardie.fr}

\author{Martintxo Saralegi-Aranguren}
\address{Laboratoire de Math{\'e}matiques de Lens\\  
      EA 2462 \\
      Universit\'e d'Artois\\
         SP18, rue Jean Souvraz\\
          62307 Lens Cedex\\
         France}
\email{martin.saraleguiaranguren@univ-artois.fr}

\author{Daniel Tanr\'e}
\address{D\'epartement de Math{\'e}matiques\\
         UMR 8524 \\
         Universit\'e de Lille~1\\
         59655 Villeneuve d'Ascq Cedex\\
         France}
\email{Daniel.Tanre@univ-lille1.fr}
\thanks{The third author is partially supported by the MINECO and  FEDER research project MTM2016-78647-P  and ANR-11-LABX-0007-01  ``CEMPI''}

\begin{document}


 
 \today

\begin{abstract} 
In previous works, we have introduced the blown-up intersection cohomology and used it
to extend  Sullivan's minimal models theory to the framework of pseudomanifolds,
and to give a positive answer to a conjecture of M. Goresky and W. Pardon on Steenrod squares in intersection homology.

In this paper, we establish the main properties of this cohomology. One of its major feature is the existence 
of cap and cup products for any filtered space and any commutative ring of coefficients, at the cochain level. Moreover, we show that
each stratified map induces an homomorphism between the blown-up intersection cohomologies, compatible with
the cup and cap products. We prove also its topological invariance in the case of a pseudomanifold with no codimension one strata.
Finally, we compare it with the intersection cohomology studied by G. Friedman and J.E. McClure.

A great part of our results involves general perversities, defined independently on each stratum,
and a tame intersection homology adapted to  large perversities.
\end{abstract}

\maketitle

\tableofcontents
\section*{Introduction}
Intersection homology  was defined by  M.~Goresky and R.~MacPherson  in the case of PL-pseudomanifolds in \cite{MR572580} and with sheaf theory in  \cite{MR696691} where the authors restore Poincar\'e duality for pseudomanifolds:  if  $X$ is a compact oriented pseudomanifold of dimension $n$
and $\ov p, \ov q$ are two complementary  perversities,  there is an isomorphism
$$\lau H {\ov p}
r {X;\Q}
\cong
 \lau H {n-r}{\ov q} {X;\Q},
$$
with $\lau H {n-r}
{\ov q} {X;\Q} := \hom(\lau H {\ov q} {n-r}{X;\Q};\Q)$.
This restoration also presents
some disadvantages :
\begin{enumerate}
\item[(1)] If we don't restrict ourselves to perversities comprised between the zero $\ov  0$
and the top $\ov t$ perversities respecting some growth condition, then we don't
have Poincar\'e duality anymore and we need to use the more general {\em tame
intersection cohomology}
$\lau \gH *{\ov p}{X;R}$ of
 $X$ (see
 \defref{tameNormHom}),

\item[(2)]  The cohomology is defined via a linear dual process and not via a geometrical one,

\item[(3)]  This duality is not satisfied for any commutative ring $R$. Especially for a commutative
ring of positive characteristic properties of ``torsion freeness" for	
the links of the singularities are needed.
\end{enumerate}

{\em Blown-up intersection cohomology} aims to overcome some of these difficulties.
The goal in this article is to develop the main properties of the  Blown-up intersection
cohomology $\lau \IH * {\ov p} {X;R}$ for filtered spaces and more precisely for CS sets.
This theory  was
successfully used in \cite{CST1} to extend the theory of Sullivan's minimal models for intersection cohomology and in \cite{CST2} to answer a conjecture on Steenrod squares.

\medskip
The first sections contain some  reminders about the type of
stratified spaces and maps used in the rest of the paper. After that come the
definitions by a local system of  the
blown-up complex $\lau\tN*{\ov p}{X;R}$, first for a regular simplex, then by a limit argument
for a general weighted simplicial complex in \propref{prop:pasdeface} and finally for
a perverse space $(X,\ov p)$ in \defref{def:thomwhitney}.
The interesting features  are that

\begin{enumerate}
\item[(1)] the perversities $\ov p$ used here are completely general, lifting the restriction
on the perversities used for classical intersection cohomology,

\item[(2)] it is well defined for any commutative ring $R$ as coefficients ring, regardless
of its characteristic.
\end{enumerate}

The existence and properties of a cup product for this cohomology are proved in
\propref{42} by the existence of an associative multiplication
 $$
- \cup  - \colon \lau \tN k {\ov p} {X;R} \otimes \lau \tN   \ell  {\ov q}   {X;R}
\to  \lau \tN  { \ell+k}   {\ov p +\ov q} {X;R},
$$ inducing a graded commutative  multiplication with unity
 $$
- \cup  - \colon \lau \IH k {\ov p} {X;R} \otimes \lau \IH   \ell  {\ov q}   {X;R}
\to  \lau \IH  { \ell+k}   {\ov p +\ov q} {X;R}.
$$ The same goes for the cap product, proving in \propref{prop:cap} that the cap product
allows an interaction between the blown-up intersection cohomology $\lau \IH * {\ov p} {X;R}$
and
the tame intersection homology $\lau \gH  {\ov p}* {X;R}$.

\medskip
The first part ends with two sections detailing the behavior of a stratified map
with respect to blown-up intersection cohomology, first at the local level and then
globally. The key result being here the \thmref{MorCoho} stating that any stratified map
$f \colon  X \to  Y$ as defined in the \defref{def:applistratifieeforte} induces a chain map
$$
f^*\colon \lau \tN k {\ov p} {Y;R}  \to  \lau \tN k {\ov p} {X;R}
$$
compatible with the cup product.
This theorem also states that the induced maps $f_*$
and $f^*$ are  compatible with the interaction between the blown-up intersection
cohomology $\lau \IH   *  {\ov p}   {X;R}$  and the tame intersection homology $\lau \gH  {\ov p}* {X;R}$
of \propref{prop:cap}.

\medskip
The second part of the paper is devoted to the various properties of the blown-up
intersection cohomology and their proofs. The first property comes from the \thmref{thm:Upetits}
 stating the existence of the complex of $\cal U$-small chains $
\lau \tN {*,\cal U}{\ov p} {X,R}$ where $\cal U$ is an
open cover of the perverse space $(X,\ov p)$ and the existence of a quasi-isomorphism
$
\lau \tN {*}{\ov p} {X,R} \to \lau \tN {*,\cal U}{\ov p} {X,R}$. This leads then to the \thmref{thm:MVcourte} where the Mayer-
Vietoris exact sequence is proved. The two next sections detail results about explicit
computations with the blown-up intersection cohomology, the first one is the
\thmref{prop:isoproduitR} determining the blown-up intersection cohomology of the
product of a filtered space $X$ with the line $\R$, from the existence of an isomorphism
$$
\lau \IH   *  {\ov p}   {X \times \R;R} \cong \lau \IH   *  {\ov p}   {X;R}.
$$
The second result is the \thmref{prop:coneTW} giving the computation for the blown-up intersection
cohomology of an open cone $\rc X$ over a compact filtered space $X$.

\smallskip
The paper finishes with the two following theorems. The \thmref{thm:TWGMcorps} compares
the blown-up intersection cohomology $
\lau \IH   *  {\ov p}   {X;R}$
 to the already known tame intersection
cohomology $
\lau \gH   * {\ov p}   {X;R}
$ for $(X, \ov p)$ a paracompact separable perverse CS set
and proves that if $R$ is a field or if the space $X$  is a locally $(D\ov p,R)$-torsion free
pseudomanifold there is an isomorphism
$$
\lau \IH   *  {\ov p}   {X;R}
\cong \lau \gH   * {D\ov p}   {X;R}
$$
where $D\ov p$ is the complementary perversity. The last \thmref{inv} states that the blown-up intersection cohomology is a topological
invariant when working with GM-perversities and separable paracompact CS sets with no codimension one strata .

\bigskip

We fix for the sequel a commutative ring $R$ with unity. All  (co)homologies in this work are considered with coefficients in $R$. For a topological space $X$, we denote by $\tc X= X \times [0,1]/ X \times \{ 0\}$ the \emph{cone} on $X$ and $\rc X = X \times [0,1[/ X \times \{ 0\}$ the \emph{open cone} on $X$.

\bigskip
We thank the anonymous referee for her/his comments and suggestions which have in particular contributed to improve the writing
and the organization of this introduction.

\part{Blown-up  intersection cohomology. Stratified maps}\label{part:TWcoho}

We introduce the main notion of this work: the blown-up  intersection cohomology and its associated cup and cap products.

\section{Some reminders.}\label{sec:rappels}

This section contains the basic definitions and properties of the main notions used in this work.

\bd\label{def:espacefiltre} 
A \emph{filtered space} is a  Hausdorff topological space endowed with a filtration by closed sub-spaces 
\bee
\emptyset = X_{-1} \subseteq X_0\subseteq X_1\subseteq\ldots \subseteq X_{n-1} \subsetneqq X_n=X.
\eee
The
\emph{formal dimension} of $X$ is $\dim X=n$. 

The non-empty  connected components of $X_{i}\backslash X_{i-1}$ are the  \emph{strata} of $X$.
Those of $X_n\backslash X_{n-1}$ are \emph{regular strata}, while the others are  \emph{singular strata.}  
The family of strata of $X$ is denoted by $\cal S_X$.
The {\em singular set} is $X_{n-1}$, denoted by $\Sigma_X$ or simply $\Sigma$. 
 The {\em formal dimension} of a stratum $S \subset X_i\backslash X_{i-1}$ is $\dim S=i$.
The {\em formal codimension} of $S$ is $\codim S = \dim X -\dim S$.
\ed

An open subset $U$ of $X$ can be provided with the \emph{induced filtration,} defined by
$U_i = U \cap X_{i}$.
If $M$ is a topological manifold, the \emph{product filtration} is defined by 
$\left(M \times X\right) _i = M \times X_{i}$ (see remarks about shifted filtrations of subsections \ref{SF1} and \ref{SF}).

\smallskip

The more restrictive concept of stratified space provides a better behavior of the intersection (co)homology with regard to  continuous maps.

\bd
A \emph{stratified space} is a filtered space verifying the following frontier condition: 
for any two strata $S, S'\in \cal S_X$ such that
$S\cap \overline{S'}\neq \emptyset$ then $S\subset \overline{S'}$.
\ed

In their work (\cite{MR572580}, \cite{MR696691}), Goresky and MacPherson proved that the intersection (co)ho\-mo\-logy has richer properties on a singular space $X$ with  local conical structure: these are the classical stratified pseudomanifolds. The local structure is characterized by the fact that any stratum $S$ is a manifold having a conical transversal structure over the link $ L$. This link must be in turn a compact stratified pseudomanifold.
Friedman \cite{LibroGreg} observed  that we can suppose the link $L$ to be only a compact filtered space to preserve the (co)homological properties of $X$. These are the CS sets of 
 L. Siebenman, \cite{MR0319207}.

\begin{definition}
A \emph{CS set} of dimension $n$ is a filtered space,
$$
\emptyset\subset X_0 \subseteq X_1 \subseteq \cdots \subseteq X_{n-2} \subseteq X_{n-1} \subsetneqq X_n =X,
$$
such that, for each $i$, 
$X_i\backslash X_{i-1}$ is a topological manifold of dimension $i$ or the empty set. 
Moreover, for each point $x \in X_i \backslash X_{i-1}$, $i\neq n$, there exist
\begin{enumerate}[(i)]
\item an open neighborhood $V$ of  $x$ in $X$, endowed with the induced filtration,
\item an open neighborhood $U$ of $x$ in $X_i\backslash X_{i-1}$, 
\item a compact filtered space $L$,
with formal dimension $n-i-1$, where the open cone, $\rc L$, is provided with the {\em conical filtration}, $(\rc L)_{j}=\rc L_{j-1}$, with $\rc \emptyset = \tv$ the apex of $\rc L$,%
\item a homeomorphism, $\varphi \colon U \times \rc L\to V$, 
such that
\begin{enumerate}[(a)]
\item $\varphi(u,\tv)=u$, for each $u\in U$, 
\item $\varphi(U\times \rc L_{j})=V\cap X_{i+j+1}$, for each $j\in \{0,\ldots,n-i-1\}$.
\end{enumerate}
\end{enumerate}
The pair  $(V,\varphi)$ is a   \emph{conical chart} of $x$
 and the filtered space $L$ is a \emph{link} of $x$. 
The CS set is called  \emph{normal} if the links are connected.
\end{definition}

In the above definition, the links are always non-empty sets.
Therefore, the open subset $ X_ {n} \backslash X_ {n-1} $ is dense.
Links are not necessarily CS sets but they are always filtered spaces.
Note also that the links associated to points living in the same stratum may be not homeomorphic
but they always have the same intersection homology, see for example \cite [Chapter 5] {LibroGreg}.
Finally, note that any open subset of a CS set is a CS set, that any CS set is a locally compact stratified space \cite[Theorem G]{CST1}
and that a paracompact CS set
 is metrizable \cite [Proposition 1.11] {CST3}.

The pseudomanifolds are special cases of CS sets. Their definition varies in the literature;
we consider here the original definition of M. Goresky and R. MacPherson \cite{MR572580}.

\begin{definition}\label{def:pseudomanifold}
An \emph{$n$-dimensional  pseudomanifold } (or simply pseudomanifold) is an $n$-dimensional CS set, where  the link $L$ of a point $x \in X_i\menos X_{i-1}$ is an $(n-i-1)$-dimensional pseudomanifold. We refer to a pseudomanifold such that  $X_{n-1} = X_{n-2}$ as a \emph{classical $n$-dimensional  pseudomanifold }.
\end{definition}

There are several notions of stratified maps.

\begin{definition}\label{def:applistratifieeforte}
Let $f\colon X \to Y$ be a continuous map between two stratified spaces.
The map $f$ is a  \emph{stratified map,}  if it sends a stratum $S$ of $X$ in a stratum $\bi S f$ of $Y$, $f(S) \subset \bi S f$, verifying $\codim S \geq \codim \bi Sf$.

The stratified map $f$ is a \emph{stratified homeomorphism,} if $f$ and   $f^{-1}$ are stratified maps.

The map $f$ is a  \emph{stratum preserving stratified map} if $n=\dim X = \dim Y$ and $f^{-1}(Y_{n-\ell})=X_{n-\ell}$, for any $\ell \in \{0,\dots,n\}$.
\end{definition}

This notion of stratified map does not match exactly  with those found in \cite{LibroGreg}  or  \cite{MR696691}.

There are two families of perversities: filtration depending (introduced by Goresky-MacPherson  \cite{MR572580}) and the stratification depending (introduced by MacPherson \cite{RobertSF}).

\begin{definition}\label{def:perversite} 
A \emph{Goresky-MacPherson perversity } (or \emph{$GM$-perversity}) (see \cite{MR572580}) is  a map $ \ov {p} \colon \N \to \Z $, verifying
$\ov p(0)=\ov{p}(1)=\ov{p}(2)=0$ and $\ov p (i) \leq \ov p (i+1) \leq \ov p (i) +1$ for each $i\geq 2$.

\end{definition}

We use in this work the notion of perversity introduced  by MacPherson \cite{RobertSF}
and also present in \cite{MR1245833,MR2210257,MR2721621,MR2461258,MR2796412,MR3046315}.
Unlike classic perversities, these  perversities 
are not maps depending only on the codimension of the strata but are maps defined on the strata themselves.

\begin{definition}\label{def:perversitegen}
A \emph {perversity on a filtered space} $ X$ is a map
$ \ov {p} \colon \cal S_X \to \Z \cup \{\pm\infty\}$ taking the value ~ 0 on the regular strata.
The  pair $(X, \ov {p}) $ is called a \emph {perverse space.} 
If $ X $ is a  CS set, we will say that $ (X, \ov {p}) $ is a \emph {perverse CS set.}

The \emph{top perversity} is  the perversity defined by $ \ov {t} (S) = \codim (S) -2 $
  on singular strata. The {\em complementary perversity} of a perversity $\ov p$ is the perversity $D\ov p = \ov t - \ov p$.
  The \emph {zero perversity} is defined by $ \ov {0} (S) =  0$.

  A GM-perversity
 induces a perversity on a filtered space $X$, still denoted by $\ov p$, defined by $\ov p (S) = \ov p (\codim S)$. 
\end{definition}

\bd Let $f \colon X \to Y$ be a stratified map. The \emph{pull-back of a perversity } $\ov q$ of $Y$ is the perversity $f^* \ov q$ of $X$ defined by $f^*\ov q (S) = \ov q(S^f)$.
\ed

\section{Blown-up  complex of a weighted simplicial complex.}\label{sec:TWcomplexepoids}

The blown-up  complex is based on the blow up of a filtered simplex. This technique goes back to
  \cite{MR1143404}, (see also \cite{MR2210257}). In this section, we introduce the notions and first
properties associated to the blow up  of a  regular euclidean simplex (see \defref{def:eclate}) and, more generally, of a weighted simplicial complex (see \defref{def:poidssommet}).

The cone of an euclidean simplex $ \Delta = [e_0, \dots, e_m]$ is the simplicial simplex $ \tc \Delta = [e_0, \dots, e_m, \tv] $. 

The notation
$\nabla \triangleleft \Delta$ means that $\nabla $ is a  face of $\Delta $.

\subsection{Differential complexes associated to a simplicial complex}\label{subsec:petitspoids}
Let $ \cal L $ be an oriented simplicial complex whose family of vertices is $ \cal V (\cal L) = \{e_{0}, \dots, e_{m} \} $.
We denote $ (\hiru N   {*} {\cal L} ,\partial) $ the $ R $-complex of  linear simplices of $ \cal L$.
The differential of $ (\hiru N   {*} {\cal L} , \partial) $ is defined by
$$\partial [e_{i_{0}},\dots,e_{i_{p}}]=
\sum_{k=0}^{p} (-1)^k
[e_{i_{0}},\dots,\hat{e}_{i_{k}},\dots,e_{i_{p}}].$$
For any vertex $e\in\cal V(\cal L)$, we define the homomorphism
$-*e\colon N_{p}(\cal L)\to N_{p+1}(\cal L)$,
by
\begin{equation*}\label{equa:starsimplex}
[e_{i_{0}},\dots, e_{i_{p}}]*e=
\left\{
\begin{array}{cl}
[e_{i_{0}},\dots ,e_{i_{p}},e]&\text{if }[e_{i_{0}},\dots ,e_{i_{p}},e]\subset \cal L,\\
0&\text{if not.}
\end{array}\right.
\end{equation*}
This operator is extended to the empty set by
$ \emptyset \ast e = [e] $.
When  $ [e_ {0}, \dots, e_{r}] $ does not match an ordered simplex, by convention,
it means
$ (- 1) ^ {\varepsilon (\sigma)} [e_{\sigma (0)}, \dots, e_ {\sigma (r)}] $,
where $ \sigma $ is the permutation for which $ [e _ {\sigma (0)}, \dots, e_{\sigma (r)}] $ is an
ordered simplex and $ \varepsilon (\sigma) $ its signature.
We also put $ [e _ {{0}}, \dots, e_{r}] = 0 $ if two vertices $ e_{i} $ are equal.

 Let $ (\Hiru N   {*} {\cal L} , \delta) $ be the $ R $-dual, $ \hom (\hiru N   {*} {\cal L} , R), $ equipped with the transpose differential of $  \partial $,
defined by $ (\delta f) (v) = - (- 1) ^ {| f | } f(\partial v) $.
Among the elements of $ \Hiru N   {*} {\cal L} $, we consider the dual basis of the simplices of
$ \cal L$; i.e., if $ F $
is a $ p $-simplex of $ \cal L$, we denote by
$ \1_ {F} $ the $ p $-cochain taking the value $ 1$ on $ F$ and $0$ on the other simplices of $ \cal L$.
If $ e  \in \cal V (\cal L) $, we also introduce a homomorphism,
 $ - * e \colon \Hiru N   {p} {\cal L}\to \Hiru N   {p+1} {\cal L} $,
defined by
\begin{equation*}
\1_{[e_{i_{0}},\dots,e_{i_{p}}]}\ast e=
(-1)^p\1_{[e_{i_{0}},\dots,e_{i_{p}},e]},
\end{equation*}
putting $ \1_ {[e_ {i_ {0}}, \dots, e_ {i_ {p}} e]} =  0$ if $ [e_{i_ {0}}, \dots, e_ {i_ { p}} ,e] \not\subset \cal L$.
On the elements of the dual basis,
the differential $ \delta \colon \Hiru N * {\cal L} \to \Hiru N {*+ 1} {\cal L}  $
takes the value \\
$\delta\1_{[e_{i_{0}},\dots,e_{i_{p}}]}=\sum_{{k}\notin \{i_{0},\dots,i_{p}\}}  \1_{[e_{i_{0}},\dots,e_{i_{p}},e_{{k}}]}.$
Using $[e_{i_ 0}, \dots, e_{i_k}] = 0 $ when two of the vertices are equal, we can
 consider that the above sum is indexed by the set of vertices of $ \cal L$; i.e.,
\begin{equation}\label{equa:ledelta1}
\delta \1_{[e_{i_{0}},\dots,e_{i_{p}}]}
=\sum_{e\in\cal V(\cal L)} \1_{[e_{i_{0}},\dots,e_{i_{p}},e]}=
(-1)^p\sum_{e\in\cal V(\cal L)} \1_{[e_{i_{0}},\dots,e_{i_{p}}]}\ast e.
\end{equation}
The cochain $ \delta \1_ {F} $ depends on the simplicial complex $ {\cal L} $ in which the face $ F $ lives.
 We denote by $ \delta^{\cal L} \1_{F} $  the differential of $ \1_ {F} $ in ~ $ {\cal L} $,
 when necessary.

\begin{definition}\label{def:poidssommet}
A \emph{weighted simplicial complex} is an oriented simplicial complex,
where each vertex $ e $ has a weight $ w (e) \in \{0, \dots, n \} $.
The integer $ n $ is called the {\em formal dimension of $ \cal L$.}
\end{definition}

Let $ \cal L_{i} $ be the union of the simplices of $\cal L$  whose vertices have a weight equal to $i$.
Any simplex $ F $ of $ \cal L$ is written (modulo orientation) as a join,
$ F = F_ {0} \ast \dots \ast F_ {n} $,
with $ F_i \subset \cal L_i$.
They are the \emph{filtered euclidean simplices}.

Let $ \tc \cal L_ {i} $ be the cone on $ \cal L_{i} $, whose apex $ \tv_ {i} $ is  called 
\emph{virtual vertex},
\label{virt}
  chosen as the last vertex of the cone; i.e.,
if $ F_ {i} = [e_ {i_{0}}, \dots, e_ {i_ {p}}] \subset \cal L_ {i} $, then $ \tc F_ {i} = [e_ {i_{ 0}}, \dots, e_{i_p}, \tv_ {i}] $.
The  face $ [\tv_ {i}] $  can be considered  as the cone on the empty set.
For a comprehensive treatment of all 
these cases, we consider the empty set as a simplex of $\cal L_i $.

The \emph{full blow up} of a weighted simplicial complex
$\cal L$ is the prismatic set
$\tcL^{\fc}=\tc\cal L_{0}\times\dots\times\tc\cal L_{n}$.
For each $ i \in \{0, \dots, n \} $, we  denote by $ 
L_ {i} $
- a simplex $ F_{i} $ of  $ \cal L_{i} $ - or the cone $ \tc F_{i} $ of a simplex  of  $\cal L_i $,
- either the singleton reduced to the virtual vertex $ \tv_{i} $.
The products $ L_{0} \times \dots \times  L_{n}$
are called \emph{faces  of $\tcL^{\fc}$}  and represented by
\begin{equation*}\label{equa:lafaceseclate}
(F,\varepsilon)=(F_{0},\varepsilon_{0})\times \dots \times  
(F_{n},\varepsilon_{n}),
\end{equation*}
satisfying the following conventions:
\begin{itemize}
\item the iterated join, $F=F_{0}\ast\dots\ast F_{n}$, is a simplex of $\cal L$;
\item if $\varepsilon_{i}=0$ and $F_{i}\neq \emptyset$, then
$(F_{i},0)=F_{i}$ is a simplex of $\cal L_{i}$;
\item if $\varepsilon_{i}=1$ and $F_{i}\neq \emptyset$, 
then $(F_{i},1)=\tc F_{i}$ is the cone over the simplex $F_{i}$ of $\cal L_{i}$;
\item  if $F_{i}=\emptyset$,  one must have $\varepsilon_{i}=1$. 
\end{itemize}

\medskip

Let $(F,\varepsilon)=(F_{0},\varepsilon_{0})\times \dots \times (F_{n},\varepsilon_{n}) $ be a face of 
  $\tcL^\fc$ and
let $\gamma\in\{0,\dots,n\}$.  
We put,
\begin{itemize}
\item 
$|(F,\varepsilon)|_{< \gamma}=\sum_{i< \gamma}(\dim F_{i}+\varepsilon_{i})$
and
$|(F,\varepsilon)|_{> \gamma}=\sum_{i> \gamma}(\dim F_{i}+\varepsilon_{i})$,
\item 
$|(F,\varepsilon)|_{\leq \gamma}=\sum_{i\leq \gamma}(\dim F_{i}+\varepsilon_{i})$
and
$|(F,\varepsilon)|_{\geq \gamma}=\sum_{i\geq \gamma}(\dim F_{i}+\varepsilon_{i})$,
\item $|(F,\varepsilon)|=|(F,\varepsilon)|_{\leq n}$.
\end{itemize}

\begin{definition}\label{def:eclate}
The \emph{blow up} of  a weighted simplicial complex of dimension $n$, $ \cal L$, is the sub-prism  $ \tcL$ of $\tcL^\fc$ consisting of  the $ (F, \varepsilon)'s  $
such that $  \varepsilon_ {n} =  0$.
Notice that the corresponding simplices, $ F = F_ {0} \ast \dots \ast F_{n} \subset \cal L$, must verify $ F_ {n} \neq \emptyset$. This kind of filtered simplices are called
\emph{regular}.
\end{definition}

\begin{example}\label{exem:eclate}
The generic case of a weighted simplicial complex is a regular simplex 
$ \Delta = \Delta_{0} \ast \dots \ast \Delta_{n} $.
Assume $ \Delta $ is  oriented by an order on its vertices, $ (e_ {i})_{0 \leq i \leq m}, e_ {i} <e_ {i + 1} $, in a compatible way with the join decomposition. More precisely, expressing the maximal simplex by its vertices, and indicating by a vertical bar the filtration change, we have
\begin{equation}\label{equa:ordresommets}
\Delta=[\underbrace{e_{0},\dots,e_{k_{0}}}_{\Delta_{0}}\mid 
\underbrace{e_{k_{0}+1},\dots,e_{k_{1}}}_{\Delta_{1}}\mid\dots\mid 
\underbrace{\emptyset}_{\Delta_{i}}\mid\dots\mid 
\underbrace{e_{k_{n-1}+1},\dots,e_{k_{n}}}_{\Delta_{n}}],
\end{equation}
with $ k_ {{n}} = {m} $. The vertices
$ \{e_{0} ,\dots, e_{k_{\ell}} \} $ generate the simplex
$ \Delta_{0} \ast \dots \ast \Delta_{\ell} $.
A  face $ \Delta_{i} $  can be empty, as shown in the decomposition above.
As we have already noticed,
giving the  filtration $\Delta = \Delta_ {0} \ast \dots \ast \Delta_n $ of the simplex $ \Delta $ is 
equivalent to giving a weight on each vertex of $ \Delta $, in the sense of \defref{def:poidssommet}.
The corresponding blow up is
\be\label{bup}
\tDelta=\tc \Delta_{0}\times\dots\times \tc\Delta_{n-1}\times \Delta_{n}.
\ee
\end{example}

\begin{example}\label{exem:codageeclate}
The coding of the faces of the blow up  of a regular simplex under the form of a product
$ (F, \varepsilon) = (F_ 0 ,\varepsilon_{0}) \times \dots \times
(F_ {n}  ,\varepsilon_n) $
is used throughout the text. To familiarize the reader with this representation, we 
specify the blow up $\tDelta$ of
$\Delta=\Delta_{0}\ast\Delta_{1}=[e_{0}]\ast [e_{1},e_{2}]$.

\vskip -2cm

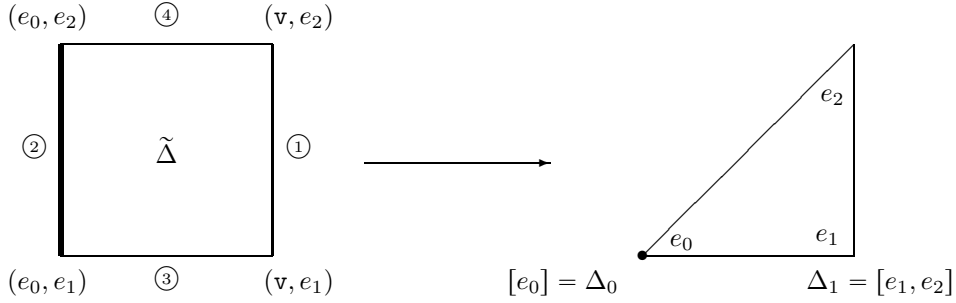
$$
\begin{picture}(250,150)(100,0)

\put(55,-10){\makebox(0,0){$(e_0,e_1)$}}

\put(55,90){\makebox(0,0){$(e_0,e_2)$}}

\put(100,40){\makebox(0,0){$\tDelta$}}

\put(150,-10){\makebox(0,0){$(\tv,e_1)$}}

\put(150,90){\makebox(0,0){$(\tv,e_2)$}}

\put(150,40){\makebox(0,0){\textcircled{\tiny 1}}}
\put(50,40){\makebox(0,0){\textcircled{\tiny 2}}}
\put(100,-10){\makebox(0,0){\textcircled{\tiny 3}}}
\put(100,90){\makebox(0,0){\textcircled{\tiny 4}}}

\linethickness{,7mm}

\put(60,0){\line(0,1){80}} 
\thinlines
\put(60,0){\line(1,0){80}} 
\put(140,80){\line(-1,0){80}} 
\put(140,80){\line(0,-1){80}}


\put(280,0){\makebox(0,0){$\bullet$}}

\put(250,-10){\makebox(0,0){$[e_0] =\Delta_0 $}}

\put(370,-10){\makebox(0,0){$\Delta_1 =[e_1,e_2] $}}

\put(295,5){\makebox(0,0){$e_0$}}
\put(350,7){\makebox(0,0){$e_1$}}
\put(352,60){\makebox(0,0){$e_2$}}

\put(280,0){\line(1,1){80}} 

\put(280,0){\line(1,0){80}}

\put(360,80){\line(0,-1){80}}


\put(210,40){\makebox(0,0){\vector(1,0){70}}}


\end{picture} 
$$

\bigskip\bigskip
The four one dimensional faces of the blow up 
$\tDelta = \tc \Delta_0 \times \Delta_1$ are encoded as\\

\begin{center}
\begin{tabular}{ll}
\textcircled{\tiny 1}=$(\emptyset ,1) \times (\Delta_1,0) ,$& 
\textcircled{\tiny 2}=$([e_{0}] ,0) \times (\Delta_1,0),$\\[.2cm]
 \textcircled{\tiny 3}=$([e_{0}]  ,1) \times ([e_1],0), $& 
 \textcircled{\tiny 4}=$([e_{0}]  ,1) \times ([e_2],0),$
  \end{tabular}
 \end{center}
 \smallskip
 \noindent corresponding to 
 \smallskip
 \begin{center}
\begin{tabular}{lllllll}
 \textcircled{\tiny 1}=$[\tv]\times \Delta_1,$& & 
\textcircled{\tiny 2}=$\Delta_0  \times \Delta_1,$& &
 \textcircled{\tiny 3}=$\tc\Delta_0  \times [e_1] ,$& &
 \textcircled{\tiny 4}=$\tc \Delta_0 \times [e_2].$
 \end{tabular}
 \end{center}

\end{example}

\subsection{Blown-up  complex associated to a weighted simplicial complex}\label{subsec:twpoids}
We recall the notations of the previous paragraph. The element $ \1_{(F_{i}, \varepsilon_ {i})} $ is the cochain on
$ \tc \cal L_ {i} $, taking the value $ 1$ on the simplex $(F_ {i} ,
\varepsilon_i)$ and $ 0$ on other simplices of $ \tc \cal L_i$.
If $F=F_{0}\ast\dots\ast F_{n}\subset\cal L$ and
$(F,\varepsilon)=(F_{0},\varepsilon_{0})\times \dots \times (F_{n},\varepsilon_{n}) $,
we write
$$\1_{(F,\varepsilon)}=\1_{(F_{0},\varepsilon_{0})}\otimes\dots\otimes \1_{(F_{n},\varepsilon_{n})}.$$
The $ R $-module generated by these elements when $ F $ runs over the simplices of $ \cal L$ is denoted by $ \Hiru\tN {\fc, *} {\cal L; R} $
or $ \Hiru\tN {\fc, *} {\cal L} $
 if there is no ambiguity.
The sub-$ R $-module generated by the elements such that $ \varepsilon_ {n} =0 $, endowed with the differential
\begin{eqnarray}\label{equa:ladiff}
\delta \1_{(F,\varepsilon)}
&=&
\sum_{i=0}^{n-1}(-1)^{|(F,\varepsilon)|_{<i}}
\1_{(F_{0},\varepsilon_{0})}\otimes\dots\otimes \delta^{\tc\cal L_{i}}\1_{(F_{i},\varepsilon_{i})}
\otimes\dots\otimes\1_{F_{n}}\\
&&
+ (-1)^{|(F,\varepsilon)|_{<n}}
\1_{(F_{0},\varepsilon_{0})}\otimes\dots\otimes 
\1_{(F_{n-1},\varepsilon_{n-1})}
\otimes\delta^{\cal L_{n}}\1_{F_{n}},\nonumber
\end{eqnarray}
 is called the \emph{blown-up  complex} of 
$ \cal L $ and denoted by $\Hiru \tN  * {\cal L; R} $
or simply $ \Hiru\tN * {\cal L} $.
(Recall that $ \delta^{\tc \cal L_ {i}} $ is the differential of the cochain complex on the simplicial complex
$ \tc \cal L_  {i}$.)

The elements of the blown-up  complex are provided with an additional degree
 which looks like the degree of a differential form along the fiber in a bundle.

\begin{definition}\label{def:degrepervers}
Let $ \ell \in \{1, \dots, n \} $. %
The \emph{$\ell$-perverse degree of the cochain 
$ \1_{(F, \varepsilon) }$}
is equal to
$$
\|\1_{(F,\varepsilon)}\|_{\ell}=\left\{
\begin{array}{ccl}
-\infty&\text{if}
&
\varepsilon_{n-\ell}=1,\\
|(F,\varepsilon)|_{> n-\ell}
&\text{if}&
\varepsilon_{n-\ell}=0.
\end{array}\right.$$
If $ \omega \in  \Hiru \tN  * {\cal L; R} $ decomposes  as 
$ \omega = \sum_ {\mu} \lambda_{\mu} \, \1_{(F_{\mu} \varepsilon_{\mu})} $, with $ \lambda_ {\mu} \neq 0$,
its \emph {$\ell$-perverse degree }
 is equal to
 \begin{equation}\label{equa:degrepervers}
\|\omega\|_{\ell}=\max_{\mu}\|\1_{(F_{\mu},\varepsilon_{\mu})}\|_{\ell}.
\end{equation}
By convention, we set $ \| 0 \|_{\ell} = - \infty $.
\end{definition}

\bex
We compute the perverse degree of some cochains 
of the blow up  $ \tDelta = c\Delta_0 \times c\Delta_1 \times \Delta_2$
of the regular simplex 
$\Delta =\Delta_0 * \Delta_1 * \Delta_2$.

\bigskip

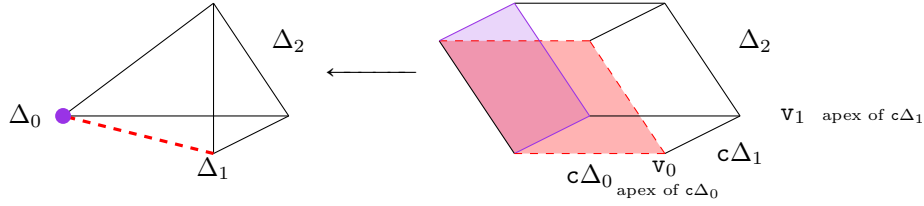
\begin{center}
\begin{tikzpicture}
\definecolor{zzttqq}{rgb}{0.6,0.2,0.9}
\draw[color=black] (2.1,1.1) node {$\xleftarrow{\hspace*{.5cm} \hspace*{.5cm}}$};
\draw [color=black] (0,0)-- (1,0.5);
\draw [color=black] (1,0.5)--  (0,2);
\draw [color=black]  (0,2)-- (0,0);
\draw [color=black]  (1,0.5)-- (-2,0.5);
\draw [color=red,very thick,dashed]  (0,0)-- (-2,0.5);
\draw [color=black]  (0,2)-- (-2,0.5);
\fill [color=zzttqq] (-2,0.5) circle (3pt);
\draw[color=black] (-2.5,0.5) node {$\Delta_0$};
\draw[color=black] (0,-0.2) node {$\Delta_1$};
\draw[color=black] (1,1.5) node {$\Delta_2$};
\draw [color=black] (6,0)-- (7,0.5);
\draw [] (7,0.5)--  (6,2);
\draw [color=black]  (5,1.5)-- (6,2);
\fill[color=black,fill=zzttqq,fill opacity=0.2,dashed] (4,0) -- (5,0.5) -- (4,2) -- (3,1.5) -- (4,0) -- cycle;
\fill[color=black,fill=red,fill opacity=0.3] (4,0) -- (6,0)-- (5,1.5)  -- (3,1.5) -- (4,0) -- cycle;
\draw [color=zzttqq] (4,0)-- (5,0.5);
\draw [color=black] (5,0.5)--  (7,0.5);
\draw [color=black] (6,2)--  (4,2);
\draw [color=black]  (4,0)-- (3,1.5);
\draw [color=red,dashed]  (6,0)-- (4,0);
\draw [color=red,dashed]  (3,1.5)-- (5,1.5);
\draw [color=zzttqq]  (3,1.5)-- (4,2);
\draw [color=red,dashed] (6,0 )--  (5,1.5);
\draw [color=zzttqq]  (5,0.5)-- (4,2);
\draw[color=black] (7.2,1.5) node {$\Delta_2$};
\draw[color=black] (5,-0.3) node {$\tc\Delta_0$};
\draw[color=black] (7,0) node {$\tc\Delta_1$};

\draw[color=black] (8.5,0.5) node {${\mathtt v_1}$ \tiny{ apex of $\tc\Delta_1$} };
\draw[color=black] (6,-.2) node {$\tiny{\mathtt  v_0}$  };
\draw[color=black] (6,-.5) node {\tiny{ apex of $\tc\Delta_0$}  };

\end{tikzpicture}
\end{center}

$$
\begin{array}{ll}
\|\1_{ {\mathtt v_0} \times {\mathtt v_1} \times \Delta_2}\|_2 =-\infty 
&
\|\1_{ {\mathtt v_0} \times {\mathtt v_1} \times \Delta_2}\|_1 =-\infty
\\
\|\1_{c \Delta_0 \times\Delta_1 \times \Delta_2}\|_2=-\infty  
&
\|\1_{c \Delta_0 \times\Delta_1 \times \Delta_2}\|_1= \dim \Delta_2 \\
\|\1_{ \Delta_0 \times \tc\Delta_1 \times \Delta_2}\|_2 =\dim \tc\Delta_1 + \dim \Delta_2 &\|\1_{ \Delta_0 \times \tc\Delta_1 \times \Delta_2}\|_1 =-\infty\\
\|\1_{ \Delta_0 \times \Delta_1 \times \Delta_2}\|_2 =\dim \Delta_1 +\dim \Delta_2
&
\|\1_{ \Delta_0 \times \Delta_1 \times \Delta_2}\|_1 =\dim \Delta_2
\end{array}
$$
(We have used $\tv_i$ the apex of the cone  $\tc \Delta_i$, $i=0,1$).

\eex

\bd
In the case of a \emph{ regular simplex} $  \Delta = \Delta_{0} \ast \dots \ast \Delta_ {n} $, $\Delta_n\ne \emptyset$,
the blown-up   complex is the tensor product,
$\Hiru \tN  * {\Delta} = \Hiru N  * {\tc \Delta_{0}} \otimes \dots \otimes  \Hiru N  * {\tc \Delta_ { n-1}} \otimes \Hiru N  * {\Delta_ {n}} $ which corresponds to \cite[Definition 1.31]{CST1}.

The \emph{face  operators} of  a filtered simplex $ \Delta = \Delta_ {0} \ast \dots \ast \Delta_n$ are maps
$ \mu \colon \Delta' = \Delta'_{0} \ast \dots \ast \Delta'_{n} \to
\Delta = \Delta_{0} \ast \dots \ast \Delta_ {n} $,
of the form $ \mu = \ast_ {i = 0}^n \mu_{i} $ where $ \mu_ {i} $ is an injective map preserving the order.
The face $\Delta'$ is a codimension 1 regular face.
Each operator face of a factor of the join
 $\delta_{\ell}\colon \Delta'_{{i}}\to \Delta_{{i}}$
 induces a map,
still denoted by 
 ${\delta}_{\ell}\colon \hiru N* {\tc\Delta'_{0}}\otimes \dots\otimes  \hiru  N *{\tc\Delta'_{n-1}} \otimes  \hiru  N *{\Delta'_{n}}
 \to 
  \hiru  N *{\tc\Delta_{0}}
  \otimes \dots\otimes 
   \hiru N *{\tc\Delta_{i}} \otimes \dots\otimes \hiru  N *{\Delta_{n}}
$,
 obtained by carrying out the join with the identity map on factors $ \Delta_{j} $, $ j \neq \ell $.
We call
 ${\delta}_{\ell}^*\colon \Hiru N *{\tc\Delta_{0}}\otimes \dots\otimes \Hiru N*{\tc\Delta_{n-1}}\otimes \Hiru N *{\Delta_{n}}
 \to \Hiru N *{\tc\Delta'_{0}}\otimes \dots\otimes \Hiru N *{\tc\Delta'_{i}}\otimes \dots\otimes \Hiru N *{\Delta'_{n}}
 $,
 the transpose of the linear map $\delta_\ell$.
 \ed
 
The blown-up  complex of a weighted simplicial complex is expressed from the 
  blown-up  complex of regular simplices.
  
\begin{proposition}\label{prop:pasdeface}
Let $\cal L$ be a weighted simplicial complex. Then
$$
\Hiru \tN *{\cal L;R}\cong
\varprojlim_{\Delta\subset  \cal L, \Delta\,\text{regular}} 
\Hiru \tN *{\Delta;R}.$$
\end{proposition}

 \begin{proof}
 An element $\omega$ of $ \varprojlim
\Hiru \tN  * {\Delta} $ is a family
 $ \omega = \left( \omega_{\Delta} \right)_{\Delta} $,
 indexed by the regular simplices of $ \cal L$, with $ \omega_{\Delta} \in \Hiru \tN  * {\Delta} $ and satisfying the following face  compatibility conditions:
 if $ \delta_{\ell} \colon \nabla \to \Delta $ is a regular face of codimension 1 of a simplex $ \Delta $ of $ \cal L$, then
 $ \delta_{\ell}^* {\omega _ {\Delta}} = \omega_{\nabla} \in 
 \Hiru \tN  * {\nabla} $.
 To $ \omega \in \varprojlim \Hiru \tN  * {\Delta} $, we associate the cochain
$ \omega_{\cal L} \in  \Hiru \tN  k {\cal L} $ defined by
$$\omega_{\cal L}(F,\varepsilon) =
\omega_{F}(F,\varepsilon).$$

Conversely, let
$\omega_{\cal L}=
\sum_{F}
\alpha_{(F,\varepsilon)}\1_{(F,\varepsilon)}\in\tN^k(\cal L)$,
where $ F $ runs over the regular simplices of $ \cal L$,
$ | (F, \varepsilon) | =k  $
and $ \alpha_{(F, \varepsilon)} \in R $.
For each regular simplex $ \Delta $ of $ \cal L$, we define $ \omega_{\Delta} \in  \Hiru\tN  k {\Delta} $ by
$$\omega_{\Delta}=
\sum_{F\fa \Delta}
\alpha_{(F,\varepsilon)}\1_{(F,\varepsilon)}.$$
It remains to prove the 
compatibility  with respect  to the face operators.
Let $ \delta_{\ell} \colon \nabla \to \Delta $ be  a  regular face of $\Delta$ and  $ F $ a  simplex of $\Delta $. By definition of the dual basis one has
$${\delta}^*_{\ell}(\1_{(F,\varepsilon)})=
\left\{\begin{array}{cl}
\1_{(F,\varepsilon)}
&\text{if}\;
F\subset\nabla,\\
0&\text{if not.}
\end{array}\right.$$
It follows
$$\delta_{\ell}^*(\omega_{\Delta})=
\sum_{F\fa \nabla}\alpha_{(F,\varepsilon)}\1_{(F,\varepsilon)}
=\omega_{\nabla}.$$
\end{proof}


\subsection{Adjunction of a vertex to a cochain of the blow up}
\bd\label{def:etunpointun}
Let
$(F,\varepsilon)=(F_{0},\varepsilon_{0})\times \dots \times (F_{n},\varepsilon_{n})$ be a face of ${\tcL^\fc}$ and  $\ell\in\{0,\dots,n\}$.
 \emph{The adjunction of a vertex 
 $e\in \cal L_{\ell}$}
to the cochain $\1_{(F,\varepsilon)}$
is the cochain
 \begin{equation*}\label{equa:etunpointunaussi}
\1_{(F,\varepsilon)}\ast e=
(-1)^{|(F,\varepsilon)|_{> \ell}}\;
\1_{(F_{0,}\varepsilon_{0})}\otimes\dots\otimes
(\1_{(F_{\ell},\varepsilon_{\ell})}\ast e)
\otimes\dots\otimes
\1_{(F_{n},\varepsilon_{n})}.
\end{equation*}
Likewise, for the virtual vertex $\tv_{\ell}$,  we set,
\begin{equation*}\label{equa:etunpointvirtuel}
\1_{(F,\varepsilon)}\ast \tv_{\ell} =
(-1)^{|(F,\varepsilon)|_{> \ell}}\;
\1_{(F_{0,}\varepsilon_{0})}\otimes\dots\otimes
(\1_{(F_{\ell},\varepsilon_{\ell})}\ast \tv_{\ell})\otimes\dots\otimes
\1_{(F_{n},\varepsilon_{n})}.
\end{equation*}
We extend by linearity this adjunction to
$\Hiru  \tN {\fc,*}{\cal L;R}$.
\ed

	The next property follows directly from the definition.
	
\bl\label{lem:unecellulepuislautre}
Let $\cal L$  be a weighted simplicial complex  and 
$(F,\varepsilon)$ 
 a face of ${\tcL^\fc}$. Consider two vertices of $\cal L$,
 $e_{\alpha}\in \cal L_{\ell(\alpha)}$
 and
$e_{\beta}\in \cal L_{\ell(\beta)}$, 
and two virtual vertices
$\tv_{\ell}$ and $\tv_{\ell'}$. Then the following properties are verified,
\begin{eqnarray*}
\1_{(F,\varepsilon)}\ast e_{\alpha}\ast e_{\beta}
&=&
- \1_{(F,\varepsilon)}\ast e_{\beta}\ast e_{\alpha}, \\
\1_{(F,\varepsilon)}\ast e_{\alpha}\ast \tv_{\ell}
&=&
- \1_{(F,\varepsilon)}\ast \tv_{\ell}\ast e_{\alpha},\\
\1_{(F,\varepsilon)}\ast \tv_{\ell}\ast \tv_{\ell'}
&=&
- \1_{(F,\varepsilon)}\ast \tv_{\ell'}\ast \tv_{\ell}.
\end{eqnarray*}
\el

\begin{proposition}\label{prop:dfaceajout}
Let $\cal L$ be a weighted simplicial complex.
The blown-up  differential of an element 
$\1_{(F,\varepsilon)}\in\Hiru \tN *{\cal L;R}$
can be written as,
\begin{equation}\label{equa:diffetpoint}
\delta\1_{(F,\varepsilon)}= 
(-1)^{|(F,\varepsilon)|}
\left(
\sum_{e\in \cal V(\cal L)} \1_{(F,\varepsilon)}\ast e
+
\sum_{i=0}^{n-1} \1_{(F,\varepsilon)}\ast \tv_{i}
\right).
\end{equation}
\end{proposition}

\begin{proof}
Let
$\1_{(F,\varepsilon)}=
\1_{(F_{0},\varepsilon_{0})}\otimes\dots\otimes \1_{(F_{n-1},\varepsilon_{n-1})}
\otimes\1_{F_{n}}$.
By replacing $\delta^{\tc\cal L_{i}}$ and $\delta^{\cal L_{n}}$ by their values from (\ref{equa:ledelta1}) 
in the equality (\ref{equa:ladiff}) and denoting $|(F_i,\varepsilon_i)| = \dim F_i  + \varepsilon_i$, we get,
\begin{eqnarray*}
\delta \1_{(F,\varepsilon)}
&=&\\
\sum_{i=0}^{n}(-1)^{|(F,\varepsilon)|_{<i}}
\sum_{e_{i}\in\cal V(\cal L_{i})}
(-1)^{|(F_{i},\varepsilon_{i})|}
\1_{(F_{0},\varepsilon_{0})}\otimes\dots\otimes (\1_{(F_{i},\varepsilon_{i})}\ast e_{i})
\otimes\dots\otimes\1_{(F_{n},\varepsilon_{n})}&&\\
+\sum_{i=0}^{n-1}(-1)^{|(F,\varepsilon)|_{\leq i}}
\1_{(F_{0},\varepsilon_{0})}\otimes\dots\otimes 
(\1_{(F_{i},\varepsilon_{i})}\ast \tv_{i})
\otimes\dots\otimes 
\1_{F_{n}}.
\end{eqnarray*}
The wanted formula  follows from the definition of the adjunction of a vertex.
\end{proof}

From \propref{prop:dfaceajout} and \lemref{lem:unecellulepuislautre}, we directly deduce the
behavior of the adjunction of a vertex with respect to the differential of the blown-up  complex.
\begin{corollary}\label{cor:astdiff}
Let  $\cal L$ be a weighted simplicial complex
and let
$\1_{(F,\varepsilon)}\in \Hiru \tN *{\cal L;R}$.
For each vertex $e_{\alpha}\in \cal V(\cal L)$  and each virtual vertex
 $\tv_{\ell}$, one has
$ \delta(\1_{(F,\varepsilon)}\ast e_{\alpha})=
(\delta\1_{(F,\varepsilon)})\ast e_{\alpha}
\;\text{ and }\;
\delta(\1_{(F,\varepsilon)}\ast \tv_{\ell})=
(\delta\1_{(F,\varepsilon)})\ast \tv_{\ell}.
$
\end{corollary}

\section{Blown-up  intersection cohomology.}\label{sec:TWcohomologie}

In this section, we define the blown-up intersection cohomology of a perverse space.

\begin{definition}\label{def:filteredsimplex}
Let  $X$ be a filtered space.   A
 \emph{filtered singular simplex} is a continuous map, $\sigma\colon\Delta\to X$, where the euclidean simplex $\Delta$ is endowed with a decomposition
$\Delta=\Delta_{0}\ast\Delta_{1}\ast\dots\ast\Delta_{n}$,
called \emph{$\sigma$-decomposition of $\Delta$},
verifying
$$
\sigma^{-1}X_{i} =\Delta_{0}\ast\Delta_{1}\ast\dots\ast\Delta_{i},
$$
for each~$i \in \{0, \dots, n\}$. 
The filtered singular simplex $\sigma$ is \emph{regular} if $\Delta_n\ne \emptyset$.
To specify that the filtration of the euclidean simplex $ \Delta $ is induced from that of $ X $ by the map
$ \sigma $, we sometimes write $ \Delta = \Delta_{\sigma} $. This notation is particularly useful when a simplex carries two filtrations associated with two distinct maps.
\end{definition}

The dimensions of the simplices $ \Delta_i$  of the $ \sigma $-decomposition measure the 
  non-transver\-sality of the  simplex $ \sigma $ with respect to the strata of $X$.
These simplices $ \Delta_i$ may be empty, with the convention $ \emptyset * Y = Y $, for any space $ Y $.
  Note also that a singular simplex $ \sigma \colon \Delta \to X $ is filtered if each $ \sigma^{- 1} (X_i) $,
  $ i \in \{0, \dots, n \} $, is a face of $ \Delta $.
  
\medskip

To any regular simplex, $\sigma\colon \Delta=\Delta_0\ast\dots\ast \Delta_n\to X$, 
we associate the cochain complex defined by 
$$\tres \tN*\sigma=\tN^*(\Delta)=
\Hiru N *{\tc\Delta_0}\otimes\dots\otimes
\Hiru  N *{\tc\Delta_{n-1}} \otimes \Hiru N *{\Delta_n}.$$

If $\delta_{\ell}\colon \Delta'=\Delta'_{0}\ast\dots\ast\Delta'_\ell\ast\dots\ast\Delta'_{n} \to
 \Delta=\Delta_{0}\ast\dots\ast\Delta_\ell\ast\dots\ast\Delta_{n}$
is a face operator, 
we denote $\partial_{\ell}\sigma$ 
the filtered simplex defined by
$\partial_{\ell}\sigma=\sigma\circ\delta_{\ell}\colon 
\Delta'\to X$. The face operator $\delta_{\ell}$ is regular if $\Delta'_n\ne \emptyset$.
 
 \begin{definition}\label{def:thomwhitney}
Let $X$ be a filtered space. 
The \emph{blown-up  complex of $X$ with coefficients in $R$,} $\Hiru \tN*{X;R}$,
is the cochain complex 
 formed by the elements $ \omega $, associating to any regular filtered simplex,
 $\sigma\colon \Delta_{0}\ast\dots\ast\Delta_{n}\to X$,
an element
 $\omega_{\sigma}\in \tN^*_{\sigma}$,  
so that $\delta_{\ell}^*(\omega_{\sigma})=\omega_{\partial_{\ell}\sigma}$,
for any  regular face operator,
 $\delta_{\ell}$.
 The differential $\delta\omega$ of $\omega\in\Hiru \tN *{X;R}$ is defined by
 $(\delta\omega)_{\sigma}=\delta(\omega_{\sigma})$
 for any regular filtered simplex $\sigma$.
 \end{definition}
 
For any $\ell \in \{1,\dots,n\}$,  the element  $\omega_{\sigma}\in \tN^*_{\sigma}$
is endowed with  the \emph{perverse degree} $\|\omega_{\sigma}\|_{\ell}$, 
introduced in \defref{def:degrepervers}. 
We extend this degree to the elements of
$ \Hiru \tN *{X;R}$
as follows.

\begin{definition}\label{def:transversedegreeblowup}
Let $X$ be a filtered space and  $\omega\in\Hiru\tN *{X;R}$.
The \emph{perverse degree of $\omega$ along a singular stratum,} $S$, is equal to
\begin{equation*}\label{equa:perversstrate}
\|\omega\|_{S}=\sup\left\{\|\omega_{\sigma}\|_{\codim S}\mid 
\sigma\colon\Delta\to X \text{ regular with }\sigma(\Delta)\cap S\neq \emptyset\right\}.
\end{equation*}
We denote by $\|\omega\|$  the map associating to any singular stratum $ S $ of $ X  $ the element $\|\omega\|_{S}$ and 0 to a regular stratum.
\end{definition}

Notice that, by definition,  face operators $\delta_{\ell}^*$  
decrease the perverse degree.

\begin{definition}\label{def:admissible}
 Let $(X,\ov{p})$ be a perverse space. A cochain $\omega\in \Hiru \tN *{X;R}$ is \emph{$\ov{p}$-allowable} if
\begin{equation}\label{equa:legraal}
\|\omega\|\leq \ov{p}.
\end{equation}

A cochain $\omega$ is a \emph{$\ov{p}$-intersection cochain} if $\omega$ and its coboundary, $\delta \omega$,
are $\ov{p}$-allowable. 
We denote by $\lau\tN*{\ov{p}}{X;R}$
the complex of  $\ov{p}$-intersection cochains and
 $\lau \IH {\ov{p}} *{X;R}$  its homology, called
 \emph{blown-up  intersection cohomology} of $X$ with coefficients in~$R$,
for the perversity $\ov{p}$. %
\end{definition}

\subsection{Shifted filtrations}\label{SF1}
Blown-up  intersection cohomology does not depend on the dimension of the strata of  $X$ but on the codimension of these strata (see \cite[Section 4.3.1]{LibroGreg}).  Let us consider on $X$ the shifted filtration $Y$, where $m \in \N^*$:
\begin{align}\label{shif}
\emptyset = Y_0 = \cdots = Y_{m-1} \subset Y_m =X_0 \subset \dots \subset Y_{n+m-1} =\\ \nonumber
 X_{n-1} \subset Y_{n+m} =X_n =X.
\end{align}
So, the formal dimension of $Y$ is $n+m$. For example, on $\R^m \times X$ we have two natural shifted filtrations: $(\R^m \times X)_i = \R^m \times X_i$, with $i\in \{0,\ldots,n\}$, and $(\R^m \times X)_i = \R^m \times X_{i-m}$, with $i\in \{0,\ldots,n+m\}$.

The family of  regular simplices is the same for both filtrations. The perverse degree of a filtered simplex $\sigma\colon \Delta \to X=Y$ is the same for both filtrations. Let us see that. If  $S$ is a singular stratum of $X$ (and therefore of $Y$) with $\im \sigma \cap S \ne \emptyset$, we have $\codim_XS =\codim_Y S $. Let $\ell$ be this codimension.
On the other hand, if $\Delta_\sigma^X = \Delta_0* \cdots * \Delta_n$ is the induced filtration by $\sigma \colon \Delta \to X$ then $\Delta_\sigma^Y =  \underbrace{\emptyset *\dots *\emptyset}_{m \ times}*\Delta'_{m}* \cdots * \Delta'_{n+m}$, with $\Delta'_{m+r} = \Delta_r$, is the induced filtration by $\sigma \colon \Delta \to Y$. So, if $(F=F_0*\cdots *F_n,\varepsilon=(\varepsilon_0 *\dots*\varepsilon_n))$ is a face of the blow up of $\Delta_\sigma^X$ then 
$(F'=\underbrace{\emptyset, \dots,\emptyset}_{m \ times}*F_0*\cdots *F_{n},\varepsilon'=(\underbrace{1, \dots,1}_{m \ times},\varepsilon_0,\dots,\varepsilon_{n}))$ is the corresponding face of 
the blow up of $\Delta_\sigma^Y$. We have, for $\ell \in \{1,\dots,n\}$
\begin{eqnarray*}
\|\1_{(F,\varepsilon)}\|_\ell^X &= &
\left\{
\begin{array}{ll}
-\infty & \hbox{if } \varepsilon_{n-\ell} =1\\
|(F,\varepsilon)|_{>n-\ell}  & \hbox{if } \varepsilon_{n-\ell} =0
\end{array}
\right.
=
\left\{
\begin{array}{ll}
-\infty & \hbox{if } \varepsilon'_{m+n-\ell} =1\\
|(F',\varepsilon')|_{>m+n-\ell}  & \hbox{if } \varepsilon'_{m+n-\ell} =0
\end{array}
\right.\\
&=&
\|\1_{(F',\varepsilon')}\|_{\ell}^Y,
\end{eqnarray*}
and
$
\|\1_{(F',\varepsilon')}\|_\ell^Y = -\infty \ \ \hbox{ if } \ell \in \{n+1, \ldots, m+n\}.
$
We get $\lau \tN *{\ov p} {X;R} = \lau \tN * {\ov p} {Y;R}$ and therefore $\lau \IH* {\ov p} {X;R} = \lau \IH * {\ov p} {Y;R}$.

\bigskip

We  prove in \thmref{MorCoho} that any stratified map induces a morphism in blown-up intersection cohomology. The following Proposition, which is a weaker version of this result, is used in \partref{part:mayervietoris}.

\begin{proposition}\label{prop:applistratifieeforte}
Let $f \colon X \to Y$ be a stratum preserving stratified map. The \emph{induced map}  $f \colon \Hiru \tN * {Y;R} \to \Hiru \tN * {X;R} $, defined by
$(f^*(\omega))_{\sigma}=\omega_{f\circ\sigma}$,
 is a well defined chain map.
 
 Consider a perversity $\ov p$ on $X$ and a perversity $\ov q$ on $Y$ verifying $\ov p \geq f^*\ov q$.
 The induced operator $f ^*\colon \lau \tN * {\ov q} {Y;R} \to \lau \tN * {\ov p} {X;R}$,
 is a well defined chain map inducing the morphism 
$f ^*\colon \lau \IH * {\ov q} {Y;R} \to \lau \IH* {\ov p} {X;R}$.
\end{proposition}

\begin{proof}

The definition makes sense since 
$(\delta_{\ell}^*f^*(\omega))_{\sigma}
=
\delta_{\ell}^*(f^*(\omega)_{\sigma})
=
\delta_{\ell}^*(\omega_{f\circ\sigma})$
$
=
\omega_{f\circ\sigma\circ\delta_{\ell}}
=
(f^*(\omega))_{\sigma\circ\delta_{\ell}}$,
for each $\omega \in \Hiru \tN * {Y;R}$, each regular simplex $\sigma \colon \Delta \to X$
and each regular face $\delta_\ell \colon \nabla \to \Delta$.

The induced morphism is a chain map since
 $(\delta f^*(\omega))_{\sigma}
=
\delta (f^*(\omega)_{\sigma})
=
\delta(\omega_{f\circ\sigma})
=
(\delta\omega)_{f\circ\sigma}
=
(f^*(\delta\omega))_{\sigma}$,
for each $\omega \in \Hiru \tN * {Y;R}$ and each regular simplex $\sigma \colon \Delta_\sigma \to X$.

We turn now to perversities and  firstly to the filtrations.
The euclidean simplex $ \Delta $ has two filtrations, respectively induced by $ \sigma $ and $ f \circ \sigma $,
denoted by $ \Delta_{\sigma} $ and $ \Delta_{f \circ \sigma} $, following the conventions of \defref{def:filteredsimplex}.
For each $\ell \in\{0,\dots,n\}$,
the hypothesis $f^{-1}(Y_{n-\ell})=X_{n-\ell}$ implies equalities,
$\sigma^{-1}(X_{n-\ell })=\sigma^{-1}f^{-1}(Y_{n-\ell })=(f\circ\sigma)^{-1}(Y_{n-\ell })$
and $\Delta_{\sigma}=\Delta_{f\circ\sigma}$.

Let $ S $ be an $\ell$-codimensional  singular stratum of $ X $  such that $ S \cap \im \sigma \neq \emptyset$. The  stratum $S^f$ of $ Y $, characterized by $ f (S) \subset S^f $, also has codimension $ \ell $ and verifies $ S^f \cap \im (f \circ \sigma) \neq \emptyset$. 
Moreover, the definition of $ f^* $ and the previous observation on the two filtrations of $ \Delta $ give
$
\| f^* (\omega)_{\sigma} \|_{\ell} = \| \omega_ {f \circ \sigma} \|_{\ell} $.
From \defref{def:transversedegreeblowup}, we deduce
$\|f^*(\omega)\|_{S}\leq \|\omega\|_{S^f}\leq \ov{q}(S^f)$. The assumption made on the perversities $ \ov {p} $ and $ \ov q $ allows us  to conclude
$\|f^*(\omega)\|_{S}\leq \ov{p}(S)$. The same argument applied to the  $ \ov q$-allowable form,  $ \delta \omega $, gives
$\|f^*(\delta\omega)\|_{S}\leq \ov{p}(S)$. The compatibility of $ f ^ * $ with the differential $ \delta $ implies
$f^*(\omega)\in \lau \tN*{\ov{p}}{X;R}$.
\end{proof}

\begin{remark}\label{rem:fortementstratifieetidentite}
If $f$ is a stratum preserving stratified map, 
for any filtered simplex 
$\sigma\colon\Delta\to X$,
the
$\sigma$- and $f\circ\sigma$- decompositions of $\Delta$ are the same:
$\Delta_{\sigma}=\Delta_{f\circ \sigma}$. This is not the case for a general stratified map (see \cite[A.25]{CST1}).
\end{remark}

\section{Cup product.}\label{subsec:cupTW}

We have defined the notion of cup product in  \cite{CST1}
for "filtered face sets".
In
\cite{CST2}, we also introduced the notion of  $ cup_i$-products and Steenrod squares on the blown-up intersection cohomology.
We give below a definition of cup product for any coefficient ring, in our context of stratum depending
perversities.

\begin{definition}\label{def:cupsurDelta}
Two ordered simplices,
 $F=[a_{0},\dots,a_{k}]$ and $G=[b_{0},\dots,b_{\ell}]$,
of an euclidean simplex $\Delta$, are said to be \emph{compatible}
 if $a_{k}=b_{0}$.
In this case, we set
 $F\cup G=[a_{0},\dots,a_{k},b_{1},\dots,b_{\ell}]\in N_{*}(\Delta)$.
 The \emph{cup product} is defined on the dual basis of  $\Hiru N*{\Delta}$ by
 $$\1_{F}\cup \1_{G}=(-1)^{k \cdot \ell} \ \1_{F\cup G}$$ %
 if $F$ and $G$ are compatible and  0 if not.
 \end{definition}
  In the case of the cone $\tc\Delta$, this law appears on
  $\Hiru N* {\tc\Delta}$ as:
  \begin{equation}\label{equa:cupbase}
  \1_{(F,\varepsilon)}\cup \1_{(G,\kappa)}
  =\left\{\begin{array}{cl}
  (-1)^{|F| \cdot |(G,\kappa)|} \  \1_{(F\cup G,\kappa)}
  &
  \text{if}\; F, G\;\text{compatible}\;\text{and}\;\varepsilon=0,\\
  \1_{(F,\varepsilon)}
  &
  \text{if}\;
 (G,\kappa) = (\emptyset,1) \hbox{ and }  \varepsilon=1,\\
  0
  &
  \text{if not.}
  \end{array}\right.
  \end{equation}

 If $\Delta=\Delta_{0}\ast\dots\ast\Delta_{n}$ 
 is a regular euclidean simplex,
 this definition extend to
  $\Hiru \tN *{\Delta}$ as a law defined on a tensor product.
More precisely,
 if
 $\omega_{0}\otimes\dots\otimes\omega_{n}$ and
 $\eta_{0}\otimes\dots\otimes\eta_{n}$
belong to
 $\Hiru N* {\tc\Delta_{0}} \otimes\dots\otimes \Hiru N*{\Delta_{n}}$,
we set
 \begin{equation}\label{equa:cupsimplexfiltre}
 (\omega_{0}\otimes\dots\otimes\omega_{n})\cup
 (\eta_{0}\otimes\dots\otimes\eta_{n})=
 (-1)^{\sum_{i>j}|\omega_{i}|\,|\eta_{j}|}
 (\omega_{0}\cup\eta_{0})\otimes\dots\otimes
 (\omega_{n}\cup\eta_{n}).
 \end{equation}
 
 Given two perversities $\ov p,\ov q$ defined on a filtered set $X$ we define the perversity $\ov p + \ov q$ by $(\ov p + \ov q )(S) = \ov p(S) + \ov q(S)$, with the convention $-\infty + \ov p(S) =\ov p(S)  -\infty = -\infty$.
 
 \begin{proposition}\label{42}
For each filtered space $X$ endowed with two perversities $\ov{p}$ and $\ov{q}$,
 there exists an associative multiplication, 
 \begin{equation*}\label{equa:cupprduitespacefiltre}
 -\cup -\colon \lau \tN k {\ov{p}}{X;R}\otimes \lau \tN {\ell}{\ov{q}}{X;R}\to \lau \tN {k+\ell} {\ov{p}+\ov{q}}{X;R},
 \end{equation*}
 defined by
 $(\omega\cup\eta)_{\sigma}=\omega_{\sigma}\cup \eta_{\sigma}$,
 for each pair of cochains $(\omega,\eta)\in 
 \lau \tN k {\ov{p}}{X;R}\times \lau \tN \ell{\ov{q}}{X;R}$
and each regular simplex
 $\sigma\colon\Delta\to X$.
 It induces a graded commutative multiplication with unity called \emph{intersection cup product},
  \begin{equation*}\label{equa:cupprduitTWcohomologie}
-\cup -\colon  \lau \IH k {\ov{p}}{X;R}\otimes \lau \IH {\ell} {\ov{q}}{X;R}\to
 \lau \IH{k+\ell}{\ov{p}+\ov{q}}{X;R}.
 \end{equation*}
 \end{proposition}
 
 \begin{proof}
 Begin by verifying that the product locally defined on each regular simplex extends
  to the blown-up  complex. For this, consider two cochains $ \omega ,\eta \in \Hiru \tN*{X;R}$ and
  $ \delta_{\ell} \colon \nabla \to \Delta $
 a  regular face of a simplex $ \sigma \colon \Delta \to X $.
   The map induced by
$ \delta_{\ell}$ at the cochain level,
$\delta_{\ell}^*\colon \Hiru \tN*{X;R}\to \Hiru \tN*{X;R}$,
is the identity on each factor of the tensor product, except for one 
  where it is induced by a canonical inclusion.
It is compatible with the cup product and we can write,
$
\delta_{\ell}^*(\omega\cup \eta)_{\sigma}
=
\delta_{\ell}^*(\omega_{\sigma}\cup\eta_{\sigma})
=_{(1)}
\delta_{\ell}^*(\omega_{\sigma})\cup \delta_{\ell}^*(\eta_{\sigma})
=
\omega_{\sigma\circ\delta_{\ell}}\cup \eta_{\sigma\circ\delta_{\ell}}
=
(\omega\cup\eta)_{\sigma\circ\delta_{\ell}},
$
where $=_{(1)}$ comes from the naturality of the usual cup product.
Then the product $\omega\cup\eta$ of the statement is well defined.

Let us now study the behavior of the cup product with respect to 
the perverse degree. This is a local issue. Let $\omega \in \lau \tN * {\ov p}{X;R}$ and $\eta \in \lau \tN * {\ov q} {X;R}$. 
We need to prove that 
\begin{equation*}\label{equa:cupperversite}
\|\omega_\sigma \cup \eta_\sigma \|_\ell\leq
(\ov p + \ov q )(S),
\end{equation*}
where
$S$ is a singular stratum of $X$, $\ell = \codim S$, 
and
$\sigma \colon \Delta=\Delta_{0}\ast\dots\ast\Delta_{n} \to X$  is a regular simplex with $\im \sigma \cap S \ne \emptyset$.

 Without loss of generality, we can suppose  $\omega_\sigma=\1_{(F,\varepsilon)}$
 and $\eta_\sigma=\1_{(G,\kappa)}$. The result is clear if $\|\omega_\sigma\cup \eta_\sigma\|_{\ell} =-\infty$. So, we can suppose that $\|\omega_\sigma\cup \eta_\sigma\|_{\ell}\ne -\infty$. 
 If $\ov p (S) =-\infty$, condition $\|\omega_\sigma \|_\ell \leq -\infty$ implies $\epsi_{n-\ell}=1$ and therefore $\|\omega_\sigma\cup \eta_\sigma\|_\ell = -\infty$. So, we can suppose that $\ov p (S) >-\infty$ and similarly $\ov q (S) >-\infty$, which imply $(\ov p+ \ov q)(S) = \ov p (S) + \ov q(S)$.
 Using (\ref{equa:cupbase}) and \defref{def:degrepervers}, 
 we see that condition
 $\|\omega_\sigma\cup \eta_\sigma\|_{\ell}\neq-\infty$ gives
 $\varepsilon_{n-\ell}=\kappa_{n-\ell}=0$. The usual degree of the cup product of two cochains gives 
  $\|\omega_\sigma\cup \eta_\sigma\|_{\ell}\leq 
  \|\omega_\sigma\|_\ell + \|\eta_\sigma\|_{\ell}
  \leq \ov p(S) + \ov q (S) = (\ov p+ \ov q)(S).$
 
Then, the cup product of a $\ov{p}$-allowable  cochain  and  of a $\ov{q}$-allowable cochain  is $(\ov{p}+\ov{q})$-allowable. For the intersection  cochains, the result follows from the formula
 $\delta(\omega \cup \eta) = \delta(\omega)\cup \eta+(-1)^{|\omega|} \omega\cup\delta(\eta)$.

Associativity  is deduced from the associativity of the usual cup product. The unity element is the 0-cochain taking the value 1 on any face.

 Let us verify commutativity. 
 We have constructed the product \\
$-\cup_1 -\colon \lau  \tN k {\ov{p}}{X,\Z_2}\otimes  \lau \tN {\ell} {\ov{q}}{X,\Z_2}\to
\lau  \tN {k+\ell-1} {\ov{p}+\ov{q}}{X,\Z_2}
 $
 verifying the Leibniz condition
 $\delta (x_1 \cup_1 x_2)=x_1 \cup x_2 +x_2 \cup x_1 +\delta x_1 \cup_1 x_2 +x_1 \cup_1 \delta x_2$ (see \cite{CST2}).
 For general coefficients, and following the same procedure, we construct the product 
 $-\cup_1 -\colon  \lau \tN k{\ov{p}}{X;R} \otimes \lau \tN {\ell} {\ov{q}} {X;R}\to
 \lau \tN {k+\ell-1} {\ov{p}+\ov{q}}{X;R},
 $
taking into account the signs (see for example \cite[p. 414]{MR1700700}). So, for each 
 $\omega \in  \lau \tN k {\ov{p}}{X;R}$ and  $\eta \in \lau \tN \ell {\ov{q}}{X;R}$ we have
 $$
 \delta (\omega \cup_1\eta )
 =
 (-1)^{p+q-1}\omega \cup \eta + (-1)^{p+q+p q}\eta \cup \omega + \delta \omega \cup_1 \eta +(-1)^p \omega \cup_1 \delta \eta.
  $$
  So, if $\delta \omega =\delta \eta =0$, we get the commutativity
  $
  [\omega] \cup [\eta] = (-1)^{p q} [\eta] \cup [\omega].
  $
\end{proof}

  See \cite{MR2607414, MR3046315} for other cup products.

\begin{remark}\label{43}
Given perversities $\ov a \leq \ov p$ and $\ov b \leq \ov q$, the cup product 
$
- \cup - \colon \lau \tN * {\ov{a}}{X;R}\otimes  
\lau \tN  *  {\ov{b}}  {X;R}  \to \lau \tN  * {\ov{a}+\ov{b}}  {X;R}
$
is the restriction of 
$
- \cup - \colon \lau \tN * {\ov{p}} {X;R}\otimes  
\lau \tN  *  {\ov{q}}  {X;R}  \to \lau \tN  * {\ov{p}+\ov{q}}  {X;R}.
$
\end{remark}

\section{Intersection and tame intersection (co)homology.}

\label{IC}


For a  perversity $\ov p$ such that $\ov p \not \leq \ov t$, we may have $\ov p$-allowable simplices that are not regular. This failure has bad consequences; for instance, Poincaré duality may be not satisfied in this case (see
\cite{CST4}).
To overcome this problem the tame intersection homology $\lau \gH {\ov p} * {X;R}$
has been  introduced in \cite{MR2210257}. 
In this work, we use the simpler presentation of G. Friedman (see \cite{MR2209151,LibroGreg,CST3}).

We begin by recalling the  notions of intersection homology. As proved in \cite[Proposition A29]{CST1}, it can be computed  using filtered chains.
\bd
Consider a perverse space  $(X,\ov p)$  and 
 a filtered  simplex
$\sigma\colon\Delta=\Delta_{0}\ast \cdots\ast\Delta_{n} \to X$.
\begin{enumerate}[{\rm (i)}]
\item The \emph{perverse degree of } $\sigma$ is  the $(n+1)$-tuple,
$\|\sigma\|=(\|\sigma\|_0,\ldots,\|\sigma\|_n)$,  
where
 $\|\sigma\|_{i}=\dim \sigma^{-1}(X_{n-i})=\dim (\Delta_{0}\ast\cdots\ast\Delta_{n-i})$, 
with the convention $\dim \emptyset=-\infty$.
 \item Given a stratum $S$ of  $X$, the \emph{perverse degree of $\sigma$ along $S$} is defined by \ 
 $$\|\sigma\|_{S}=\left\{
 \begin{array}{cl}
 -\infty,&\text{if } S\cap \im \sigma=\emptyset,\\
 \|\sigma\|_{\codim S}&\text{if } S\cap \im \sigma\ne\emptyset.\\
  \end{array}\right.$$

  \item The filtered singular simplex $\sigma$  is  \emph{$\ov{p}$-allowable} if
  $
  \|\sigma\|_{S}\leq \dim \Delta-\codim S+\ov{p}(S),
  $
   for any stratum $S$.

   \item A chain $c$ is 
   \emph{$\ov{p}$-allowable} if it  is a linear combination of  $\ov{p}$-allowable simplices.
   The chain $c$ is a  \emph{$\ov{p}$-intersection chain} if $c$ and $\partial c$ are $\ov{p}$-allowable chains.
   \end{enumerate}
   \ed

 \begin{definition}\label{def:chaineintersection} 
 Consider a perverse space $(X,\ov p)$.
 We denote by
$\hiru C {*}{X;R}$ the complex of filtered chains of  $X$, generated by filtered simplices. 
 The dual complex is
$\lau C {*} {}{X;R}=\hom(\lau C {} *{X;R},R)$.

The complex of $\ov{p}$-intersection chains of $X$  with the differential $\partial$
is denoted by
$\lau C{\ov{p}}* {X;R} $. 
Its homology  $\lau H {\ov{p}} * {X;R}$  is the \emph{$\ov{p}$-intersection homology} of $X$ (see \cite[Théorème A]{CST3}).
The dual complex
$\lau C {*} {\ov{p}}{X;R}=\hom(\lau C {\ov{p}} *{X;R},R)$ computes the \emph{$\ov p$-intersection cohomology} $\lau H {*} {\ov{p}}{X;R}$ (see \cite{LibroGreg}).

\end{definition}
\bd
Given a regular simplex $\Delta = \Delta_0 * \cdots *\Delta_n$ we denote by $\gd$ the regular part of the chain $\partial \Delta$.  That is
$\gd \Delta =\partial (\Delta_0 * \cdots * \Delta_{n-1})* \Delta_n$, if $|\Delta_n| = 0 $, or $\gd \Delta = \partial \Delta$,
if $|\Delta_n|\geq 1$.
\ed\bd \label{defgd}
Let  $(X,\ov p)$ be a perverse space.
 Given a regular simplex $\sigma \colon\Delta \to X$, we define the chain $\gd \sigma$ by $\sigma \circ \gd$.
 Notice that $\gd^2=0$.
 We denote by $\hiru \gC * {X;R}$ the chain complex generated by the regular simplices, endowed with the differential $\gd$.
\ed

\bd\label{tameNormHom}
Let  $(X,\ov p)$ be a perverse space.
A $\ov p$-allowable  filtered simplex $\sigma \colon \Delta \to X$  is  \emph{$\ov{p}$-tame} if $\sigma$ is also a regular simplex.
   A chain $\xi$ is 
   \emph{$\ov{p}$-tame} if is a linear combination of  $\ov{p}$-tame simplices.
   The chain $\xi$ is a  \emph{tame $\ov{p}$-intersection chain} if $\xi$ and $\gd \xi$ are $\ov{p}$-tame chains.

We write  $\lau \gC {\ov{p}} * {X;R} \subset \hiru \gC * {X;R}$ the complex of tame $\ov{p}$-intersection chains endowed with the differential  $\gd $.
Its homology  $\lau \gH {\ov{p}}{*} {X;R}$  is the \emph{tame $\ov{p}$-intersection homology} of $X$
\end{definition}

Main properties of this homology have been studied in \cite{CST3,LibroGreg}.
We have proven in \cite{CST3} that the homology
$\lau \gH {\ov{p}} * {X;R}$
coincides with those of \cite{MR2210257,MR2276609} (see also \cite[Chapter 6]{LibroGreg}).
It is also proved that
$ \lau \gH {\ov{p}} * {X;R} = \lau H {\ov{p}} * {X;R}$
if $\ov{p}\leq\ov{t}$.
The  \emph{tame $\ov{p}$-intersection cohomology} $\lau \gH *{\ov p}  {X;R}$
is defined  by using the dual complex  $\lau \gC * {\ov p}  {X;R} = \Hom (\lau \gC  {\ov p} *  {X;R},R)$.
It verifies $ \lau \gH * {\ov{p}}  {X;R} = \lau H * {\ov{p}}  {X;R}$
if $\ov{p}\leq\ov{t}$.

\subsection{Shifted filtrations}\label{SF}
Intersection homology  does not depend on the dimension of the strata of  $X$ but on the codimension of these strata  (see \cite[Section 4.3.1]{LibroGreg}). Let us consider on $X$ the shifted filtration $Y$ of \eqref{shif} in subsection \ref{SF1}.
So, the formal dimension of $Y$ is $n+m$. 
Following subsection \ref{SF1} we have:
$$
\|\sigma\|_\ell^{^X} = \dim \sigma^{-1}(X_{n-\ell}) = \dim \sigma^{-1}(Y_{n+m-\ell})=
\|\sigma\|_{\ell}^{^Y}.
$$
The allowability condition is the same. The perversity  $\ov p$ is also a perversity  on $Y$. If $S$ is a stratum of $X$ (and therefore of $Y$) with $\im \sigma \cap S \ne \emptyset$ we have $\ell = \codim_XS =\codim_Y S $. So,
$$
\|\sigma\|_\ell^{^X}  \leq \dim \Delta -\ell + \ov p(S) \Longleftrightarrow 
\|\sigma\|_\ell^{^Y}  \leq \dim \Delta -\ell + \ov p(S).
$$
We get $\lau C {\ov p} *{X;R} = \lau C {\ov p} *{Y;R}$ and therefore $\lau H {\ov p} *{X;R} = \lau H {\ov p} *{Y;R}$.
By the same reasons, we also have $\lau H *{\ov p} {X;R} = \lau H *{\ov p} {Y;R}$, $\lau \gH {\ov p} *{X;R} = \lau \gH {\ov p}*{Y;R}$ and
$\lau \gH *{\ov p} {X;R} = \lau \gH *{\ov p} {Y;R}$.

\section{Cap product.}\label{ProdCap}

We introduce the notion of cap product for the blown-up intersection (co)homo\-logy, already treated in  \cite{CST7} in a different context. First of all, we work at the simplex level.

Let $\Delta$ be an $m$-dimensional  euclidean simplex. 
We denote by $[\Delta]$ its  face of maximal dimension.
The \emph{classical cap product}
 $-\cap [\Delta]\colon \Hiru N *\Delta\to \hiru N {m-*}\Delta$
  is defined by
\be\label{equa:cap0}
\1_{F}\cap [\Delta]=\left\{
\begin{array}{cl}
G
&
\text{if  $F\cup G=\Delta$  (where $\cup$ is the map of  \defref{def:cupsurDelta})},\\
0
&
\text{otherwise.}
\end{array}\right.
\ee

Consider now the cone  $\tc\Delta$ whose apex is denoted by $\tv$, which is the last vertex of the cone (see  \pagref{virt}). We have the formula:
\begin{equation}\label{equa:cap1}
\1_{(F,j)}\cap [\tc\Delta]=\left\{
\begin{array}{cl}
(G,1)&\text{if } j=0 \hbox{ or } (F,j) = (\Delta,1),\\
%
0&\text{if not,}
\end{array}\right.
\end{equation}
where the simplex $G$ is the face of $\Delta$ with  $F\cup G=\Delta$ (cf. \defref{def:cupsurDelta}) if $j=0$ and $G = \emptyset$ if $(F,j) = (\Delta,1)$.

\medskip

We extend it to regular simplices $\Delta=\Delta_{0}\ast\dots\ast\Delta_{n}$
as follows. Denote  $\hiru \tN{*} \Delta =\hiru N{*}{\tc\Delta_{0}} \otimes\dots\otimes \hiru N {*}{\tc \Delta_{n-1}}\otimes \hiru N{*}{\Delta_{n}}$.

\bd\label{61}
We define the \emph{cap product} 
$
-\tcap \tDelta\colon  \Hiru \tN {*} \Delta \to  \hiru \tN{*-m} \Delta,
$
linearly from 
\be\label{equa:lecapsurdelta}
\1_{(F,\varepsilon)} \tcap \tDelta
=
(-1)^{\nu(F,\varepsilon,\Delta)}
(\1_{(F_{0},\varepsilon_{0})}\cap \tc[\Delta_{0}])\otimes\dots\otimes
(\1_{F_{n}}\cap [\Delta_{n}]),
\ee
where  $\1_{(F,\varepsilon)}=\1_{(F_{0},\varepsilon_{0})}\otimes\dots\otimes \1_{(F_{n-1},\varepsilon_{n-1})}\otimes \1_{F_{n}}$, 
 $\tDelta = \tc \Delta_0 \times \cdots \times \tc \Delta_{n-1} \times \Delta_n$ (see \eqref{bup}) and  $\nu(F,\varepsilon,\Delta)=\sum_{j=0}^{n-1}(\dim\Delta_{j}+1) \,(\sum_{i=j+1}^n|(F_{i},\varepsilon_{i})|)
$, with the convention $\varepsilon_{n}=0$.

At the filtered simplex level, the \emph{local intersection cap product}
$-\cap  \tDelta\colon \Hiru  \tN* \Delta\to \hiru N{m-*}\Delta$ is defined by
$$c\cap \tDelta=\mu_{*}(c\tcap\tDelta),$$
where the map
$\mu_{*}\colon \hiru \tN {*}\Delta \to \hiru N{*}\Delta$ is defined by
\begin{equation}\label{AmuA}
\mu_{*}\left(\otimes_{k=0}^{n-1} (F_k,\varepsilon_k) \otimes F_n
\right)=\left\{
\begin{array}{cl}
F_{0}\ast\dots\ast F_{\ell}&\text{if } 
\dim (F,\varepsilon) =\dim (F_{0}\ast\dots\ast F_{\ell}),\\
0&\text{otherwise,}
\end{array}\right.
\end{equation}
where 
$(F,\varepsilon)=(F_{0},\varepsilon_{0})\otimes\dots\otimes (F_{n-1},\varepsilon_{n-1})\otimes F_{n}$ and $\ell$  is the smallest integer, $j$, such that $\varepsilon_{j}=0$.

\ed

\bd
Since $\tDelta=\tc\Delta_{0} \times\dots\times \tc \Delta_{n-1}\times \Delta_{n}$ (see \eqref{bup}), we have the boundary chain
\begin{eqnarray*}
\partial\tDelta
=
\mathop{\sum_{i=0}^n}_{\Delta_i\ne \emptyset} (-1)^{|\Delta|_{\leq {i-1}}+1}
[\tc \Delta_{0}]\otimes\dots\otimes [\tc \partial\Delta_{i}]\otimes\dots\otimes [\Delta_{n} ]
 &&
 \\+\mathop{\sum_{i=0}^{n-1}}_{\Delta_i\ne \emptyset}
(-1)^{|\Delta|_{\leq i} + 1}
[\tc\Delta_{0}]\otimes\dots\otimes [\Delta_{i}]\otimes\dots\otimes [\Delta_{n}].
\end{eqnarray*}
Let $i\in \{0, \dots, n-1\}$ such that $\Delta_i\ne \emptyset$, the face 
$\mathcal H_i = \tc \Delta_{0} \times\dots \times \tc \Delta_{i-1} \times \Delta_{i}\times \tc \Delta_{i+1} \times \dots\times \Delta_{n}$ is called a \emph{hidden face} of $\tDelta$. This gives
\begin{eqnarray}\label{Hid}
\partial\tDelta=
\widetilde{\gd\Delta} +
\mathop{\sum_{i=0}^{n-1}}_{\Delta_i\ne \emptyset}
(-1)^{|\Delta|_{\leq i} + 1} \mathcal \  \mathcal  H_i, 
\end{eqnarray}
where $\widetilde {\gd \Delta}$ is the blow up of the weighted simplicial complex  $\gd \Delta$ (cf. \defref{def:eclate}).
\ed

\begin{remark} \label{GG}

\color{white}.\normalcolor

\begin{itemize}
\item 
For each $\omega \in \Hiru \tN *\Delta$ the chain $\omega \cap \tDelta$ is regular or 0. Let us see that.
Let $\1_{(F,\varepsilon)}=
\1_{(F_{0},\varepsilon_{0})}\otimes\cdots\otimes \1_{F_{n}} \in \tN^*(\Delta)$.  
From the definition of the cap product we have, up to a sign,
$$\1_{(F,\varepsilon)}\cap\tDelta=G_{0}\ast\cdots \ast G_{n},$$
if $\1_{(F,\varepsilon)}\cap\tDelta\ne 0$. The faces $G_\bullet$ have been defined in \eqref{equa:cap0}, \eqref{equa:cap1} and \eqref{AmuA}.
The simplex $\1_{(F,\varepsilon)}\cap\tDelta$ is regular: $F_n\ne \emptyset
\Rightarrow G_n\ne \emptyset$.

\item Given a hidden face $\mathcal H_i = \tc \Delta_{0} \times\dots  \times \Delta_{i} \times \dots\times \Delta_{n}$, we have
\begin{equation*}
\mu_{*}\left(\cal H_i\right)=\left\{
\begin{array}{cl}
\Delta_{0}\ast\dots\ast \Delta_{i}&\text{if } 
\dim(\Delta_{i+1}* \dots * \Delta_n)=0\\
0&\text{otherwise}
\end{array}\right.,
\end{equation*}
which is a non regular face or 0.
\end{itemize} 
\end{remark}

\medskip

We define the intersection cap product for any filtered space $X$. 

\bd Let $X$ be a filtered space.
The \emph{intersection cap product} 
\bee
- \cap - \colon \Hiru \tN * {X;R} \otimes \lau \gC {}m {X;R}
\to \lau \gC{} {m-*} {X;R}
\eee
is defined by linear extension of 
$$
\omega\cap \sigma=
\sigma_{*}(\omega_{\sigma}\cap \tDelta),
$$
where  $\omega\in \Hiru \tN *{X;R}$ and  $\sigma\colon \Delta\to X$ is a regular simplex.
\ed

\bl\label{bebe}
Let $\Delta$ be a regular simplex.
The map
$\mu_{*}\colon \hiru \tN {*}\Delta \to \hiru N{*}\Delta$
commutes with differentials.
\el

\begin{proof}
Set $\Delta= \Delta_0 * \cdots * \Delta_n$. We proceed by  induction on $n$.
The result is clear when $n=0$. We consider the regular simplex $\nabla=\Delta_{1}\ast\dots\ast\Delta_{n}$,
et $\tN_{*}(\nabla)=N_{*}(\tc\Delta_{1})\otimes\dots\otimes N_{*}(\tc \Delta_{n-1})\otimes N_{*}(\Delta_{n})$. The map $\mu_{\Delta,*} $ is the composition
$$
\mu_{\Delta,*} \colon N_{*}(\tc \Delta_{0})\otimes \tN_{*}(\nabla)
\xrightarrow[]{\id\otimes \mu_{\nabla,*} }
N_{*}(\tc\Delta_{0})\otimes N_{*}(\nabla)
\xrightarrow[]{\nu}
N_{*}(\Delta),
$$
(see \eqref{AmuA}),
where 
$$\nu((F_{0},\varepsilon_{0})\otimes G)=\left\{
\begin{array}{cl}
F_{0}&\text{if } \varepsilon_{0}=0 \text{ and } \dim G=0,\\
0&\text{if }\varepsilon_{0}=0 \text{ and } \dim G>0,\\
F_{0}\ast G&\text{if } \varepsilon_{0}=1.
\end{array}\right.$$
By induction hypothesis, the operator $\mu_{\nabla,*}$ is compatible with differentials. It remains to prove that $\nu$ is also compatible with differentials. We distinguish four cases.
\begin{itemize}
\item If $\varepsilon_{0}=0$ and $\dim G=0$, we have 
$\nu(\partial((F_{0},0)\otimes G))=
\nu((\partial F_{0},0)\otimes G)=\partial F_{0}=\partial \nu((F_{0},0)\otimes G)$.

\item If $\varepsilon_{0}=0$ and $\dim G=1$, we have 
$\nu(\partial((F_{0},0)\otimes G))=
\nu((\partial F_{0},0)\otimes G)+(-1)^{|F_{0}|}\nu((F_{0},0)\otimes \partial G)=
0+ F_{0}-F_{0}=0=\partial \nu((F_{0},0)\otimes G)$.

\item If $\varepsilon_{0}=0$ and $\dim G>1$, all the terms vanish.

\item If $\varepsilon_{0}=1$, the result comes from
\begin{eqnarray*}
\nu(\partial((F_{0},1)\otimes G))=&&\\
\nu((\partial F_{0},1)\otimes G)+
(-1)^{|F_{0}|+1} \left(\nu((F_{0},0)\otimes  G) +
 \nu((F_{0},1)\otimes \partial G)\right)=&&\\
 (\partial F_{0})\ast G + (-1)^{|F_{0}|+1}\left\{
 \begin{array}{cl}
 F_{0}&\text{if } \dim G=0,\\
 F_{0}\ast \partial G&\text{if } \dim G\geq 1,
 \end{array}\right.=&&\\
 \partial (F_{0}\ast G)=\partial \nu((F_{0},1)\otimes G).&&
\end{eqnarray*}
\end{itemize}
\end{proof}

\begin{proposition}\label{prop:lecap}
The intersection  cap product 
verifies the following properties

\begin{itemize}
\item[(i)] $\gd (\omega\cap\xi)=(\delta \omega)\cap \xi+(-1)^{|\omega|}\omega\cap \gd \xi$, and
\item[(ii)] $(\omega\cup\eta)\cap \xi=(-1)^{|\omega| \cdot |\eta|} \ \eta\cap \omega\cap\xi $.
\end{itemize}
where $\omega  , \eta\in \lau \tN * {} {X;R}$ and  $\xi \in \lau \gC {} * {X;R}$.
\end{proposition}

\begin{proof} 
(i) It suffices to prove it for a regular simplex  $\sigma \colon \Delta \to X$.  We write $k = |\omega|$.
Using formula \eqref{Hid}, the classic Leibniz formula for the cap product gives
\begin{eqnarray*}
\partial (\omega_\sigma \cap \tDelta )=\partial \mu_* (\omega_\sigma \tcap \tDelta )  \stackrel{\lemref{bebe}}{=}
\mu_* \partial (\omega_\sigma \tcap \tDelta )=&& \\
\mu_*\left(
(\delta\omega_\sigma )\tcap \tDelta+ (-1)^k \omega_\sigma \tcap (\partial\tDelta)
\right)
=&&\\
\stackrel{\eqref{Hid}}{=}\mu_* \left( (\delta\omega_\sigma )\tcap \tDelta+(-1)^k \omega_\sigma \tcap \widetilde{\gd \Delta}
+ \mathop{\sum_{i=0}^{n-1}}_{\Delta_i\ne \emptyset}(-1)^{(-1)^{k+ |\Delta|_{\leq i}}} \omega_\sigma \tcap \mathcal H_i \right) =&&\\
=
\underbrace{\mu_* \left(  (\delta\omega_\sigma )\tcap \tDelta+(-1)^k \omega_\sigma \tcap \widetilde{\gd \Delta}\right)}_{\hbox{\tiny regular or 0}}
+ \mathop{\sum_{i=0}^{n-1}}_{\Delta_i\ne \emptyset}(-1)^{(-1)^{k+ |\Delta|_{\leq i}}}  
\underbrace{\mu_* (\omega_\sigma \tcap \mathcal H_i )}_{\hbox{\tiny non regular or 0}}
\end{eqnarray*}
(cf. \remref{GG}). So,
$
\gd(\omega_\sigma \cap \tDelta )
=
 \mu_* ((\delta\omega_\sigma )\tcap \tDelta)
 +(-1)^k  \mu_*(\omega_\sigma \tcap \widetilde{\gd \Delta})
 =
 (\delta\omega_\sigma )\cap \tDelta
 +(-1)^k \omega_\sigma \cap \widetilde{\gd \Delta},
 $
and therefore
$
\sigma_*\gd(\omega_\sigma \cap \tDelta )=
 (\delta\omega)\cap \Delta
 +(-1)^k \omega \cap \gd \Delta.
 $

 The regular simplex $\sigma \colon \Delta \to X$ is in fact a stratified map when one considers on $\Delta$ the filtration $\Delta_0 \subset \Delta_0*\Delta_1 \subset \cdots \subset \Delta_0 * \cdots * \Delta_n$.
 So, $\sigma_* \circ \partial = \partial \circ \sigma_*$ \cite[Theorem F]{CST1}.
 Since $\Delta_0* \cdots * \Delta_{n-1} = \sigma^{-1}(\Sigma_X)$ then  $\sigma_* \circ \gd = \gd \circ\sigma_*$.
We get property (i) since
$
\gd (\omega \cap \Delta ) = \gd \sigma_* (\omega_\sigma  \cap \tDelta )
=
\sigma_* \gd  (\omega_\sigma  \cap \tDelta ).
 $

\medskip

(ii) Without loss of generality we can suppose $\omega_\sigma =\1_{(F,\varepsilon)}=\1_{(F_{0},\varepsilon_{0})}\otimes\dots\otimes \1_{(F_{n-1},\varepsilon_{n-1})}\otimes \1_{F_{n}}$
and
$\eta_\sigma =\1_{(H,\kappa)}=\1_{(H_{0},\kappa_{0})}\otimes\dots\otimes \1_{(H_{n-1},\kappa_{n-1})}\otimes \1_{H_{n}}$. It suffices to prove
$$
(\1_{(F,\varepsilon)} \cup \1_{(H,\kappa)}) \tcap \tDelta = (-1)^{|(F,\varepsilon)| \cdot |(H,\kappa)|} \ \1_{(H,\kappa)} \tcap (\1_{(F,\varepsilon)} \tcap \tDelta).
$$
For $n=0$ we find the classic property of cup/cap products.
In the general case, with the convention $\epsi_n=\kappa_n=0$, we have
\begin{eqnarray*}
(\1_{(F,\epsi)}  \cup \1_{(H,\kappa)}) \tcap \tDelta
=
\left( \bigotimes_{a=0}^n \1_{(F_a,\epsi_a)} 
 \cup
\bigotimes_{a=0}^n \1_{(H_a,\kappa_a)}
 \right)
 \tcap \tDelta  \stackrel{\eqref{equa:cupsimplexfiltre}}{=}&&\\
 (-1)^{\ltimes_1} \cdot 
\left(\bigotimes_{a=0}^n \ \left( \1_{(F_a,\epsi_a)} \cup \1_{(H_a,\kappa_a)}\right)
 \right)\tcap  \tDelta \stackrel{\eqref{equa:lecapsurdelta}}{=}&&\\
 (-1)^{\ltimes_2} \cdot
\bigotimes_{a=0}^{n-1} \ 
\left(
\left( \1_{(F_a,\epsi_a)} \cup \1_{(H_a,\kappa_a)} 
\right)
\cap [\tc \Delta_a ]\right) \otimes  
\left( \left( \1_{F_n} \cup \1_{H_n} \right)
\cap  [\Delta_n]\right)  \stackrel{classic}{=}&&\\
   (-1)^{\ltimes_3} \cdot
\bigotimes_{a=0}^{n-1} \ 
 \left(\1_{(H_a,\kappa_a)} \cap\1_{(F_a,\epsi_a) }\cap [\tc \Delta_a ]\right)
 \otimes
 (\1_{H_n} \cap\1_{F_n }\cap [\Delta_n]) \stackrel{\eqref{equa:lecapsurdelta}}{=}&&
\\
  (-1)^{\ltimes_4}\cdot
  \1_{(H,\kappa)}  \tcap
 \left( 
 \bigotimes_{a=0}^{n-1} \
 \left(\1_{(F_a,\epsi_a) }\cap [\tc \Delta_a ]\right)
 \otimes
 (\1_{F_n }\cap [\Delta_n])
 \right)
   \stackrel{\eqref{equa:lecapsurdelta}}{=}&&\\
 (-1)^{|(F,\epsi)| \cdot |(H,\kappa)|} \cdot \1_{(H,\kappa)}  \tcap
 \1_{(F,\epsi) }  \tcap  \tDelta.
\end{eqnarray*}
where
$
\ltimes_1 =\sum_{i>j}|(F_i,\epsi_i)| \cdot |(H_j,\kappa_j)|,
\ltimes_2 =\ltimes_1 + \sum_{j=0}^{n-1}(\dim\Delta_{j}+1) \,(\sum_{i>j}|(F_{i},\varepsilon_{i})|+ |(H_i,\kappa_i)|),
\ltimes_3=\ltimes_2 +\sum_{i}|(F_i,\epsi_i)| \cdot |(H_i,\kappa_i)|$ and
\begin{eqnarray*}
\ltimes_4 &=&\ltimes_3 +
  \sum_{j=0}^{n-1}(\dim\Delta_{j}+1 - |(F_j,\epsi_j)| ) \,(\sum_{i>j} |(H_i,\kappa_i)|)
=\sum_{i\geq j}|(F_i,\epsi_i)| \cdot |(H_j,\kappa_j)| \\
  && 
  + \sum_{j=0}^{n-1}(\dim\Delta_{j}+1) \,(\sum_{i>j}|(F_{i},\varepsilon_{i})|) - 
  \sum_{j=0}^{n-1}|(F_j,\epsi_j)| \,(\sum_{i>j} |(H_i,\kappa_i)|)\\
  &=&
  \sum_{j=0}^{n-1}(\dim\Delta_{j}+1) \,(\sum_{i>j}|(F_{i},\varepsilon_{i})|) + 
  |(F,\epsi)|\cdot |(H,\kappa)|.
\end{eqnarray*}
\end{proof}

The cap product between a cochain $\omega \in \lau \tN * {} {X;R}$ and a chain $\xi \in \lau C*{} {X;R}$ may not exist. This  happens when a simplex $\sigma$ of  $\xi$ lies in  the singular part of $X$. In this case, $\sigma$ is not regular and we cannot construct the blow up of $\Delta_\sigma$.
This is the reason to use the tame intersection homology in the definition of the cap product: all the involved simplices are regular.

 \begin{proposition}\label{prop:cap}
For each filtered space, $X$, endowed with two perversities, $\ov{p}$ and $\ov{q}$,
the intersection cap product 
 \begin{equation*}\label{equa:capprduitespacefiltre}
 -\cap -\colon \lau \tN * {\ov{p}}{X;R}\otimes \lau \gC {\ov{q}} m {X;R}\to \lau \gC{\ov{p}+\ov{q}}  {m-*} {X;R},
 \end{equation*}
 is well defined and induces a morphism 
  \begin{equation*}\label{equa:capprduitBUcohomologie}
-\cap -\colon  \lau \IH * {\ov{p}}{X;R}\otimes \lau \gH  {\ov{q}} m {X;R}\to
 \lau \gH {\ov{p}+\ov{q}} {m-*}{X;R}.
 \end{equation*}
 \end{proposition}
 
 \bpr
Following \propref{prop:lecap} it remains to 
study the behavior of the cap product with respect to the perverse degree on a  singular stratum $S$.
Let us consider $\omega \in \tN^*(X;R)$ with $\|\omega\|_S \leq \ov{p}(S)$, $\sigma \colon \Delta \to X$ a regular simplex with $\im \sigma \cap S \ne \emptyset$ and  $\|\sigma\|_S  \leq |\Delta|-\ell + \ov{q}(S)$, where  $\ell=\codim S$. 
We need to prove
\be\label{kap}
\|\omega \cap \sigma\|_S \leq |\omega \cap \sigma| - \ell + (\ov p + \ov q)(S).
\ee
Without loss of generality, we can suppose  $\omega \cap \sigma \ne  0$.

 When $\ov p (S) =-\infty$, the condition $\|\omega \|_S \leq -\infty$ implies $\epsi_{n-\ell}=1$ and therefore the $(n-\ell)$-factor of $\omega_\sigma  \tcap \tDelta$ is $[\tv_{n-\ell}]$ which gives $\omega \cap \sigma \cap S =\emptyset$ and therefore  \eqref{kap}. On the other hand, if $\ov q (S) =-\infty$, the condition $\|\sigma\|_S =-\infty$ gives $\Delta_{n-\ell} =\emptyset$ then $\omega \cap \sigma \cap S =\emptyset$ and  therefore \eqref{kap}. 
 
 We can suppose $(\ov p + \ov q )(S) = \ov p(S) +\ov q(S)$. Then we have \eqref{kap} from
$$
\|\omega\cap\sigma\|_{S}
\stackrel{claim}{\leq}
\|\sigma\|_S+\|\omega_{\sigma}\|_{\ell}-|\omega_{\sigma}|
\leq
|\Delta|-\ell+\ov{p}(S)+\ov{q}(S)-|\omega_{\sigma}|
=
|\omega\cap\sigma|-\ell +\ov{p}(S)+\ov{q}(S).
$$
We verify the claim.
It suffices to consider $\omega_\sigma = \1_{(F,\varepsilon)}$ with $\1_{(F,\varepsilon)}\cap \tDelta \ne 0$.
We know, from \remref{GG}, that $\1_{(F,\varepsilon)}\cap \Delta = \pm G_0*\dots*G_n$ with $|\Delta_i| = |G_i| + |(F_i,\varepsilon_i)|$, for each $i \in \{0, \dots ,n\}$. Then, $\| \omega \cap \sigma\|_S $ is equal to
\begin{eqnarray*}
|G_0*\dots * G_{n-\ell}|
&=&
|G_0| + \dots + |G_{n-\ell}| + n-\ell \\
&=&
|\Delta_0|  + \dots + | \Delta_{n-\ell}|   + n-\ell - |(F_0,\varepsilon_0)|  - \dots - |(F_{n-\ell},\varepsilon_{n-\ell})| 
\\
&=&
\|\sigma\|_S - |(F,\varepsilon)|  + |(F,\varepsilon)|_{>n-\ell} =
\|\sigma\|_S - |\omega_\sigma|  +\|\omega_\sigma\|_\ell.
\end{eqnarray*}
 \epr

\begin{remark}\label{67}
The intersection cap product is natural respectively to perversities. 
Given perversities $\ov a \leq \ov p$ and $\ov b \leq \ov q$, the cap product 
$
- \cap - \colon \lau \tN * {\ov{a}} X\otimes  
\lau \gC    {\ov{b}} {m}  X  \to \lau \gC   {\ov{a}+\ov{b}}  {m-*}X
$
is the restriction of 
$
- \cap - \colon \lau \tN * {\ov{p}} X\otimes  
\lau \gC    {\ov{q}} {m}  X  \to \lau \gC   {\ov{p}+\ov{q}}  {m-*}X.
$
\end{remark}

\section{Stratified maps: the local level.}\label{EAmal}


In  this section and the next one, we prove that any stratified map induces a homomorphism between  the blown-up intersection cohomologies if a certain compatibility of the involved perversities is satisfied. In particular, if $\ov p$ is a GM-perversity, any stratified map $f \colon X \to Y$ induces a homomorphism $f \colon \lau \IH * {\ov p} {X;R} \to \lau \IH * {\ov p} {Y;R}$.

Let $f \colon X \to Y$ be a stratified map as in   \defref{def:applistratifieeforte}. If 
$\sigma\colon \Delta \to X$ is a filtered simplex then the composite $f \circ \sigma \colon \Delta \to Y$ is also a filtered simplex but, in general, the two filtrations induced on $\Delta$ by $\sigma$ and $f \circ \sigma$ differ. We denote them by $\Delta_\sigma$ and $\Delta_{f\circ \sigma}$. In  \cite[Corollary A.25 and Lemma A.24]{CST1}, we prove that the two filtrations $\Delta_\sigma$ and $\Delta_{f\circ \sigma}$ (on the same euclidean simplex) match through a finite number of elementary amalgamations that we describe and study in this section.

\bd \label{elemkam}
Let  $k \in \{0, \ldots,n-1\}$. An {\em elementary $k$-amalgamation} of 	a regular simplex $\Delta = \Delta_0 *  \cdots *  \Delta_n$ 
is the  regular simplex: 
$
\Delta' =  \Delta'_0 *   \cdots * 
  \Delta'_{n-1}
$ with
$$
\Delta'_i
=
\left\{
\begin{array}{ll}
\Delta_i & \hbox{if } i \leq n-k-2,\\
\Delta_{n-k-1} * \Delta_{n-k} & \hbox{if } i = n-k-1,\\
\Delta_{i-1} & \hbox{if } i \geq n-k.
\end{array}
\right.
$$
If $\Delta_{n-k-1}=\emptyset$ then the elementary amalgamation  is \emph{$k$-simple}.

\ed
We study the effect of these maps at the blow up level. First of all, we introduce a technical tool. 

\subsection{Stratification} 
The  simplices
$\Delta = \Delta_0 *  \cdots *  \Delta_n$ and  $\Delta' = \Delta'_0 *  \cdots *  \Delta'_{n-1}$ are filtered spaces (see \defref{def:espacefiltre}) of respective dimension  $n$ and $n-1$.
We set $\N_\Delta = \{ i \in \{0,\ldots,n\} \ / \ 
\Delta_0 * \cdots * \Delta_{n-i - 1} \ne 
\Delta_0 * \cdots * \Delta_{n-i}\}$. We observe that the family of strata of $\Delta$ is
$$\cal S_\Delta =\{ S_{i} =
\Delta_0 * \cdots * \Delta_{n-i } \backslash 
\Delta_0 * \cdots * \Delta_{n-i - 1} \  / \ i \in \N_\Delta\}.
$$
Notice that  $\codim S_i =i$ for each $i \in \N_\Delta$.

If $k \in \{0, \ldots,n-1\}$, we denote $\epsi_k \colon \{0,\ldots,n\} \to \{0,\ldots,n-1\}$ the map defined by $\epsi_k(i) = i
$ if $i \leq k$ and $\epsi_k(i) = i-1$ if $i\geq k+1$.
We check easily that this maps restricts as a map $a_k \colon \N_\Delta \to \N_{\Delta'}$, called the \emph{index map}.
If $S_i \in \cal S_\Delta$, we have $S_{i} \subset S'_{a_k(i)}$.
Thus by definition, the identity map is a stratified map from $\Delta$ to $\Delta'$, we denote it 
$$
\cal A [k] \colon \Delta  = \Delta_0 * \cdots * \Delta_n\to \Delta' = \Delta_0 * \cdots *\Delta'_{n-1}.
$$
If $ {S} \in \cal S_\Delta$, with the notations of \defref{def:applistratifieeforte} we have ${S}^{\cal A [k] } = S'_{a_k(i)}$, with $i=\codim  {S}$, and therefore
\be\label{indexa}
\codim  {S}^{\cal A [k] } =  a_k(\codim  {S}).
\ee

We also say that 
$
\cal A [k] 
$
is an elementary amalgamation.
If the elementary $k$-amalgam\-ation is simple, then $a_k$ is a bijection and $\cal A[k] $ is a  stratified homeomorphism (see \defref{def:applistratifieeforte}).

\bigskip

The elementary $k$-amalgamation has an effect only on two components . We first focus on them. For the sake of the simplicity we write $E_0 = \Delta_{n-k-1}$ and $E_1 = \Delta_{n-k}$, they can be empty.

\bd\label{LemaTec}
We define a map
$$\theta \colon  
 \hiru N *{\tc  E_{0}} \otimes  \hiru N * {\tc E_1}
  \to 
\hiru N *{\tc (E_0*  E_1)}$$
by
\begin{eqnarray*}\label{deftetapeq*}
\theta
((F_0,\epsi_0) \otimes (F_{1},\epsi_1) ) =
\left\{
\begin{array}{ll}
(F_0*F_1,\epsi_1) & \hbox{if } \epsi_0=1,\\
(F_0,0) & \hbox{if } \epsi_0=0 \hbox{ and } |(F_1,\epsi_1) |=0,\\
0 & \hbox{if not.}
\end{array}
\right.
\end{eqnarray*}
Notice that the restriction of $\theta$ to $ \hiru N *{\tc  E_{0}} \otimes  \hiru N * { E_1}$ gives a map, still denoted $\theta$,
$$\theta \colon  
 \hiru N *{\tc  E_{0}} \otimes  \hiru N * { E_1}
  \to 
\hiru N *{E_0*  E_1}.$$
By duality, we get a map
$$\Xi \colon  
 \Hiru N *{\tc (E_0*  E_1)} \to 
\Hiru N *{\tc  E_{0}} \otimes  \Hiru N * {\tc E_1}$$
which verifies
\bee
\left\{
\begin{array}{lcll}
\Xi(\1_{( F_0*F_1,\varepsilon) })&=&
(-1)^{|(F_0,1)| \cdot |(F_1,\varepsilon)|} \1_{(F_0,1)}\otimes \1_{(F_1,\varepsilon)} &\hbox{if }(F_1,\epsi) \ne  (\emptyset,0).
\\[.2cm] \nonumber
  \Xi(\1_{(F_0,0)}) 
&=&
\1_{(F_0,0)}\otimes \lambda_{\tc E_1},
&\hbox{with }
\lambda_{\tc E_1} = \1_{(\emptyset,1)}   \\
&&&
+ 
\displaystyle \sum_{e \in\cal V( E_1)}  \1_{(e,0)}.
\end{array}
\right.
\eee

\ed
We also denote
$\Xi \colon  
 \Hiru N *{E_0*  E_1} \to 
\Hiru N *{\tc  E_{0}} \otimes  \Hiru N * { E_1}$
 the restriction of  $\Xi$.

\bl\label{CompBlow}
The morphisms $\theta $ and $\Xi$ are  chain maps. 
Moreover,
\begin{enumerate}[(1)]
\item if $E_0=\emptyset$ then $\theta$ and $\Xi$ are isomorphisms,

\item $\Xi$ is compatible with the cup products,

\item $\theta$ and $\Xi$ are compatible with the cap product, i.e.,
\begin{eqnarray*}
\theta (\Xi \left(\1_{(F_0 *F_1,\varepsilon)}  \right) \tcap ([\tc E_0] \otimes [\tc E_1] )) &=&
\1_{(F_0 *F_1,\varepsilon)}   \cap [\tc(E_0*E_1)], \hbox{ and} \\
\theta (\Xi \left(\1_{F_0 *F_1}  \right) \tcap ([\tc E_0] \otimes[ E_1] )) &=&
\1_{F_0 *F_1}   \cap [E_0*E_1]
\end{eqnarray*}
where $[\nabla]$ is the maximal simplex of an euclidean simplex $\nabla$, $(F_0*F_1,\varepsilon)$,
is a face of $\tc(E_0 * E_1) $ and $F_0*F_1$,
is a face of $E_0 * E_1 $.

\end{enumerate}

\el
We postpone the proof of this Lemma and deduce first some properties on amalgamations from it.

\bd\label{85}
Let us consider the elementary $k$-amalgamation $\cal A[k] \colon \Delta \to \Delta'$ with  $k\in\{0, \ldots,n-1\}$ (see \defref{elemkam}).
We define two homomorphisms
$
\cal A[k]_* \colon  \hiru \tN * {\Delta}\to \hiru \tN * {\Delta'}
$
 and
$
\cal A[k]^* \colon \Hiru \tN * {\Delta'} \to \Hiru \tN * {\Delta}
$
 by
 \be\label{Abajo}
\cal A[k]_* = \underbrace{\id \otimes \cdots \otimes \id}_{n-k-1 \ times}
\otimes \ \theta \ \otimes \underbrace{\id  \otimes \cdots \otimes \id }_{k \ times}
\ee
and
\be\label{Aalto}
\cal A[k]^* = \underbrace{\id \otimes \cdots \otimes \id}_{n-k-1 \ times}
\otimes \ \Xi \ \otimes \underbrace{\id  \otimes \cdots \otimes \id }_{k \ times}
\ee
\ed

If there is not ambiguity, we use the notation $\cal A_*$ and $\cal A^*$ for these two maps and $a$ for the index map.  In order to get coherent notation, we write $\|-\|_0=0$.

\bp\label{CMAmal}
Let $\cal A \colon \Delta \to \Delta'$ be an elementary $k$-amalgamation. Then,  the maps $\cal A_*$  and  $\cal A^*$ defined above satisfy the following properties.
\begin{enumerate}[(1)]

\item They are chain maps and $\cal A^*$ is compatible with the cup product.

\item Let $\omega \in \Hiru \tN * {\Delta'}$, then we have 
$$
\cal A_*( \cal A^*(\omega) \tcap \tDelta) =  \omega \tcap \widetilde{\Delta'}.
$$

\item In the case of simple amalgamations, they are isomorphisms.

\item For each regular face operator $\delta_\ell \colon \nabla \to \Delta$  we denote by $\cal A_{\Delta}^*$ the previous map and by $\cal A_\nabla^*$ the map corresponding to the amalgamation induced on $\nabla$.
Then we have
$$ \delta_\ell^* \circ  \cal A ^*_\nabla=  \cal A ^*_\Delta  \circ   \delta_\ell^*.
$$

\item The map $\cal A^*$ decreases the perverse degree, i.e., for $\omega \in \Hiru \tN * {\Delta'}$ and $\ell \in \N_\Delta$, we have $$
\|\cal A^*(\omega)\|_\ell  \leq 
\|\omega\|_{a(\ell )}.$$
Moreover, if the amalgamation is simple, this inequality becomes an equality,
$$
\|\cal A^*(\omega)\|_\ell  =
\|\omega\|_{a(\ell )}.$$

\end{enumerate}
\ep
\bpr
The properties (1) and (3) are consequences of  \lemref{CompBlow} (2),(1).

\medskip

(2) From the definition of $\cal A_*$ and $\cal A^*$ it suffices to apply \lemref{CompBlow} (3).

\medskip

(4) Direct from definitions.

\medskip

(5) 
Let $\omega=\1_{(F_0,\varepsilon_0)} \otimes \cdots   \otimes \1_{(F_{n -k-2},\varepsilon_{n-k-2} )}  \otimes \1_{(F_{n-k-1}*F_{n-k},\varepsilon)}  \otimes  \1_{(F_{n-k+1},\varepsilon_{n-k+1} )}  \otimes
 \cdots \otimes \1_{F_{n}} \in \Hiru \tN * {\Delta'}$.
 We distinguish two  cases (see Definitions \ref{LemaTec}, \ref{85}):
 
 \bigskip
 
   \begin{itemize}
   
  \item Suppose $(F_{n-k} ,\epsi) \ne (\emptyset,0)$. We have, up to a sign, 
$
\cal A^* (\omega) = \pm
\1_{(F_0,\varepsilon_0)} \otimes \cdots  \otimes \1_{(F_{n-k-2},\varepsilon_{n-k-2})}  \otimes 
 \1_{(F_{n-k-1},1)} \otimes    \1_{(F_{n-k},\varepsilon)} \otimes  \1_{(F_{n-k+1},\varepsilon_{n-k+1} )} \otimes
 \cdots \otimes\1_{(F_{n},\varepsilon_{n})}.$
  The result comes from
  $$
\| \cal A^* (\omega) \|_\ell  =
\left\{
\begin{array}{cl}
\|\omega\|_{\ell -1}  & \hbox{if } \ell \geq k +2\\
-\infty & \hbox{if } \ell  = k+1\\
\|\omega\|_{\ell } & \hbox{if } \ell \leq k
\end{array}
\right.
\leq
\left\{
\begin{array}{cl}
\|\omega\|_{\ell -1}  & \hbox{if } \ell \geq k +1\\
\|\omega\|_{\ell } & \hbox{if } \ell \leq k
\end{array}
\right.
=\|\omega\|_{a(\ell)},
$$
for $\ell \in \N_\Delta$.

\item Suppose $(F_{n-k},\epsi) =(\emptyset,0)$. 
We have the equality
$
\cal A^* (\omega) =
\1_{(F_0,\varepsilon_0)} \otimes \cdots  \otimes \1_{(F_{n-k-2},\varepsilon_{k-2})}  \otimes 
 \1_{(F_{n-k-1},0)} \otimes    \1_{(\emptyset,1)} \otimes \1_{(F_{n-k+1},\varepsilon_{n-k+1} )}  \otimes
 \cdots \otimes\1_{(F_{n},\varepsilon_{n})}
 +
\sum_{e \in \cal V(\Delta_k)} 
\1_{(F_0,\varepsilon_0)} \otimes \cdots  \otimes \1_{(F_{n-k-2},\varepsilon_{n-k-2})}  \otimes 
 \1_{(F_{n-k-1},0)} \otimes    \1_{(e,0)} \otimes 
 $ \\
 $
 \1_{(F_{n-k+1},\varepsilon_{n-k+1} )} \otimes
 \cdots \otimes\1_{(F_{n},\varepsilon_{n})}.
 $
Then we have
  $$
\|\cal A^* (\omega) \|_\ell  =
\left\{
\begin{array}{cl}
\|\omega\|_{\ell -1}  & \hbox{if } \ell \geq k +2\\
\|\omega\|_{\ell } & \hbox{if } \ell \leq k
\end{array}
\right.
=
\|\omega\|_{a(\ell)},
$$
for $\ell \in \N_\Delta$.
Since the amalgamation is simple, the equality comes from the fact that $k+ 1\not\in \N_\Delta$.
\end{itemize}

\epr

\begin{proof}[Proof of \lemref{CompBlow}.] We begin with the compatibility with the differentials.
Let $H= (F_0,\varepsilon_0)  \otimes (F_1,\varepsilon_1) $
and suppose
$\varepsilon_0=1$ and $|(F_1,\varepsilon_1)|\geq 2$. We have
\begin{eqnarray*}
\theta \partial H  & =&
\theta(
(\partial F_0,1) \otimes (F_1,\varepsilon_1) ) 
 +(-1)^{|(F_0,1)|}\theta( (F_0,1) \otimes (\partial F_1,\varepsilon_1))\\
&&+
(-1)^{|(F_0,1) + |F_1|+1}
\left\{
\begin{array}{ll}
0 &\hbox{if } \varepsilon_1=0\\
\theta( (F_0,1) \otimes (F_1,0)) & \hbox{if } \varepsilon_1=1
\end{array}
\right.\\
&=&
(\partial F_0 *  F_1,\varepsilon_1) +(-1)^{|(F_0,1)|}(F_0*\partial F_1,\varepsilon_1) \\
&&+
\left\{
\begin{array}{ll}
0 &\hbox{if } \varepsilon_1=0\\
(-1)^{|(F_0,1) + |F_1|+1}(F_0*F_1,0)  & \hbox{if } \varepsilon_1=1
\end{array}
\right.\\
&=&
 \partial (F_0 *  F_1,\varepsilon_1)  
\end{eqnarray*}
On the other hand, we get $\partial \theta H = \partial (F_0*F_1,\varepsilon_1)$ and we have established the equality $\theta\partial =\partial \theta$ in the case $\varepsilon_0=1$ and $|(F_1,\varepsilon_1)|\geq 2$.
The proof is similar in the other cases. The compatibility of $\Xi$ with the differentials follows by duality.

\medskip

(1) If $E_0 = \emptyset$, then $\epsi_0 =1$ and, by definition of the map $\theta$, we have $\theta((\emptyset,1) \otimes (F_1,\epsi_1)) = (F_1,\epsi_1)$. Therefore, $\theta$ is an isomorphism and so is $\Xi$ by duality.

\medskip

(2) Let $ \1_{(F_0* F_1,\varepsilon)} ,\1_{(F_0'*F_1',\varepsilon')}   \in \Hiru \tN * {\tc (E_0*E_1)}$. We suppose that $(F_1,\varepsilon)$ and $(F'_1,\varepsilon')$ are different from $(\emptyset,0)$. Let us consider $A =\Xi \left(\1_{(F_0* F_1,\varepsilon)} \right) 
\cup \Xi \left(\1_{(F_0'*F_1',\varepsilon')}  \right)$ and 
$B = \Xi \left(\1_{(F_0*F_1,\varepsilon)}  \cup \1_{(F'_0* F'_1,\varepsilon')}  \right)$.
A direct computation gives the following equalities from definitions.
\begin{eqnarray*}
A
 =(-1)^{ |(F_0,1)| \cdot |(F_1,\epsi)|  + |(F'_0,1)| \cdot |(F'_1,\epsi')|  }  \left(\1_{(F_0,1)} \otimes \1_{(F_1,\epsi)}   \right)\cup  \left(\1_{(F'_0,1)} \otimes \1_{(F'_1,\epsi')}\right) = &&\\
 (-1)^{|(F_1,\epsi)| \cdot |(F'_0,1)| + |(F_0,1)| \cdot |(F_1,\epsi)|  + |(F'_0,1)| \cdot |(F'_1,\epsi')|  }  \left(\1_{(F_0,1)} \cup\1_{(F'_0,1)}  \right)\otimes  \left(\1_{(F_1,\epsi)} \cup\1_{(F'_1,\epsi')}\right)&&
\end{eqnarray*}
If $(F'_0,\epsi) =(\emptyset, 0)$  and  $F_1,F'_1$ are compatibles (see \defref{def:cupsurDelta}), we have
\begin{eqnarray*}
A
&=&
(-1)^{|(F_1,0)|  \left( |(F_0,1)|  +  |(F'_1,\epsi')|\right)}  
\1_{(F_0,1)} \otimes\1_{(F_1\cup F'_1,\epsi')}\\
&=&
(-1)^{|(F_0*F_1,0)| \cdot |(F'_1,\epsi')|  } \Xi (\1_{(F_0*F_1 \cup F'_1,\epsilon' ) })
=
B.
\end{eqnarray*}
If  $(F'_0,\epsi) =(\emptyset, 1)$ and $(F'_1,\epsi') = (\emptyset,1)$, we have
$
A
=
(-1)^{|(F_0,1)| \cdot |(F_1,1)|   }  \1_{(F_0,1)} \otimes\1_{(F_1,1)}	=
  \Xi( \1_{(F_0*F_1,1)} ) 	 =B.
$
In the other cases we have $A=B=0$.

We have established the equality $A=B$ in the case where $(F_1,\varepsilon)$ and $(F'_1,\varepsilon')$ are different from $(\emptyset,0)$. The verification is similar in the other cases.

\medskip

(3) 
Suppose  $(F_1,\epsi)\ne (\emptyset,0)$. We have
\begin{eqnarray*}
C&=&\theta
  (\Xi (\1_{(F_0* F_1,\varepsilon)}  ) \tcap ([\tc E_0 ]\otimes [\tc E_1] ))\\
&=&
\mathfrak (-1)^{|(F_0,1)| \cdot |(F_1,\epsi)|} \theta((\1_{(F_0,1)}\otimes \1_{(F_1,\epsi)} )\tcap ( [\tc E_0 ]\otimes [\tc E_1] ))\\&=&
(-1)^{|(F_1,\epsi)| \cdot |\tc E_0| + |(F_0,1)| \cdot |(F_1,\epsi)|} \mathfrak \theta ( (\1_{(F_0,1)}  \cap [\tc E_0]) \otimes (\1_{(F_1,\epsi)} \cap [\tc E_1 ])).
\end{eqnarray*}
We apply the definition of cap product \eqref{equa:cap1} and we consider three different cases.
\begin{itemize}
\item If $F_0=E_0, F_1=E_1$ and $\epsi=1$, we have
$$
C =
\theta( [\tv ]\otimes [\tv] )  =[\tv]  = \1_{(F_0*F_1,\varepsilon) }\cap [\tc(E_0*E_1)].
$$

\item If $F_0=E_0, F_1 \cup G_1 =E_1$ and $\epsi=0$, we have
$$
C=
\theta( [\tv ]\otimes (G_1,1) ) =(G_1,1) =
\1_{(F_0*F_1,\varepsilon) }\cap [\tc(E_0*E_1)].
$$

\item The other cases correspond to $C = \1_{(F_0*F_1,\varepsilon) }\cap [\tc(E_0*E_1) ] =0$.
\end{itemize}

We have established the property (3) in the case $(F_1,\epsi)\ne (\emptyset,0)$. The verification is similar in the other cases.

\end{proof}

\section{Stratified maps: the global level.
\thmref{MorCoho}.}\label{AmalSec}
Let $f \colon X \to Y$ be a stratified map, $\sigma \colon \Delta_\sigma \to X$ and $f \circ \sigma \colon \Delta_{f\circ \sigma} \to Y$ filtered simplices of $X$ and $Y$ respectively. In 
\cite[Corollary A.25]{CST1} we have proved that the filtrations $\Delta_\sigma$ and $\Delta_{f\circ \sigma}$ (of the same euclidean simplex $\Delta$) are connected by an amalgamation, more exactly, by a finite sequence $\cal A_1$   of elementary amalgamations and a 
finite sequence $\cal A_2^{-1}$   of inverse of simple amalgamations as follows.

\begin{center}
\begin{tikzpicture}
\node (a) at (-6,0)  {$\Delta_\sigma $} ;
\node (b) at (-6,-4) {$\Delta_{f\circ \sigma}$} ;
\draw [->] (a) --(b) node[midway,left]{\tiny$\cal A_\sigma$} ;

\draw(-5.5,0) node{$=$} ;

\node (c)  at (-4,-.15) {$\underbrace{\Delta_0 {*} \cdots {*}\Delta_{j_0} }$} ;
\node (d)  at (-4,-2)  {$\Delta'_{i_0}$}  ;
\node (e)  at (-4,-3.85) {$\overbrace{\Delta'_0 {*} \cdots{*} \Delta'_{i_0}}$} ;
\draw [->] (c) --(d) node[midway]{} ;
\draw [->] (d) --(e) node[midway]{} ;

\draw (-2.8,0) node{$*$} ;
\draw (-2.8,-4) node{$*$} ;

\node (f)  at (-1.3,-.15) {$\underbrace{\Delta_{j_0+1} * \cdots *\Delta_{j_1} }$};
\node (g)  at (-1.3,-2)  {$\Delta'_{i_1}$}  ;
\node (h)  at (-1.3,-3.85) {$\overbrace{\Delta'_{i_0+1} * \cdots * \Delta'_{i_1}}$} ;
\draw [->] (f) --(g) node[midway]{} ;
\draw [->] (g) --(h) node[midway]{} ;

\draw (.8,0) node{$* \cdots *$} ;
\draw (.8,-4) node{$* \cdots *$} ;

\node (i)  at (3,-.15) {$\underbrace{ \Delta_{j_{a-1}+1}{*}\cdots {*} \Delta_{j_a}}$};
\node (j)  at (3,-2)  {$\Delta'_{i_a}$}  ;
\node (k)  at (3,-3.85){$\overbrace{ \Delta'_{i_{a-1}+1}{*}\cdots *\Delta'_{i_a} }$} ;
\draw [->] (i) --(j) node[midway]{} ;
\draw [->] (j) --(k) node[midway]{} ;

\node (l)  at (4.5,0){} ;
\node (m)  at (3.2,-1.9)  {}  ;
\node (mm)  at (3.2,-2.1)  {}  ;
\node (n)  at (4.5,-4)  {}  ;
\draw [->] (l)  to[out=0,in=0]  node[midway,right]{\tiny${\cal A}_1$}(m) ;
\draw [->] (mm)  to[out=0,in=0]  node[midway,right]{\tiny${\cal A}^{-1}_2$}(n) ;

\end{tikzpicture}
\end{center}

We denote the amalgamation $\cal A_\sigma = \cal A_2^{-1} \circ \cal A_1$, with $\cal A_1 = \cal A_{1,1} \circ \cdots \circ \cal A_{1,u}$ and $\cal A_2 = \cal A_{2,1} \circ \cdots \circ \cal A_{2,v}$. The elementary amalgamations  $\cal A_{1,i}$ (resp. simple amalgamations $\cal A_{2,j}$)
are written in a canonical way, going from the left to the right.

\subsection{$\mu$-amalgamation}
\label{defmu*}
Mention the amalgamation of $\Delta_\sigma$ collecting all the filtration in one factor. We call it  the \emph{$\mu$-amalgamation} $\cal A_\sigma''\colon \Delta_\sigma \to \Delta$. The induced map  $\cal A''_{\sigma,*} \colon \hiru \tN * {\Delta_\sigma} \to \hiru N * {\Delta}$ is defined by
\begin{eqnarray*}
\bigotimes_{i=0}^n (F_i,\varepsilon_i)
&\rightsquigarrow&
(F_{0} * \cdots *F_\ell,0) \otimes\bigotimes_{i> \ell}(F_i,\varepsilon_i)\rightsquigarrow
\left\{
\begin{array}{cl}
F_{0}\ast\dots\ast F_{\ell}&\text{if } |(F,\epsilon)|_{>\ell} = 0\\
0&\text{otherwise,}
\end{array}\right.
\end{eqnarray*}
with $\epsi_n=0$ and 
$\ell$  is the smallest integer, $j$, such that $\varepsilon_{j}=0$ (see \defref{61}).
We recover the map $\mu_*$ of \eqref{AmuA}. We also define $\mu^* = \cal A_\sigma''^*$.

With the previous notations, we notice the commutativity of the following diagram 
\be\label{PropmuA}
\xymatrix{
\Delta_\sigma 
\ar[r]^-{{\cal A}_\sigma }
 \ar[d]_{\cal A''_\sigma}
&
\Delta_{f\circ \sigma}
 \ar[ld]^-{\cal A''_{f \circ \sigma}}
\\
\Delta &
}
\ee

If $\ov p$ is a perversity on $X$ and $\ov q$ is a perversity on $Y$ with $\ov p \leq f^* \ov q$, we have proved in \cite[Proposition 3.6]{CST3} that the map $f$ induces a homomorphism $f_* \colon \lau H {\ov p} *{X;R} \to \lau H{\ov p} *{Y;R} $.
We define now the induced map in blown-up intersection cohomology and study its properties.

\bd\label{local1}
Let $f \colon X \to Y$ be a stratified map. The \emph{induced map}  $f \colon \Hiru \tN * {Y;R} \to \Hiru \tN * {X;R} $ is the map defined by
$
(f^* \omega)_\sigma = \cal A_\sigma^*\omega_{f \circ \sigma },
$
for each $\omega \in \Hiru \tN * {Y;R}$ and each regular simplex $\sigma \colon \Delta_\sigma \to X$.
\ed

With \propref{CMAmal}(4)  this definition makes sense. It generalizes  the notion of induced map of  \propref{prop:applistratifieeforte} (see \remref{rem:fortementstratifieetidentite}).

\bt \label{MorCoho}
Let $f \colon X \to Y$ be a stratified map. 
\begin{enumerate}[(1)]
\item The induced morphism
$f ^*\colon \Hiru \tN *  {Y;R} \to \Hiru \tN *  {X;R}$
 is a  chain map.
 
 \item The map $f^*$ is compatible with the cup product.
 
 \item The maps $f_*\colon \lau \gC {} * {X;R} \to \lau \gC {} * {Y;R}$ and $f ^*\colon \Hiru \tN *  {Y;R} \to \Hiru \tN *  {X;R}$ verify
 $$f_*(f^* \omega\cap \xi ) = \omega \cap f_* \xi.$$
for each $\omega\in \lau \tN * {} {Y;R}$ and each  $\xi \in \lau \gC {} * {X;R}$.
 
 \item Let  $\ov p$ be a  perversity on $X$ and  $\ov q$  a perversity on $Y$ such that $\ov p \geq f^*\ov q$.
 Then $f$ induces a chain map $f ^*\colon \lau \tN * {\ov q} {Y;R} \to \lau \tN * {\ov p} {X;R}$,
thus a  homomorphism 
$f ^*\colon \lau \IH * {\ov q} {X;R} \to \lau \IH* {\ov p} {Y;R}$.
\end{enumerate}
\et

As a GM-perversity verifies $\ov p\geq f^* \ov p  $, the previous Theorem has an adaptation in this context.

\bc \label{GMp}
Let $f \colon X \to Y$ be a stratified map  and $\ov p$ a GM-perversity.
Then $f$ induces a homomorphism $f ^*\colon \lau \IH * {\ov p} {Y;R} \to \lau \IH* {\ov p} {X;R}$
compatible with cup products.
\ec

The maps induced in cohomology and homology for a GM-perversity satisfy also the condition (3) of \thmref{MorCoho}.

\begin{proof}[Proof of \thmref{MorCoho}]
Properties (1) and (2) are direct consequences of \propref{CMAmal}. 

\medskip

(3) Notice first that $f_*\colon \lau \gC {} * {X;R} \to \lau \gC {} * {Y;R}$ is well defined since $f_*$ preserves filtered simplices (see \cite[Theorem F]{CST1}) and  $f(X\menos \Sigma_X) \subset Y \menos \Sigma_Y$, since $f$ is  stratified.

We can suppose that $\xi$ is a regular simplex $\sigma \colon \Delta_\sigma \to X$. 
Applying \propref{CMAmal}(2) repeatedly we have
$
\cal A_{\sigma,*}(\cal A ^*_\sigma \omega_{f\circ\sigma} \tcap \tDelta_\sigma )
= \omega_{f\circ \sigma} \tcap \tDelta_{f \circ \sigma}.
$
 The result follows from
 \begin{eqnarray*}
 f_*(f^*\omega \cap \sigma ) =
f_* \sigma_*  \mu_{\Delta_\sigma, *} ((f^*\omega )_\sigma \tcap \tDelta_\sigma)
=
f_* \sigma_*  \mu_{\Delta_\sigma, *}
 (
\cal A_\sigma^* \omega_{f\circ \sigma} \tcap \tDelta_\sigma)
&&\\
\stackrel{\eqref{PropmuA}}{=}
f_* \sigma_*  \mu_{\Delta_{f\circ\sigma}, *}\cal A_{\sigma ,*}
 (
\cal A_\sigma^* \omega_{f\circ \sigma} \tcap \tDelta_\sigma)
=
(f \circ \sigma)_* \mu_{\Delta_{f\circ \sigma},*}(
\omega_{f\circ \sigma} \tcap \tDelta_{f \circ \sigma})
&&\\
=\omega \cap  f_* (\sigma).&&
\end{eqnarray*}

\medskip

(4) Let $\omega \in \lau \tN * {\ov q} {Y;R} $. 
Let $\sigma \colon \Delta \to X$ be a regular simplex  and $S$  a singular stratum of $X$ with  $\im \sigma \cap S\ne \emptyset$. We have to prove
$$
\|\omega\|_{S^f} \leq \ov q(S^f)  \Rightarrow \|f^*\omega\|_{S}   \leq \ov p(S)  
$$
Since $\ov p \geq f^* \ov q$, it suffices  to prove
$
\|\cal A_\sigma^*\omega_{f \circ \sigma}\|_{\codim S}  \leq \|\omega_{f \circ \sigma}\|_{\codim S^f}.$
Recall the decomposition,
$
\cal A_\sigma=\cal A^{-1}_{2,v} \circ \dots \circ \cal A^{-1}_{2,1} \circ \cal A_{1,1} \circ \dots \circ \cal A_{1,u},
$
at the beginning of the section and denote by  $a_{1,i}, a_{2,j}$ the associated index maps. We have
$$
\codim S^f = (a^{-1}_{2,v} \circ \dots \circ  a^{-1}_{2,1} \circ
a_{1,1} \circ \dots \circ a_{1,u})(\codim S)
$$
(see \eqref{indexa}). Now, the result follows from \propref{CMAmal} (5).
\end{proof}

  \part{Properties of the Blown-up intersection cohomology}\label{part:mayervietoris}
  
  We present the properties of the blown-up intersection cohomology used in this work. We end this Part by comparing this cohomology with the intersection cohomology obtained from the dual cochains.

\section{$\cal U$-small chains. \thmref{thm:Upetits}.}\label{sec:subdivision}

We establish  a Theorem of $\cal U $-small filtered simplices, where  $ \cal U $ is  an open cover of $ X $
(\thmref{thm:Upetits}).

\medskip
Let $\Delta$ be an euclidean simplex whose vertices are $e_{0},\dots,e_{m}$. 
Let $e_{\{i_{0}\dots i_{p}\}}$ be the barycenter of the simplex  $[e_{i_{0}},\dots,e_{i_{p}}]$ and
$\sub\,\Delta$ the simplicial complex given by the  
barycentric subdivision of  $\Delta$.
The subdivision linear map,
$\sub_{*}\colon N_{*}(\Delta)\to N_{*}(\sub\,\Delta)$, is defined by
$\sub_{*} [e_{i_{0}},\dots,e_{i_{p}}]
=
(-1)^p (\sub_{*}\partial [e_{i_{0}},\dots,e_{i_{p}}])\ast e_{\{i_{0}\dots i_{p}\}}
$.
Any face of $ \sub \, \Delta $ appears at most once in the expression of $ \sub_{*} $ and some faces do not.

\begin{definition}\label{def:simplexcomplet}
An $\ell$-simplex of $\sub\,\Delta$ is a \emph{subdividing simplex}  if it appears
in the  subdivision of an $\ell$-simplex of $\Delta$. 
 \end{definition}
 

From the definition of $ \sub_{*} $ we know that the subdividing simplices of $ \sub \, \Delta $ are
  the simplices of the form
$$\nabla=[e_{i_{0}},e_{\{i_{0}i_{1}\}},e_{\{i_{0}i_{1}i_{2}\}},\dots,e_{\{i_{0}\dots i_{p}\}}].$$
The vertex $e_{i_{0}}$ is the \emph{first vertex} and the vertex $e_{i_{p}}$ is the \emph{last vertex of $\nabla$}. We write $e_{i_{p}}=\der\nabla$.

\medskip
The transpose map of $\sub_{*}$ is denoted by 
$\sub^*\colon \Hiru N *{\sub\,\Delta} \to \Hiru  N *{\Delta}$. It is defined on the subdividing simplices by
$${\sub^*} (\1_{[e_{i_{0}},e_{\{i_{0}i_{1}\}},e_{\{i_{0}i_{1}i_{2}\}},\dots,e_{\{i_{0}\dots i_{p}\}}]})= \1_{[e_{i_{0}},\dots,e_{i_{p}}]},$$
and by 0 on the other simplices.
Since the linear map ${\sub_{*}}$ is compatible with  differentials, then ${\sub^*}$ also is a cochain map.

\medskip

To construct a homotopy between the barycentric subdivision operator and the identity,
we consider the simplicial complex $ K (\Delta) $,
whose simplices  are the joins, $F  \ast  G$, of a simplex $ F $
of $ \Delta $ and of a simplex $ G $ of $ \sub \, \Delta $, such that:
if $F=[e_{i_{0}},\dots,e_{i_{p}}]$
and
$G=[e_{J_{0}},\dots,e_{J_{k}}]$,
then
$\{i_{0},\dots,i_{p}\}\subset J_{\ell}$,
for each $\ell\in\{0,\dots,k\}$.
In addition, simplices $ F $ and $ G $ can be empty, but not simultaneously, making $ \Delta $ and $ \sub \, \Delta $ two
sub-complexes of $ K ( \Delta) $.

\medskip
The homotopy operator, $T\colon \hiru N {*}{\Delta}\to \hiru N {*+1}{K(\Delta)}$, defined by
\\
$T[e_{i_{0}},\dots,e_{i_{p}}]=
(-1)^{p+1}\left(
[e_{i_{0}},\dots,e_{i_{p}}] -T\partial [e_{i_{0}},\dots,e_{i_{p}}]\right)
\ast e_{\{i_{0}\dots i_{p}\}}
$,
verifies
$$(\partial T+T\partial)(F)=F- \sub_{*}(F).$$
In particular, $T$ takes the values, $T(e_{0})=-e_{0}\ast e_{\{0\}}$
and
$T[e_{0},e_{1}]
=
[e_{0}e_{1}]\ast e_{\{01\}} -e_{0}\ast [e_{\{0\}},e_{\{0,1\}}]+e_{1}\ast [e_{\{1\}},e_{\{0,1\}}]$.

\begin{definition}\label{def:facecompleteK}
A \emph{full simplex} of $K(\Delta)$ is a  simplex of the form
$$F\ast G
=
[e_{i_{0}},\dots,e_{i_{p}}]\ast [e_{\{i_{0}\dots i_{p}\}}, e_{\{i_{0}\dots i_{p}i_{p+1}\}},\dots, e_{\{i_{0}\dots i_{p}\dots i_{p+r}\}}].$$
The vertex $e_{i_{p+r}}$ is the  \emph{last vertex} of $F\ast G$ and denoted by
$\der(F\ast G)$.
\end{definition}

Note that if $ F \ast G$ is full and $ p =  0$, then $ G $ is subdividing in the sense of \defref{def:simplexcomplet}.

\smallskip

Endow the simplex $\Delta$  with a filtration,
$\Delta=\Delta_{0}\ast\dots\ast\Delta_{n}$, 
compatible with the order considered on the vertices, cf. (\ref{equa:ordresommets}).
The simplicial complex \emph{$\sub\,\Delta$  is a weighted simplicial complex,} associating to the vertex
$e_{\{i_{0}\dots i_{p}\}}$ the weight $\ell$ if we have
$[e_{i_{0}},\dots,e_{i_{p}}]\fa \Delta_{0}\ast \dots\ast \Delta_{{\ell}}$
and
$[e_{i_{0}},\dots,e_{i_{p}}]\not\!\!\fa \Delta_{0}\ast \dots\ast \Delta_{{\ell}-1}$.
Since the family of the vertices of  $K(\Delta)$ is the union of those of $\Delta$ and those of  $\sub\,\Delta$, the  complex
$K(\Delta)$ also is a weighted simplicial complex.
The chosen canonical basis of
$\tN^*(K(\Delta))$ 
is composed of elements of the form,
$$\1_{(F\ast G,\varepsilon)}=\1_{(F_{0},\varepsilon_{0})}
\otimes\dots\otimes
\1_{(F_{q}\ast G_{q},\varepsilon_{q})}
\otimes\dots\otimes
\1_{(G_{n-1},\varepsilon_{n-1})}
\otimes
\1_{G_{n}},$$
with $F=F_{0}\ast\dots\ast F_{q}\fa \Delta$ and $G=G_{q}\ast\dots\ast G_{n}\fa \sub\,\Delta$.

\begin{remark}\label{rem:subplouf}
The linear operator $ \sub^*$ does not necessarily respect  the filter degree of  simplices.
Consider $ \Delta = [e_ {0}] \ast [e_ {1}, e_ {2}] $.
Previous computations give
$\sub^*(\1_{[e_{1},e_{\{0,1\}}]})=\1_{[e_{1},e_{0}]}$.
Taking into account the filtration defined on $ \sub \, \Delta $,
one has
$[e_{1},e_{\{0,1\}}]=\emptyset\ast [e_{1},e_{\{0,1\}}]$
and
$[e_{1}, e_{0}]=-[e_{0}]\ast [e_{1}]$,
from which one gets $\sub^*(\1_{\emptyset\ast [e_{1},e_{\{0,1\}}]})=-\1_{[e_{0}]\ast [e_{1}]}$.

More generally, let $\Delta=\Delta_{0}\ast\dots\ast\Delta_{n}$ and let $\nabla=\nabla_{0}\ast\dots\ast\nabla_{n}\fa \sub\,\Delta$.
The previous example shows that the use of $\sub^*$ 
for the construction of a cochain morphism,
$$\wsub\colon N^*(\tc\nabla_{0})\otimes \dots\otimes N^*(\tc\nabla_{n-1})\otimes N^*(\nabla_{n})
\to N^*(\tc\Delta_{0})\otimes \dots\otimes N^*(\tc\Delta_{n-1})\otimes N^*(\Delta_{n}),$$
requires a reordering of the vertices according to the filtrations of $\Delta$ and $\sub\,\Delta$. 
To avoid this difficulty
we define below the morphism
 $\wsub_{\Delta}\colon \Hiru \tN*{\sub\,\Delta}\to \Hiru \tN*\Delta$  and the homotopy 
$\tT_{\Delta}\colon \Hiru \tN* {K(\Delta)}\to \Hiru \tN{*-1}\Delta$
by induction.
\end{remark}

The two canonical injections of $\Delta$ and $\sub\,\Delta$ into  $K(\Delta)$ induce the chain complex epimorphisms, 
\begin{itemize}
\item $\iota^*_{\Delta}\colon \Hiru \tN*{K(\Delta)}\to \tN^*(\Delta)$,
defined by the identity on the elements
$\1_{(F,\varepsilon)}$ with $F\fa \Delta$ and 0 elsewhere,
\item 
$\iota^*_{\sub\,\Delta}\colon \tN^*(K(\Delta))\to 
\Hiru \tN*{\sub\,\Delta}$,
defined by the identity on the elements
$\1_{(G,\varepsilon)}$ with $G\fa \sub\,\Delta$ and 0 elsewhere.
\end{itemize}

\bl\label{lem:subethomotopie}
There exist a linear map, 
$\tT_{\Delta}\colon \Hiru \tN*{K(\Delta)}\to 
\Hiru \tN {*-1}\Delta$, 
and a complex cochain morphism, 
$\wsub_{\Delta}\colon \Hiru \tN*{\sub\,\Delta}\to \Hiru\tN *\Delta$, 
verifying
\begin{equation}\label{equa:cequilfaut}
\tT_{\Delta}\circ \tdelta^{K(\Delta)}+\tdelta^{\Delta}\circ \tT_{\Delta}=
\iota^*_{\Delta}-\wsub_{\Delta}\circ\iota^*_{\sub\,\Delta},
\end{equation}
where $\tdelta^{K(\Delta)}$ and $\tdelta^{\Delta}$ are the differentials on
$\tN^*(K(\Delta))$ and $\tN^*(\Delta)$.
\el

\subsection{Construction of $\wsub_{\Delta}$  and $\tT_{\Delta}$}
The homotopy $\tT_{\Delta}$ is constructed by induction at the level of the full blow up's,
$\tN^{\fc,*}(K(\Delta))$ and $\tN^{\fc,*-1}(\Delta)$,
specifying its value on
the elements $\1_{(F\ast G,\varepsilon)}$, with $F\fa \Delta$ and $G\fa \sub\,\Delta$.

We set  $\tT_{\Delta}(\1_{(F\ast G,\varepsilon)})=0$ if $F\ast G$ is not full, in the sense of \defref{def:facecompleteK}. 
Now, let,
$$\1_{(F\ast G,\varepsilon)}=\1_{(F_{0},\varepsilon_{0})}
\otimes\dots\otimes
\1_{(F_{q}\ast G_{q},\varepsilon_{q})}
\otimes\dots\otimes
\1_{(G_{n-1},\varepsilon_{n-1})}
\otimes
\1_{(G_{n},\varepsilon_{n})},$$
with $F\ast G$ full. %
We use a double induction, on the length $n$ of the filtration of $\Delta$
and on the dimension of the component $(G_{n},\varepsilon_{n})$. We set
$\Delta=\nabla\ast \Delta_{n}$.
\begin{enumerate}[a)]
\item {If $\dim (G_{n},\varepsilon_{n})=0$,} 
we distinguish three cases
\begin{enumerate}[i)]
\item If $G_{n}=\emptyset$ and $\varepsilon_{n}=1$, then $F\ast G\fa \nabla$ and we set:
\begin{equation}\label{equa:homotopiedemarragevide}
\tT_{\Delta}(\1_{(F\ast G,\varepsilon)}\otimes \1_{(\emptyset,1)})=
\tT_{\nabla}(\1_{(F\ast G,\varepsilon)})\otimes \1_{(\emptyset,1)}.
\end{equation}
\item If $\dim G_{n}=0$, $\varepsilon_{n}=0$ and $q<n$, then $G_{n}=[e_{J}]$, where $e_{J}$ is the barycenter of a face of $\Delta$.
Let  $e_{\alpha}=\der (G_{q}\ast\dots\ast G_{n})$  the last vertex of the face $G$
and $G'=G_{q}\ast\dots\ast G_{n-1}$. 
(Notice 
$e_{\alpha}\in\Delta_{n}$, 	otherwise the weight of the vertex $e_{J}$ of $\sub\,\Delta$ would be $n-1$.)
We set:
\begin{equation}\label{equa:homotopiedim0}
\tT_{\Delta}(\1_{(F\ast G',\varepsilon)}\otimes \1_{([e_{J}],0)})=
\tT_{\nabla}(\1_{(F\ast G',\varepsilon)})
\otimes
\1_{([e_{\alpha}],0)}.
\end{equation}
\item If $\dim G_{n}=0$, $\varepsilon_{n}=0$ and $q=n$, then  $G=G_{n}=[e_{\{F\}}]$ is reduced to the barycenter of the face $F$. 
(Notice    $F_{n}\neq \emptyset$
otherwise the weight of the barycenter of $F$ could not be  $n$ in $\sub\,\Delta$.) We set:
\begin{equation}\label{equa:homotopiedemarrage}
\tT_{\Delta}(\1_{(F,\varepsilon)}\ast e_{\{F\}})=
(-1)^{|(F,\varepsilon)|}
\1_{(F,\varepsilon)}.
\end{equation}
\end{enumerate}
\item If $\dim (G_{n},\varepsilon_{n})\geq 1$, then $G_{n}\neq\emptyset$ and we set $e_{\alpha}=\der (G_{q}\ast\dots\ast G_{n})$ the last vertex of the face $G$.
We write $G_{n}=G'_{n}\ast e_{J}$, so that
$G'=G_{q}\ast \dots\ast
G_{n-1}\ast G'_{n}$ is subdividing and $e_{\alpha}$ is a vertex of the face $J$ for which  $e_{J}$ is the barycenter.
We set:
\begin{equation}\label{equa:homotopiedimqcq}
\tT_{\Delta}(\1_{(F\ast G',\varepsilon)}\ast e_{J})=
\tT_{\Delta}(\1_{(F\ast G',\varepsilon)})\ast e_{\alpha}.
\end{equation}
\end{enumerate}

The construction  of $\tT_{\Delta}$ does not change the value of $\varepsilon_{n}$; moreover, the image of an element
$\1_{(F\ast G,\varepsilon)}$ with $\varepsilon_{n}=0$ and $G_{n}\neq\emptyset$ also has a non-empty component in $\Delta_{n}$. So we built a linear map
$$
\tT_{\Delta}\colon \tN^*(K(\Delta))\to \tN^{*-1}(\Delta).
$$
We construct the map
$\wsub_{\Delta}\colon \tN^*(\sub\,\Delta)\to \tN^*(\Delta)$ from $\tT_{\Delta}$, by 
$$\wsub_{\Delta} 
= -\tT_{\Delta}\circ \tdelta^{K(\Delta)}.$$ 

\begin{remark}\label{rem:pleinetsommets}
To any simplex, $F\ast G$, of $K(\Delta)$, we associate a family of vertices, $\cal V_{\Delta}(F\ast G)$, consisting of the vertices of  $F$ and the vertices of the faces of $\Delta$   having a vertex of $G$ as barycenter.
Adding the  virtual vertices, we set
$\cal V_{\Delta}(\1_{(F\ast G,\varepsilon)})=
\cal V_{\Delta}(F\ast G,\varepsilon)$.
By construction of $\tT_{\Delta}$, we have
$$\cal V_{\Delta}(\1_{(F\ast G,\varepsilon)})=\cal V_{\Delta}(\tT_{\Delta}(\1_{(F\ast G,\varepsilon)})).$$
\end{remark}

\bigskip

The proof of \lemref{lem:subethomotopie}  is postponed to the end of the proof of \lemref{lem:subethomotopie},
 \pagref{subsec:constructionhomotopie}.
First, we take advantage of this result 
in terms of the blown-up  intersection cohomology.
%


Let $(X,\ov{p})$ be a perverse space.
Recall that a cochain $\omega\in \lau \tN* {\ov{p}}{X;R}$ associates to any regular filtered simplex
$\sigma\colon \Delta\to X$,
a cochain $\omega_{\sigma}\in \tres \tN*\sigma$ and verifies
$
\|\omega\|\leq \ov{p}$,
$\|\delta\omega\|\leq\ov{p}$
 and $
\delta^*_{\ell}(\omega_{\sigma})=\omega_{\partial_{\ell}\sigma},
$ 
for any regular face operator, $\delta_{\ell}$.

\begin{definition}\label{def:Upetit}
Let $\cal U$ be an open cover of $X$.
A \emph{$\cal U$-small simplex} is a regular simplex, $\sigma\colon \Delta=\Delta_{0}\ast\dots\ast\Delta_{n}\to X$,
 such that there exists $U\in\cal U$ with $\im\sigma\subset U$. The family of  $\cal U$-small simplices is denoted by $\si_{\cal U}$.

 The \emph{blown-up  complex of $\cal U$-small cochains of $X$
 with coefficients in $R$,} written $\Hiru \tN{*,\cal U}{X;R}$,
is the cochain complex made up of elements  $\omega$, associating to any $\cal U$-small simplex,
 $\sigma\colon\Delta= \Delta_{0}\ast\dots\ast\Delta_{n}\to X$,
an element
 $\omega_{\sigma}\in \Hiru \tN*\Delta$, %
 so that $\delta_{\ell}^*(\omega_{\sigma})=\omega_{\partial_{\ell}\sigma}$,
for any face operator,
 $\delta_{\ell}\colon \Delta'_{0}\ast\dots\ast\Delta'_{n}\to \Delta_{0}\ast\dots\ast\Delta_{n}$, with $\Delta'_{n}\neq\emptyset$. 
 If $\ov{p}$ is a perversity on $X$, we denote by $\lau \tN {*,\cal U}{\ov{p}}{X;R}$ the complex of 
 $\cal U$-small cochains verifying
$ \|\omega\|\leq \ov{p}$ and
$\|\delta \omega\|\leq\ov{p}$.
\end{definition}

The following Theorem  compares the complexes $\lau \tN * {\ov{p}}{X;R}$ and
$\lau \tN {*,\cal U} {\ov{p}}{X;R}$.

\begin{theorem}\label{thm:Upetits}
Let $(X,\ov{p})$ be a perverse space endowed with an open cover, $\cal U$.
Let
$\rho_{\cal U}\colon \lau \tN {*}{\ov{p}}{X;R}\to  \lau \tN {*,\cal U}{\ov{p}}{X;R}$ be the restriction map. 
The following properties are verified.
\begin{enumerate}[\rm (i)]
\item  There exist a cochain map,
$\varphi_{\cal U}\colon  \lau \tN {*,\cal U} {\ov{p}}{X;R}\to \lau \tN *{\ov{p}}{X;R}$, and a homotopy, \\
$\Theta\colon  \lau \tN*{\ov{p}}{X;R}\to\lau \tN{*-1}{\ov{p}}{X;R}$,
such that
$$\rho_{\cal U}\circ \varphi_{\cal U}=\id
\text{ and }
\delta\circ\Theta+\Theta\circ\delta=\id -\varphi_{\cal U}\circ\rho_{\cal U}.
$$
\item Furthermore, if the cochain $\omega \in \tN^{*,\cal U}_{\ov{p}}(X;R)$ 
is such that there exists a subset $K\subset X$
for which  $\omega_{\sigma}=0$ if $(\im\sigma)\cap K=\emptyset$,
then $\varphi_{\cal U}(\omega)$ also verifies
$(\varphi_{\cal U}(\omega))_{\sigma}=0$
if $(\im\sigma)\cap K=\emptyset$.
\end{enumerate}
\end{theorem}

	As an immediate consequence of property (i) we get the following corollary.

\begin{corollary}\label{cor:Upetits}
The restriction map,
$\rho_{\cal U}\colon \lau \tN{*}{\ov{p}}{X;R}\to \lau\tN{*,\cal U}{\ov{p}}{X;R}$,
is a quasi-isomor\-phism.
\end{corollary}

First we establish the existence of a subdivision at the level of the blown-up  complex.

\begin{proposition}\label{prop:wsubX}
Consider $X$ a filtered space. The maps $\tT_{\Delta}$ and $\wsub_{\Delta}$
extend as
$\tT\colon \Hiru \tN*{X;R}\to
\Hiru\tN {*-1}{X;R}$
and as
$\wsub\colon\Hiru \tN *{X;R}\to\Hiru \tN *{X;R}$
verifying
\begin{equation}\label{equa:homotopiesub}
\tT\circ \delta+\delta\circ \tT=\id -\wsub.
\end{equation}
\end{proposition}

\begin{proof}
We first construct a homotopy,
$\tT\colon \Hiru \tN*{X;R}\to \Hiru \tN{*-1}{X;R}$,
from the maps $\tT_{\Delta}$ previously defined.

Let $\sigma\colon\Delta\to X$ be a filtered simplex and
$\omega\in \Hiru \tN k {X;R}$. 
A simplex $F\ast G$  of $K(\Delta)$ 
can be described by its vertices,
$F\ast G=[e_{i_{0}},\dots,e_{i_{r}},e_{\{F_{0}\}},\dots,e_{\{F_{s}\}}]$,
where $e_{j}$ is a vertex of $\Delta$ and $e_{\{F_{j}\}}\in\cal V(\sub\,\Delta)$ is the barycenter of the face $F_{j}\fa \Delta$.
It follows  $e_{\{F_{j}\}}\in \Delta$ and we define a filtered simplex of $X$,
$\sigma_{_{F\ast G}}\colon F\ast G\to X$,
using the barycentric coordinates by
\be\label{equa:sigmaFG}
\sigma_{_{F\ast G}}(\sum_{j}t_{j}e_{j}+\sum_{\ell}u_{\ell}e_{\{F_{\ell}\}})
=
\sigma(\sum_{j}t_{j}e_{j}+\sum_{\ell}u_{\ell}e_{\{F_{\ell}\}}).
\ee
Thanks to \propref{prop:pasdeface}, we define
$\omega_{K(\Delta)}\in \Hiru\tN k{K(\Delta)}$ 
by
$$\omega_{K(\Delta)}=
\sum_{\substack{F\ast G\fa K(\Delta) \\  |(F\ast G,\varepsilon)|=k\phantom{-}}}
\omega_{\sigma_{_{F\ast G}}}(F\ast G,\varepsilon) \,\1_{(F\ast G,\varepsilon)}.
$$
Write now
$$(\tT(\omega))_{\sigma}=\tT_{\Delta}(\omega_{K(\Delta)})$$
and verify that we obtain  an element of $\Hiru \tN*{X;R}$.
Let $\sigma\colon\Delta=\Delta_{0}\ast\dots\ast\Delta_{n}\to X$ be  a regular simplex and 
$\tau=\sigma\circ\delta_{\ell}\colon\nabla\to X$,
where $\delta_{\ell}\colon \nabla\to\Delta$
is a regular face operator.
Since
$\omega_{K(\Delta)}\in \tN^k(K(\Delta))$
then
$\delta_{\ell}^*\,\omega_{K(\Delta)}=\omega_{K(\nabla)}$.
The maps
${\delta}^*_{\ell}\colon \tN^*(\Delta)\to\tN^*(\nabla)$
and
${\delta}^*_{\ell}\colon \tN^*(K(\Delta))\to \tN^*(K(\nabla))$
are defined by
$$
{\delta}^*_{\ell}(\1_{(F,\varepsilon)})=
\left\{\begin{array}{cl}
\1_{(F,\varepsilon)}
&\text{if}\;
F\fa\nabla,\\
0&\text{if not,}
\end{array}\right.
\text{and } 
{\delta}^*_{\ell}(\1_{(F*G,\varepsilon)})=
\left\{\begin{array}{cl}
\1_{(F*G,\varepsilon)}
&\text{if}\;
F*G\fa K(\nabla),\\
0&\text{if not.}
\end{array}\right.
$$
From these calculations and from 
\remref{rem:pleinetsommets}, directly follows the commutativity of the following diagram,
\begin{equation*}\label{equa:biendefini}
\xymatrix{
\Hiru \tN *{K(\Delta)}\ar[rr]^-{\tT_{\Delta}}
\ar[d]_{{\delta}^*_{\ell}}
&&
\Hiru \tN {*-1}{\Delta}
\ar[d]^{{\delta}^*_{\ell}}
\\
\Hiru \tN*{K(\nabla)}
\ar[rr]^-{\tT_{\nabla}}
&&
\Hiru \tN {*-1}{\nabla},
}
\end{equation*}
from which we deduce,
$
\delta^*_{\ell}\tT(\omega)_{\sigma}
=\delta^*_{\ell}\tT_{\Delta}(\omega_{K(\Delta)})=
\tT_{\nabla}(\delta^*_{\ell}\omega_{K(\Delta)})=
\tT_{\nabla}(\omega_{K(\nabla)})
=
\tT(\omega)_{\tau}.
$
Thus the map $\tT\colon \Hiru \tN * {X;R}\to
\Hiru \tN {*-1} {X;R}$ is well defined.
 Now we proceed to the construction of
$\wsub\colon \Hiru \tN*{X;R}\to \Hiru \tN* {X;R}$ from the operators $\wsub_{\Delta}$ of \lemref{lem:subethomotopie}. 
Let $\sigma\colon\Delta\to X$ be a regular simplex and let $\omega\in\Hiru\tN k {X;R}$.
If $G$ is a face of $\sub\,\Delta$, we denote by $\iota_{G}\colon G\to \Delta$ the canonical injection and we define (cf. \propref{prop:pasdeface})
\begin{equation*}\label{equa:leSub}
\omega_{\sub\,\Delta}=
\sum_{\substack{G\fa \sub\,\Delta \\ |(G,\varepsilon)|=k\phantom{-}}}
\omega_{\sigma\circ\iota_{G}}(G,\varepsilon)\1_{(G,\varepsilon)}
\quad\text{and}\quad
(\wsub \omega)_{\sigma}=
\wsub_{\Delta}(\omega_{\sub\,\Delta}).
\end{equation*}
Notice 
\begin{eqnarray*}
\iota^*_{\sub\,\Delta}\left(\omega_{K(\Delta)}\right)
&=&
\sum_{\substack{F\ast G\fa  K(\Delta) \\ | (F\ast G,\varepsilon)|=k\phantom{-}}}
\omega_{\sigma_{_{F\ast G}}}(F\ast G,\varepsilon)\,
\iota^*_{\sub\,\Delta}\left(\1_{(F\ast G,\varepsilon)}\right)
\\
&=&
\sum_{\substack{G\fa \sub\,\Delta \\ |(G,\varepsilon)|=k\phantom{-}}}
\omega_{\sigma\circ\iota_{G}}(G,\varepsilon)\1_{(G,\varepsilon)}
= \omega_{\sub\,\Delta}.
\end{eqnarray*}
We also have
$\iota^*_{\Delta}\left(\omega_{K(\Delta)}\right)
=
\omega_{\sigma}
$. 
It follows, 	by using  \lemref{lem:subethomotopie},
\begin{eqnarray*}
\omega_{\sigma}- (\tT\delta\omega)_{\sigma}-(\delta\tT\omega)_{\sigma}
&=&
\iota^*_{\Delta}\left(\omega_{K(\Delta)}\right)
-\tT_{\Delta}\left(\tdelta^{K(\Delta)}\omega_{K(\Delta)}\right)
-\tdelta^{\Delta}\tT_{\Delta}\left(\omega_{K(\Delta)}\right)\\
&=&
\wsub_{\Delta}\left(\iota^*_{\sub\,\Delta}(\omega_{K(\Delta)})\right)
=
\wsub_{\Delta}\left(\omega_{\sub\,\Delta}\right)
=
\left(\wsub(\omega)\right)_{\sigma}
\end{eqnarray*}
We have shown $\wsub(\omega)\in \Hiru \tN * {X;R}$ and the equality
$\tT\circ \delta+\delta\circ\tT=\id-\wsub$.
\end{proof}

\begin{proof}[Proof of  \thmref{thm:Upetits}]
(i) %
We prove
$\|\tT(\omega)\|\leq \|\omega\|$
and $\|\wsub(\omega)\|\leq \|\omega\|$,
by working on the chosen basis and the
various possible cases.
Making reference to \defref{def:degrepervers}, since maps $\tT$ and $\wsub$ 
do not modify the value of the parameter
$\varepsilon$ then we can limit ourselves to the case  $\varepsilon_{n-\ell}=0$, with $\ell\in \{1,\dots,n\}$. 
We proceed by induction on the complex $\Hiru \tN {\fc}
\Delta$, the case $n=0$ being obvious.
With the notations used in the construction of $ \tT$, the following properties
are verified.
\begin{itemize}
\item (\ref{equa:homotopiedemarragevide}): if $G_{n}=\emptyset$ and $\varepsilon_{n}=1$, the desired inequality comes from
\begin{eqnarray*}
\|\tT_{\Delta}(\1_{(F\ast G,\varepsilon)})\otimes \1_{(\emptyset,1)}\|_{\ell}
&=&
\left\{
\begin{array}{cl}
0&\text{if }\ell=1,\\
\|\tT_{\Delta}(\1_{(F\ast G,\varepsilon)})\|_{\ell-1}
&\text{if } \ell>1,
\end{array}\right.\\
&\leq&
\left\{
\begin{array}{cl}
0&\text{if }\ell=1,\\
\|\1_{(F\ast G,\varepsilon)}\|_{\ell-1}
&\text{if } \ell>1,
\end{array}\right.\leq
\|\1_{(F\ast G,\varepsilon)}\otimes\1_{(\emptyset,1)}\|_{\ell}.
\end{eqnarray*}
\item The justifications of  (\ref{equa:homotopiedim0}) and (\ref{equa:homotopiedemarrage})
are similar.
\item (\ref{equa:homotopiedimqcq}): if  $\dim (G_{n},\varepsilon_{n})\geq 1$, 
we denote by  $\ell(\alpha)$ the weight of the generator $e_{\alpha}$ of $\Delta$.
By construction of $\tT_{\Delta}$, we have
$$
\|\tT_{\Delta}(\1_{(F\ast G',\varepsilon)})\ast e_{\alpha}\|_{\ell}
=
\left\{\begin{array}{cl}
\|\tT_{\Delta}(\1_{(F\ast G',\varepsilon)})\|_{\ell}
\text{if } \ell(\alpha)\leq n-\ell,\\[.2cm]
\|\tT_{\Delta}(\1_{(F\ast G',\varepsilon)})\|_{\ell}+1
\text{if } \ell(\alpha)>n-\ell,
\end{array}\right.
\leq
\|\1_{(F\ast G,\varepsilon)}\|_{\ell}.
$$
\end{itemize}
To study the perverse degree of $\wsub_{\Delta}(-)$, 	consider a  subdividing face $G\fa \sub\,\Delta$  with first term $ e_ {0} $.
By definition of  $\wsub_{\Delta}$, one has, for each $\ell\in\{1,\dots,n\}$,
$$
\|\wsub_{\Delta}\1_{(G,\varepsilon)}\|_{\ell}
=
\|\tT_{\Delta}\tdelta^{K(\Delta)}\1_{(G,\varepsilon)}\|_{\ell}
=\|\tT_{\Delta}(\1_{(G,\varepsilon)}\ast e_{0})\|_{\ell}
\leq \|\1_{(G,\varepsilon)}\ast e_{0}\|_{\ell}
=\|\1_{(G,\varepsilon)}\|_{\ell}.
$$
(The last equality uses the fact that $e_{0}$, being the first vertex of $ G $, lies in the smallest filtration of $ G $.)
Together with \propref{prop:wsubX}, this proves that  $\tT$ and $\wsub$ define two maps from the complex $\lau\tN*{\ov{p}}{X;R}$ to itself. 

\medskip
 We construct a cochain map, $\varphi_{\cal U}\colon \lau \tN{*,\cal U}{\ov{p}}{X;R}\to \lau \tN*{\ov{p}}{X;R}$, 
and a homotopy,
 $\Theta\colon \lau \tN*{\ov{p}}{X;R}\to\lau\tN{*-1}{\ov{p}}{X;R}$, such that
 $\rho_{\cal U}\circ\varphi_{\cal U}=\id$ and $\delta\circ\Theta+\Theta\circ\delta=\id-\varphi_{\cal U}\circ\rho_{\cal U}$.
 
  Write $\Psi_{1}=\tT$ and, for $m\geq 2$, 
 $\wsub^m=\wsub^{m-1}\circ\wsub$ and
 $\Psi_{m}=\sum_{0\leq i<m}\tT\circ \wsub^i$.
A simple calculation using (\ref{equa:homotopiesub}) gives
 \begin{equation}\label{equa:psim}
 \delta\circ \Psi_{m}+\Psi_{m}\circ\delta=\sum_{0\leq i <m}\delta\circ\tT\circ\wsub^i+\tT\circ\wsub^i\circ \delta
 =\id-\wsub^m.
 \end{equation}
 
  If $\sigma\colon \Delta\to X$ is a regular simplex, we denote by $m(\sigma)$ 
the smallest integer such that the simplices of 
 $\sub^{m(\sigma)}\,\Delta$
have their image by $\sigma$ included in an open subset of the cover~$\cal U$. 
If $F$ is a simplex of $\Delta$, we denote by $\sigma_{F}\colon F\to X$ the restriction of $\sigma$
and by $(F,\varepsilon)\fa \tDelta$
the fact that  $(F,\varepsilon)$ is a face of the blow up $\tDelta$. 
Let $\omega\in \Hiru \tN k {X;R}$. The homotopy $\Theta$ is defined by
$$\left(\Theta(\omega)\right)_{\sigma}=
\sum_{\substack{(F,\varepsilon)\fa \tDelta \\  |(F,\varepsilon)|=k-1\phantom{-}}}
\left(\Psi_{m(\sigma_{F})}(\omega)\right)_{\sigma_{F}}\!\!\!\!(F,\varepsilon)\,\1_{(F,\varepsilon)}.
$$
Since the coefficient of $\1_{(F,\varepsilon)}$  depends only on $\sigma_{F}$ and
  $ \varepsilon $ then we get the compatibility with the faces. In order to construct $\varphi_{\cal U}$, we calculate $\delta\circ\Theta$ and $\Theta\circ\delta$.
Following (\ref{equa:ledelta1}), we have
  \begin{equation}\label{equa:lecobord}
  \delta {\1_{(F,\varepsilon)}} =(-1)^k
\sum_{(F,\varepsilon)\fa \partial (\nabla,\kappa)}n_{(F,\varepsilon,\nabla,\kappa)}\1_{(\nabla,\kappa)},
\end{equation}
where
$(F,\varepsilon)\fa\partial(\nabla,\kappa)$
means that $(F,\varepsilon)$ runs over the proper faces of $(\nabla,\kappa)$ and
$$\partial(\nabla,\kappa)=\sum_{(F,\varepsilon)\fa\partial(\nabla,\kappa)}
n_{(F,\varepsilon,\nabla,\kappa)}(F,\varepsilon).$$
We deduce, for all $\omega\in \Hiru \tN k {X;R}$,
$$
\left( \delta\Theta(\omega)\right)_{\sigma}
=
(-1)^k \! \!  \! \! \! \! 
\sum_{\substack{(F,\varepsilon)\fa \tDelta \\  |(F,\varepsilon)|=k-1\phantom{-}}} 
\sum_{\substack{(F,\varepsilon)\fa \partial (\nabla,\kappa)  \\  (\nabla,\kappa)\fa \tDelta\phantom{-}}}
n_{(F,\varepsilon,\nabla,\kappa)}
\left(\Psi_{m(\sigma_{F})}(\omega)\right)_{\sigma_{F}}(F,\varepsilon)\,\1_{(\nabla,\kappa)}.
$$
On the other hand, the definition of $\Theta$ implies
  \begin{eqnarray*}
 \left(\Theta(\delta\omega)\right)_{\sigma}
 &=&
 \sum_{\substack{(\nabla,\kappa)\fa \tDelta \\  |(\nabla,\kappa)|=k\phantom{-}}}
 \left(\Psi_{m(\sigma_{\nabla})}(\delta\omega)\right)_{\sigma_{\nabla}}
 (\nabla,\kappa)\,
 \1_{(\nabla,\kappa)}
 \\
 &=_{(\ref{equa:psim})}&
 \omega_{\sigma}
 -\sum_{\substack{(\nabla,\kappa)\fa \tDelta \\  |(\nabla,\kappa)|=k\phantom{-}}}
 \left( \wsub^{m(\sigma_{\nabla})}(\omega)\right)_{\sigma_{\nabla}}(\nabla,\kappa)\,\1_{(\nabla,\kappa)}
\\
&&
-\sum_{\substack{(\nabla,\kappa)\fa \tDelta \\  |(\nabla,\kappa)|=k\phantom{-}}}
\left(\delta\Psi_{m(\sigma_{\nabla})}(\omega)\right)_{\sigma_{\nabla}}(\nabla,\kappa)
\,\1_{(\nabla,\kappa)}.
 \end{eqnarray*}
It follows, again by using (\ref{equa:lecobord}),
 \begin{eqnarray*}\label{equa:levarphi}
 \left((\delta\Theta+\Theta\delta-\id)(\omega)\right)_{\sigma}
 &=&
 -(\varphi_{\cal U}(\omega))_{\sigma}
\end{eqnarray*}
with
\begin{eqnarray*}
-(\varphi_{\cal U}(\omega))_{\sigma}
=
 -\sum_{\substack{(\nabla,\kappa)\fa \tDelta  \\  |(\nabla,\kappa)|=k\phantom{-}}}
 \left( \wsub^{m(\sigma_{\nabla})}(\omega)\right)_{\sigma_{\nabla}}(\nabla,\kappa)\,\1_{(\nabla,\kappa)} + &&
 \label{equa:varphi}\\
 (-1)^k 
\sum_{\substack{(F,\varepsilon)\fa \tDelta  \\  |(F,\varepsilon)|=k-1\phantom{-}}}
\sum_{\substack{(F,\varepsilon)\fa \partial (\nabla,\kappa)  \\  (\nabla,\kappa)\fa \tDelta\phantom{-}}}
n_{(F,\varepsilon,\nabla,\kappa)}
\left((\Psi_{m(\sigma_{F})}-\Psi_{m(\sigma_{\nabla})})(\omega)\right)_{\sigma_{F}}(F,\varepsilon)\,\1_{(\nabla,\kappa)}.
&&\nonumber
\end{eqnarray*}
Observe that $(\varphi_{\cal U}(\omega))_{\sigma}$
is well defined for all $\omega\in \Hiru \tN k {X;R}$ but we have to show that it is also well defined
for every  $\omega\in \Hiru \tN{*,\cal U}{X;R}$ and any regular simplex $\sigma\colon\Delta\to X$.
The first term of the sum above is defined since the cochains $(\wsub^m(\omega))_{\sigma_{\nabla}}$
are well defined for any $m\geq m(\sigma_{\nabla})$.
The second term is also well defined since $m(\sigma_{F})\leq m(\sigma_{\nabla})$
and the cochain $(\wsub^m(\omega))_{\sigma_{F}}$ is defined for any $m\geq m(\sigma_{F})$.
	So we have $\varphi_{\cal U}(\omega)\in \Hiru \tN * {X;R}$.

\smallskip

The equality
 $\delta\circ\Theta+\Theta\circ\delta=\id-\varphi_{\cal U}\circ\rho_{\cal U}$
is thus verified by construction of $\varphi_{\cal U}$ and $\Theta$.
 If $\omega\in \Hiru \tN{*,\cal U}{X;R}$ and if $\sigma$ is $\cal U$-small, then $m(\sigma)=0$
 and the family of indexes of the terms defining  $\Psi_{m(\sigma)}$ is empty. Then
 $$\left(\rho_{\cal U}(\varphi_{\cal U}(\omega))\right)_{\sigma}=
 \varphi_{\cal U}(\omega)_{\sigma}
 =\omega_{\sigma}-0=\omega_{\sigma}.$$
We have therefore established
 $\rho_{\cal U}\circ \varphi_{\cal U}=\id$.
 From $\id-\varphi_{\cal U}=\delta\circ\Theta+\Theta\circ\delta$, 	we deduce
 $\delta-\delta\circ\varphi_{\cal U}
=\delta\circ\Theta\circ\delta=
\delta-\varphi_{\cal U}\circ\delta$,
and the compatibility of  $\varphi_{\cal U}$ with differentials.
It remains to study the behavior of $\Theta$ and $\varphi_{\cal U}$ relatively to the perverse degree.
Since $\|\Psi_{m}(\omega)\|\leq \|\omega\|$
and
$\|\wsub(\omega)\|\leq \|\omega\|$,
we have, for any singular stratum, $S$,
$\|\Theta(\omega)\|_{S}\leq \|\omega\|_{S}
\quad\text{and}\quad
\|\varphi_{\cal U}(\omega)\|_{S}\leq \|\omega\|_{S}$.
This completes the proof of the assertion (i).

\medskip
(ii) Let $\omega \in \lau \tN{*,\cal U}{\ov{p}}{X;R}$ verifying the hypothesis of (ii) for a subset $K\subset X$.
Let $\sigma\colon \Delta\to X$ such that $(\im\sigma)\cap K=\emptyset$.
The element
$(\varphi_{\cal U}(\omega))_{\sigma}$
is defined from cochains $\omega_{\sigma_{F\ast G}}$, where
$\sigma_{F\ast G}$ is the restriction of $\sigma$, as defined in
(\ref{equa:sigmaFG}). From this formula, we find that
$\im \sigma_{F\ast G}\subset \im\sigma$,
and then $\omega_{\sigma_{F\ast G}}=0$ and
$(\varphi_{\cal U}(\omega))_{\sigma}=0$.
\end{proof}

\begin{proof}[Proof of \lemref{lem:subethomotopie}]\label{subsec:constructionhomotopie}
We have to show that
$
\tT_{\Delta}\colon\Hiru  \tN *{K(\Delta)}\to \Hiru \tN {*-1}\Delta
$ verifies (\ref{equa:cequilfaut}). For it,
we use
\propref{prop:dfaceajout} to express
the value of the differential $\delta$ in function of the adding of vertices.

(i) Start with $\1_{(F,\varepsilon)}\in  \Hiru \tN *{\Delta}$. By definition, the face $F$ is not a full face (see \defref{def:facecompleteK}) on $K(\Delta)$,
then we have $\tT_{\Delta}(\1_{(F,\varepsilon)})=0$. On the other hand,
the only term of $\tdelta^{K(\Delta)} \1_{(F,\varepsilon)}$ having a non-zero image 
by $\tT_{\Delta}$  is
$\1_{(F,\varepsilon)} \ast e_{\{F\}}$ where $e_{\{F\}}$is the barycenter of the face $F$.
Therefore one has, using (\ref{equa:homotopiedemarrage}),
$$
\tT_{\Delta}(\tdelta^{K(\Delta)} \1_{(F,\varepsilon)})
=
\tT_{\Delta}\left(
(-1)^{|(F,\varepsilon)|}
\1_{(F,\varepsilon)}\ast e_{\{F\}}\right)
=
(-1)^{|(F,\varepsilon)|}(-1)^{|(F,\varepsilon)|}\1_{(F,\varepsilon)}
= \1_{(F,\varepsilon)}.
$$
It follows:
$$(\tT_{\Delta}\circ \tdelta^{K(\Delta)}+\tdelta^{\Delta} \circ \tT_{\Delta})(\1_{(F,\varepsilon)})=
(\iota^*_{\Delta}-\wsub_{\Delta}\circ\iota^*_{\sub\,\Delta})(\1_{(F,\varepsilon)}).$$

\medskip 
(ii) Continue with a simplex, $F\ast G\fa K(\Delta)$, such that $F\neq \emptyset$ and $G\neq \emptyset$.
In this case, we have to show
\begin{equation}\label{equa:pointiii}
(\tT_{\Delta}\circ \tdelta^{K(\Delta)}+\tdelta^{\Delta}\circ \tT_{\Delta})(\1_{(F\ast G,\varepsilon)})=0.
\end{equation}

--- We look first to the case of a \emph{full simplex} distinguishing the various possibilities that appear in the construction of $ \tT_{\Delta} $, located after the statement of \lemref{lem:subethomotopie}.
We use a first recurrence assuming equality (\ref{equa:pointiii})  holds for any filtered euclidean simplex with formal dimension strictly less than $ n $.

$\bullet$ If $G_{n}=\emptyset$ and $\varepsilon_{n}=1$, we have $F\ast G\fa K(\nabla)$, then:
\begin{eqnarray*}
&&
\tT_{\Delta}\left(\tdelta^{K(\Delta)}(\1_{(F\ast G,\varepsilon)}\otimes \1_{(\emptyset,1)})\right) =_{(1)}\\
&&
\tT_{\nabla}\left(\tdelta^{K(\nabla)}\1_{(F\ast G,\varepsilon)}\right)\otimes \1_{(\emptyset,1)}\\&&
+
(-1)^{|(F\ast G,\varepsilon)|}
\sum_{e\in\cal V(K(\Delta)_{n})}
\tT_{\Delta}\left(\1_{(F\ast G,\varepsilon)}\otimes \1_{(\emptyset,1)}\ast e\right)=_{(2)}\\
&&
\tT_{\nabla}\left(\tdelta^{K(\nabla)}\1_{(F\ast G,\varepsilon)}\right)\otimes \1_{(\emptyset,1)}\\&&
+
(-1)^{|(F\ast G,\varepsilon)|}
\sum_{e\in\cal V(\Delta_{n})}
\tT_{\nabla}\left(\1_{(F\ast G,\varepsilon)}\right)\otimes \1_{(\emptyset,1)}\ast e=_{(3)}\\
&&
-\tdelta^{\nabla}\tT_{\nabla}\left(\1_{(F\ast G,\varepsilon)}\right)\otimes \1_{(\emptyset,1)}
+
(-1)^{|(F\ast G,\varepsilon)|}
\tT_{\nabla}\left(\1_{(F\ast G,\varepsilon)}\right)\otimes \tdelta^{c\Delta_{n}}\1_{(\emptyset,1)}
=_{(1)}\\
&&
-\tdelta^{\Delta}
\tT_{\Delta}\left(\1_{(F\ast G,\varepsilon)}\otimes \1_{(\emptyset,1)})\right),
\end{eqnarray*}
where $=_{(1)}$ uses (\ref{equa:homotopiedemarragevide}), $=_{(2)}$ uses (\ref{equa:homotopiedimqcq}) and
(\ref{equa:homotopiedemarragevide}),
$=_{(3)}$  is the induction hypothesis  on $\nabla$.

\smallskip
$\bullet$ 
The argument is similar when $ \dim G_{n} =  0$ and $ \varepsilon_ {n} = 0$.
A second induction on the dimension of $ G_{n} $ completes the proof in the case of a full simplex $ F \ast G $.

\medskip

--- If $ F \ast G $ is not full and if the differential $ \tdelta^{K(\Delta)}\1_{(F\ast G,\varepsilon)}$ only involves 
non full simplices, then the left hand  side of (\ref{equa:pointiii})  is zero and the result is true.
So we have to
consider the case of a non full simplex whose differential involves full simplices.
Specifically, consider a  full $ k $-simplex $ F '\ast G' $ with
$F'=[e_{i_{0}},\dots,e_{i_{a}}]$
and
$G'=[e_{\{i_{0}\dots i_{a}\}},\dots,e_{\{i_{0}\dots i_{a} \dots i_{b}\}}]$
and including a non full $ (k-1) $-face $ F \ast G $ with $ F \neq \emptyset$ and $ G \neq \emptyset$.

We need to establish
\begin{equation}\label{equa:pasplein}
\tT_{\Delta}\tdelta^{K(\Delta)}\left(\1_{(F\ast G,\varepsilon)}\right)=0.
\end{equation}
Consider the various possible cases.
\begin{enumerate}[(a)]
\item Suppose $a=b$, then $F=[e_{i_{0}},\dots,\hat{e}_{i_{x}},\dots,e_{i_{a}}]$ and $G=G'=[e_{\{F'\}}]$. The equality
(\ref{equa:pasplein}) is reduced to
$\tT_{\Delta}\tdelta^{K(\Delta)}\left(\1_{(F,\varepsilon)}\ast e_{\{F'\}}\right)=0.$
By keeping in the expression of the differential only the  elements corresponding to full simplices, we have:
$$
\tT_{\Delta}\tdelta^{K(\Delta)}\left(\1_{(F,\varepsilon)}\ast e_{\{F'\}}\right)=
(-1)^{(F,\varepsilon)+1}
\tT_{\Delta}\left(
\1_{(F,\varepsilon)}\ast e_{\{F'\}}\ast e_{i_{x}}+
\1_{(F,\varepsilon)}\ast e_{\{F'\}}\ast 
e_{\{F\}}
\right)
$$
The definition of $\tT_{\Delta}$ gives,
\begin{eqnarray*}
\tT_{\Delta}\left(
\1_{(F,\varepsilon)}\ast e_{\{F'\}}\ast e_{i_{x}}\right)
=
-\tT_{\Delta}\left(
\1_{(F,\varepsilon)}\ast e_{i_{x}} \ast e_{\{F'\}}\right)
=
-(-1)^{|(F',\varepsilon)|}\,\1_{(F,\varepsilon)}\ast e_{i_{x}},
&&
\\
\tT_{\Delta}\left(\1_{(F,\varepsilon)}\ast e_{\{F'\}}\ast 
e_{\{F\}}
\right)
=
-\tT_{\Delta}\left(\1_{(F,\varepsilon)}\ast 
e_{\{F\}}
\ast e_{\{F'\}}\right)
&&\\
=
-\tT_{\Delta}\left(\1_{(F,\varepsilon)}\ast 
e_{\{F\}}\right)\ast e_{i_{x}}
=
-(-1)^{|(F,\varepsilon)|}\1_{(F,\varepsilon)}\ast e_{i_{x}}.
&&
\end{eqnarray*}

The conclusion comes from $|(F,\varepsilon)|+1=|(F',\varepsilon)|$.

\smallskip\noindent 
\emph{For the last two cases, we use an induction on the dimension of the component $G_{n}$.}
\item Suppose $a<b$ and $F=[e_{i_{0}},\dots,\hat{e}_{i_{x}},\dots,e_{i_{a}}]$, then 
$G=G'=L\ast e_{\{i_{0}\dots i_{a} \dots i_{b}\}}$.
The equality
(\ref{equa:pasplein}) is reduced to
$\tT_{\Delta}\tdelta^{K(\Delta)}\left(\1_{(F\ast L,\varepsilon)}\ast e_{\{i_{0}\dots i_{a} \dots i_{b}\}}\right)=0.$
By definition of $\tT_{\Delta}$ and following \corref{cor:astdiff}, one has:
$$\tT_{\Delta}\tdelta^{K(\Delta)}\left(\1_{(F\ast L,\varepsilon)}\ast e_{\{i_{0}\dots i_{a} \dots i_{b}\}}\right)
=
\tT_{\Delta}\left(\tdelta^{K(\Delta)}\1_{(F\ast L,\varepsilon)}\right)\ast e_{i_{b}}.
$$
The simplex $F\ast L$ is not full and $L\neq\emptyset$ since $a<b$. We can apply the induction hypothesis, so
$\tT_{\Delta}\tdelta^{K(\Delta)}\left(\1_{(F\ast L,\varepsilon)}\right)=0$.

\item The argument is similar when
$G=[e_{\{i_{0}\dots i_{a}\}},\dots,\hat{e}_{\{i_{0}\dots i_{a}\dots i_{x}\}},\dots,e_{\{i_{0}\dots i_{a} \dots i_{b}\}}]$
and $a<b$,
and thus $F=F'$.
\end{enumerate}

\medskip
(iii) It remains to verify the equality (\ref{equa:cequilfaut}) for elements
$\1_{(G,\varepsilon)}\in\tN^*(\sub\,\Delta)$. 
In this case, it is actually the \emph{definition} of $\wsub$ by
$$\wsub \1_{(G,\varepsilon)}=-\left(\tT_{\Delta}\circ \tdelta^{K(\Delta)}\right)(\1_{(G,\varepsilon)}),$$
the other terms being zero.

\medskip
To complete the proof of \lemref{lem:subethomotopie}, it remains to verify the compatibility of
$\wsub$ with differentials,
\begin{equation}\label{equa:subdiff}
\tdelta^{\Delta}\circ \tT_{\Delta}\circ \tdelta^{K(\Delta)}(\1_{(G,\varepsilon)})
=
\tT_{\Delta}\circ \tdelta^{K(\Delta)}\circ \tdelta^{\sub\,\Delta}(\1_{(G,\varepsilon)}).
\end{equation}
The two terms of this equation are zero except in the two following situations.
\begin{enumerate}[(i)]
\item $G=[e_{i_{0}},\dots,e_{\{i_{0}\dots i_{a}\}}]$,
\item $G=[e_{i_{0}},\dots,\hat{e}_{\{i_{0}\dots i_{x}\}},\dots,e_{\{i_{0}\dots i_{x}\dots i_{a}\}}]$ with $x<a$.
\end{enumerate}
Detail each case.

(i) The term $\tdelta^{K(\Delta)}(\1_{(G,\varepsilon)})$ has only one full term,  
$(-1)^{|(G,\varepsilon)|}\1_{(G,\varepsilon)}\ast e_{i_{0}}$.
The term $\tdelta^{K(\Delta)}\tdelta^{\sub\,\Delta}(\1_{(G,\varepsilon)})$ has the following full terms
$-\sum_{e_{i_{j}}\in\cal V(\Delta)}\1_{(G,\varepsilon)}\ast e_{\{i_{0}\dots i_{a}i_{j}\}}\ast e_{i_{0}}$. We can deduce,
\begin{eqnarray*}
\tT_{\Delta}\tdelta^{K(\Delta)}\tdelta^{\sub\,\Delta}(\1_{(G,\varepsilon)})
&=&
-\sum_{e_{i_{j}}\in\cal V(\Delta)}\tT_{\Delta}(\1_{(G,\varepsilon)}\ast e_{\{i_{0}\dots i_{a}i_{j}\}}\ast e_{i_{0}})\\
&=&
\sum_{e_{i_{j}}\in\cal V(\Delta)}\tT_{\Delta}(\1_{(G,\varepsilon)}\ast e_{i_{0}})\ast e_{i_{j}}
\\
&=& (-1)^{|(G,\varepsilon)|}\tdelta^{\Delta}\tT_{\Delta}(\1_{(G,\varepsilon)}\ast e_{i_{0}})=
\tdelta^{\Delta}\tT_{\Delta}\tdelta^{K(\Delta)}(\1_{(G,\varepsilon)}).
\end{eqnarray*}

(ii) The term $\tdelta^{K(\Delta)}\1_{(G,\varepsilon)}$ does not have any full term. Also, the next term
$\tdelta^{K(\Delta)}\tdelta^{\sub\,\Delta}\1_{(G,\varepsilon)}$ has the following full terms
$-\1_{(G,\varepsilon)}\ast e_{\{i_{0}\dots i_{x}\}}\ast e_{i_{0}}
-\1_{(G,\varepsilon)}\ast e_{\{i_{0}\dots i_{x-1}i_{x+1}\}}\ast e_{i_{0}}$.
As in the previous calculation, equation (\ref{equa:subdiff}) follows from the definition of $\tT_{\Delta}$:
$$\tT_{\Delta}(\1_{(G,\varepsilon)}\ast e_{\{i_{0}\dots i_{x}\}}\ast e_{i_{0}})=-\tT_{\Delta}(\1_{(G,\varepsilon)}\ast e_{\{i_{0}\dots i_{x-1}i_{x+1}\}}\ast e_{i_{0}}).$$
(The sign comes from the permutation of $e_{i_{x}}$ and $e_{i_{x+1}}$.)
\end{proof}
%

\section{Mayer-Vietoris exact sequence. \thmref{thm:MVcourte}.%
}\label{sec:MV}

If $ U \subset $ V are two open subsets of a perverse space $ (X, \ov {p}) $, canonical inclusions $ U \subset V \subset X $
induce cochain maps,
$\lau \tN* {\ov{p}}{X;R}\to \lau \tN* {\ov{p}}{V;R}\to \lau \tN *{\ov{p}}{U;R}$. 
If there is no ambiguity, we keep the same notation for a cochain and
its images by these maps.

\begin{theorem}[Mayer-Vietoris exact sequence]\label{thm:MVcourte}
Let $(X,\ov{p})$ be a paracompact perverse space, 
endowed with an open cover  $(U_{1},U_{2})$
and a subordinated partition of the unity, $(f_{1},f_{2})$. 
For $i=1,\,2$, we denote by $\cal U_{i}$ the cover of  $U_{i}$ consisting of the open subsets
$(U_{1}\cap U_{2}, f_{{i}}^{-1}(]1/2,1])$ and by $\cal U$ the cover of  $X$, union of the covers  $\cal U_{i}$.
Then, the canonical inclusions, $U_{i}\subset X$ and $U_{1}\cap U_{2}\subset U_{i}$, induce a short exact sequence ,
$$
0\to
\lau \tN {*,\cal U} {\ov{p}}{X;R}
\stackrel{\iota}{\longrightarrow}
\lau \tN {*,\cal U_{1}} {\ov{p}}{U_{1};R}
\oplus
\lau \tN {*,\cal U_{2}} {\ov{p}}{U_{2};R}
\stackrel{\varphi}{\longrightarrow}
\lau \tN * {\ov{p}}{U_{1}\cap U_{2};R}
\to
0,
$$
where $\varphi(\omega_{1},\omega_{2})=\omega_{1}-\omega_{2}$.
\end{theorem}

The following result is a direct consequence of Theorems \ref{thm:Upetits} and \ref{thm:MVcourte}.

\begin{corollary}\label{cor:MVlongue} 
Let $ (X, \ov p) $  be a paracompact perverse space provided with an open cover $ (U_ {1}, U_ {2}) $.
Then there is a long exact sequence for the blown-up  intersection cohomology,
$$
\to
\lau \IH i{\ov{p}}{X;R}
\to
\lau  \IH i{\ov{p}}{U_{1};R}\oplus \lau \IH i{\ov{p}}{U_{2};R}
\to
\lau \IH i{\ov{p}}{U_{1}\cap U_{2};R}
\to
\lau \IH{i+1}{\ov{p}}{X;R}\to
$$

\end{corollary}

\begin{proof}[Proof of  \thmref{thm:MVcourte}]
Since the open cover $  \cal U$ is the union  of the covers $ \cal U_ {1} $ and $ \cal U_ {2} $, the morphism $ \iota $ is injective.

We prove the surjectivity of the map $\varphi$.
For $i=1,\,2$, we define the fonction, $g_{i}\colon X\to \{0,1\}$, by
$$\begin{array}{ccc}
g_{1}(x)=\left\{
\begin{array}{cl}
1&\text{if } f_{1}(x)>1/2,\\
0&\text{if not,}
\end{array}\right.
&
\text{ and }
&
g_{2}(x)=\left\{
\begin{array}{cl}
1&\text{if } f_{2}(x)\geq 1/2,\\
0&\text{if not.}
\end{array}\right.
\end{array}
$$
The inequality $f_{1}(x)>1/2$ implies $f_{2}(x)<1/2$ and $g_{2}(x)=0$. The support of $g_{2}$ is therefore included in
$U_{2}\backslash f_{1}^{-1}(]1/2,1])$. It is noted as well that the support of $g_{1}$ is included
in $U_{1}\backslash f_{2}^{-1}(]1/2,1])$.
On the other hand, by construction, one has $g_{1}(x)+g_{2}(x)=1$. We denote by $\tg_{1}$ and $\tg_{2}$ the
two  0-cochains of  perverse degree 0, respectively associated  to 
 $g_{1}$ and $g_{2}$, and defined in
 \lemref{lem:0cochaine}.
If $\omega\in \lau \tN * {\ov{p}}{U_{1}\cap U_{2};R}$, we denote by $\tg_{i}\cup \omega$ the cup product 
(cf. \secref{subsec:cupTW}) of $\tg_{i}$ with $\omega$, for $i=1,\,2$.
Since the cochain $\tg_{1}$ has a  support included in $U_{1}\backslash f_{2}^{-1}(]1/2,1])$, then the cup product 
$\tg_{1}\cup \omega\in \lau \tN{*,\cal U_{1}}{\ov{p}}{U_{1};R}$. Likewise, one has
$\tg_{2}\cup \omega\in \lau \tN {*,\cal U_{2}} {\ov{p}}{U_{2};R}$. We verify
$\varphi(\tg_{1}\cup \omega, -\tg_{2}\cup \omega)=\omega$,
which gives the surjectivity of $\varphi$.

The composition $\varphi\circ\iota$ is the zero map. 
It remains to consider an element
$(\omega_{1},\omega_{2})\in \lau \tN {*,\cal U_{1}} {\ov{p}}{U_{1};R}
\oplus
\lau \tN {*,\cal U_{2}} {\ov{p}}{U_{2};R}$ such that $\varphi(\omega_{1},\omega_{2})=0$
and to construct $\omega\in \lau \tN{*,\cal U} {\ov{p}}{X;R}$ 
such that $\iota(\omega)=(\omega_{1},\omega_{2})$.
If $\sigma\colon \Delta\to X$  is a regular simplex, we set,
\begin{itemize}
\item $\omega_{\sigma}=(\omega_{1})_{\sigma}=(\omega_{2})_{\sigma}$, if $\sigma(\Delta)\subset U_{1}\cap U_{2}$, 
\item $\omega_{\sigma}=(\omega_{1})_{\sigma}$, if $\sigma(\Delta)\subset f_{1}^{-1}(]1/2,1])$, 
\item $\omega_{\sigma}=(\omega_{2})_{\sigma}$, if $\sigma(\Delta)\subset f_{2}^{-1}(]1/2,1])$.
\end{itemize}
This definition makes sense because, on the one hand  $f_{1}^{-1}(]1/2,1])\cap f_{2}^{-1}(]1/2,1])=\emptyset$
and on the other hand 
$U_{1}\cap U_{2}\cap f_{1}^{-1}(]1/2,1])\subset U_{1}\cap U_{2}$, where the two cochains
$(\omega_{1})_{\sigma}$ and $(\omega_{2})_{\sigma}$ coincide.
\end{proof}

With the notation of the previous proof, the connecting morphism of the long exact sequence of Mayer-Vietoris is defined by
$$
[\omega]\mapsto [(\delta \tilde{g}_{1})\cup \omega].$$

\begin{lemma}\label{lem:0cochaine}
Let  $(X,\ov{p})$ be a perverse space. Any map, $g\colon X\to R$, defines a 0-cochain
$\tg\in \lau \tN 0 {\ov{0}}{X;R}$.
Moreover, the association $g\mapsto \tilde{g}$ is $R$-linear.
\end{lemma}

\begin{proof}
Let $\sigma\colon \Delta=\Delta_{0}\ast\dots\ast\Delta_{n}\to X$ be a regular simplex. We want to define
$\tg_{\sigma}\in N^0(c\Delta_{0})\otimes\dots\otimes N^0(c\Delta_{n-1})\otimes N^0(\Delta_{n})$.
If $b=(b_{0},\dots, b_{n})\in c\Delta_{0}\times\dots\times c\Delta_{n-1}
\times \Delta_{n}$, we denote by  $i_{0}$ the smallest index for which $ b_i $ is not a apex of a cone
(i.e., $i_{0}=\min\{i \mid b_{i}\in \Delta_{i}\}$).
Observe that the integer $i_{0}$ exists since $b_{n}\in \Delta_{n}$ and set $\tg_{\sigma}(b)=g(\sigma(b_{i_{0}}))$. 
This map is clearly compatible with the face  operators and defines $\tg\in \Hiru \tN 0{X;R}$.

It remains to determine the perversity of the cochain $ \tg_{\sigma} $.
Since it is a 0-cochain, it is obviously $\ov{0}$-allowable 
and we only need to study
its coboundary. 
To compute the $ \ell $-perverse degree of $ \delta \tg_{\sigma} $, we
consider
$F=F_{0}\otimes\dots\otimes F_{n}=\{b_{0}\}\otimes\dots\otimes F_{j}\otimes\dots\otimes \{b_{n}\}$ with
$\dim F_{j}=1$ and $\partial F_{j}=b_{j}^1-b_{j}^0$.
We have,
$$\delta\tg_{\sigma}(F)=
\tg_{\sigma}(b_{0},\dots, b^1_{j},\dots, b_{n})
-\tg_{\sigma}(b_{0},\dots, b^0_{j},\dots, b_{n}).
$$
If $n-\ell<j$, then $i_{0}\leq n-\ell<j$ and $\delta\tg_{\sigma}(F)=g(\sigma(b_{i_{0}}))-g(\sigma(b_{i_{0}}))=0$,
and then  $\|\delta \tilde{g}_{\sigma}\|_{\ell}=-\infty$.\\
Finally, $\tg\in \lau \tN 0 {\ov{0}}X$ and the association $g\mapsto \tilde{g}$   is $ R $-linear by construction.
\end{proof}

\section{Product with the real line. \thmref{prop:isoproduitR}.}\label{sec:produitR}

Consider the product $ X \times \R $ equipped with the product filtration $(X\times \R)_i = X_i \times \R
$ and a perversity~$\ov{p} $.
We also denote by $ \ov {p} $ the perversity induced on $ X $, that is, $\ov p (S) = \ov p (S \times \R)$ for each stratum $S$ of $X$.
Let $ I_{0} \ I_{1} \colon X \to X \times \R $ be the canonical injections, defined 
by $ I_ {0} (x) = (x, 0) $ and $ I_ {1} (x) = (x, 1) $. The canonical projections are denoted by
$ \pr \colon X \times \R \to X $ and $ \pr_ {2} \colon X \times \R \to \R $.
\propref{prop:applistratifieeforte}  ensures that the  maps $I_{0} $, $ I_ {1} $ and $ \pr $ induce cochain maps
 between the blown-up  complexes.

\bt\label{prop:isoproduitR}
Let $ X $ be a filtered space and let $ X \times \R $ equipped with the filtration product and a perversity~$\ov {p} $.
Denoting also $ \ov {p} $ the perversity induced on $ X $,
the maps induced in the blown-up  intersection cohomology by
the projection $ \pr \colon X \times \R \to X $ and the canonical injections, $ I_{0} \ I_{1} \colon X \to X \times \R $ ,
verify
$(I_{0}\circ \pr)^*=(I_{1}\circ\pr)^*=\id$.
\et

The proof proceeds according to the scheme of \secref{sec:subdivision} by constructing a
homotopy $ \Theta_{\Delta} $ at the simplex level (cf. \propref{prop:homotopieR}),
then gluing them to get a homotopy at the level of blown-up  complexes (see  \propref{prop:homotopieRglobal}).

For any filtered simplex,
$\Delta=\Delta_{0}\ast\dots\ast\Delta_{n}$,
one defines a weighted simplicial complex
$\cal L_{\Delta}=\Delta\otimes [0,1]$, whose simplices are the joins
$(F,\pmb{0})\ast (G,\pmb{1})$, with $F\fa \Delta$, $G\fa \Delta$, $(F,\pmb{0})\subset \Delta\times \{0\}$
and
$(G,\pmb{1})\subset \Delta\times\{1\}$.
Henceforth, we denote by $F  \ast  G$ such simplex, meaning that the first term, $ F $,
is identified with a simplex  of $ \Delta \times \{0 \} $ and the second, $ G $,
with a simplex of $ \Delta \times \{1 \}$.

If $ F $ and $ G $ are compatible, in the sense of \defref{def:cupsurDelta}, then, for the filtration induced by $ \Delta $, we have
$F\ast G=F_{0}\ast\dots\ast(F_{p}\ast G_{p})\ast\dots\ast G_{n}$, with $p\in\{0,\dots,n\}$.
A face of the prismatic set $ \tcL_{\Delta} $, corresponding to compatible simplices, $ F $ and $ G $ of $ \Delta $,
is denoted by,
$$(F\ast G,\varepsilon)=(F_{0},\varepsilon_{0})\times\dots\times (F_{p}\ast G_{p},\varepsilon_{p})
\times\dots\times (G_{n-1},\varepsilon_{n-1})\times G_{n}.$$
In this writing, if $ j <n $, $ j \neq  p$, one authorizes $ F_{j} $ and $ G_{j} $ to be the emptyset  but the associated variable $ \varepsilon_{j}$   must be 1.

\begin{proposition}\label{prop:homotopieR}
There exists a linear map,
$\Theta_{\Delta}\colon  \Hiru \tN *{\Delta\otimes [0,1]}\to \Hiru \tN {*-1}{\Delta}$,
such that
\begin{equation}\label{equa:homotopieR}
(\Theta_{\Delta}\circ \delta+\delta\circ \Theta_{\Delta})(\1_{(F\ast G,\varepsilon)})
=\left\{\begin{array}{cl}
0&
\text{if } F\neq\emptyset \text{ and } G\neq\emptyset,\\
-  \1_{(G,\varepsilon)}&
\text{if } F=\emptyset,\\
\1_{(F,\varepsilon)}&
\text{if } G=\emptyset,
\end{array}\right.
\end{equation}
where $\delta$ 
denotes the differential of
$\Hiru \tN* {\Delta\otimes [0,1]}$
and
$\Hiru \tN {*-1}{\Delta}$.%
\end{proposition}

\begin{proof}
Define the map $\Theta_{\Delta}$ on the elements of the dual basis.
Let
$$\1_{(F\ast G,\varepsilon)}=
\1_{(F_{0},\varepsilon_{0})}\otimes\dots\otimes
\1_{(F_{p}\ast G_{p},\varepsilon_{p})}\otimes\dots\otimes
\1_{(G_{n-1},\varepsilon_{n-1})}\otimes \1_{G_{n}}.$$
We set
$$\Theta_{\Delta}(\1_{(F \ast G,\varepsilon)})=
\left\{\begin{array}{cl}
(-1)^{|(F,\varepsilon)|_{<p}+|(G_{p},\varepsilon_{p})|}\,
\1_{(F\cup G,\varepsilon)}&
\text{if}\; F\; \text{and}\; G \; \text{are compatible,}\\
0&
\text{if not,}
\end{array}\right.$$
where $\1_{(F\cup G,\varepsilon)}=
\1_{(F_{0},\varepsilon_{0})}\otimes\dots\otimes
\1_{(F_{p}\cup G_{p},\varepsilon_{p})}\otimes\dots\otimes
\1_{(G_{n-1},\varepsilon_{n-1})}\otimes \1_{G_{n}}
$ and $F_{p}\cup G_{p}$ was introduced in
 \defref{def:cupsurDelta}.

In $\1_{(F\ast G,\varepsilon)}$, we consider the faces $F$ of $\Delta\times \{0\}$ and $G$ of $\Delta\times\{1\}$. 
In the expression of
$\Theta_{\Delta}(\1_{(F \ast G,\varepsilon)})$,  we make an abuse of notation by keeping the letters
$ F $ and $ G $ for  faces of $ \Delta $.
It remains to verify (\ref{equa:homotopieR})  and for that, we consider the following cases.
We denote by
$\cal V(\tcL_{\Delta})$ the union of $\cal V(\cal L_{\Delta})$ with the family of virtual vertices.
\begin{itemize}
\item Suppose $F\neq \emptyset$, $G\neq\emptyset$, $F$ and $G$ compatible.
Using equality (\ref{equa:diffetpoint}) and \lemref{lem:thetaaste}, we obtain the equality
(\ref{equa:homotopieR}) in this case:
\begin{eqnarray*}
\delta\Theta_{\Delta}(\1_{(F\ast G,\varepsilon)})
&=&
(-1)^{|(F\ast G,\varepsilon)|+1}
\sum_{e\in\cal V(\tDelta)}\Theta_{\Delta}(\1_{(F\ast G,\varepsilon)}) \ast e
\\
&=&
(-1)^{|(F\ast G,\varepsilon)|+1}
\sum_{e\in\cal V(\tcL_{\Delta})}\Theta_{\Delta}(\1_{(F\ast G,\varepsilon)} \ast e)
=
-\Theta_{\Delta}(\delta\1_{(F\ast G,\varepsilon)}).\nonumber
\end{eqnarray*}
\item Suppose $F\neq\emptyset$, $G\neq\emptyset$, $F$ and $G$ not compatible such that 
$\delta \1_{(F\ast G,\varepsilon)}$ contains compatible elements.
This amounts to giving a $k$-simplex, $F'\ast G'$, of $\Delta\otimes [0,1]$, with $F'$ and $G'$ compatible and
having a $(k-1)$-face, $F\ast G$, with $F$ and $G$ non compatible and non empty.
Put $F'=[a_{i_{0}},\dots,a_{i_{r}}]$ and $G'=[b_{j_{0}},\dots,b_{j_{s}}]$ with $a_{i_{r}}=b_{j_{0}}\in\Delta_{p}$. 
(Recall that in the latter equality we identified
$\Delta\cong\Delta\times\{{0}\}\cong \Delta\times \{{1}\}$.)
We denote by $a_{j_{t}}\in\Delta\times\{{0}\}$ the vertex identified to $b_{j_{t}}\in\Delta\times\{{1}\}$, as well as $b_{i_{t}}\in\Delta\times\{{1}\}$ for the vertex identified to $a_{i_{t}}\in\Delta\times\{{0}\}$.
 For the simplices $ F $ and $ G $, there are only two possibilities corresponding to compatible simplices $ F '$, $ G' $.
 Equality (\ref{equa:homotopieR}) is deduced from the following calculations.

$+$ If $F=[a_{i_{0}},\dots,a_{i_{r-1}}]$ and $G=G'$, the expression $\Theta_{\Delta}(\delta\1_{(F\ast G,\varepsilon)})$
takes the values,
$$
(-1)^{|(G,\varepsilon)|+1+|(F,\varepsilon)|_{<p}+\varepsilon_{p}}
\Theta_{\Delta}(\1_{(F\ast a_{i_{r}}\ast G,\varepsilon)}+\1_{(F\ast b_{i_{r-1}}\ast G,\varepsilon)}),
$$
if   $a_{i_{r-1}}\in F_{p}$, and 
$$
(-1)^{|(F,\varepsilon)|+|(G_{p},\varepsilon_{p})+1}
\Theta_{\Delta}(\1_{(F\ast a_{i_{r}}\ast G,\varepsilon)})
+
(-1)^{|(F,\varepsilon)|_{<a}+\varepsilon_{a}}
\Theta_{\Delta}(\1_{(F\ast b_{i_{r-1}}\ast G,\varepsilon)}),
$$
if  $a_{i_{r}-1}\in F_{a}$ with $a<p$,
which implies
$$
\Theta_{\Delta}(\delta\1_{(F\ast G,\varepsilon)})=
-\1_{((F\ast a_{i_{r}})\cup G,\varepsilon)}+
\1_{(F\cup (b_{i_{r-1}}\ast G),\varepsilon)}=0.
$$

$+$ The argument is the same for $F=F'$ and
$G=[b_{j_{1}},\dots,b_{j_{s}}]$.
\item Suppose $F=\emptyset$. Let  $b_{i_{0}}\in\Delta_{p}\times \{{1}\}$ be the first vertex of $G$ and
$a_{i_{0}}$  the vertex of $\Delta_{p}\times \{{0}\}$ 
whose projection on $\Delta_{p}$ is equal to the projection of
$b_{i_{0}}$.
We have
$\Theta_{\Delta}(\1_{(G,\varepsilon)})=0$.
To determine the second term, first note,
\begin{eqnarray*}
\1_{(G,\varepsilon)}\ast a_{i_{0}}
&=&
(-1)^{|(G,\varepsilon)|_{>p}}
(\1_{(G_{p},\varepsilon_{p})}\ast a_{i_{0}})\otimes\dots\\
&=&
(-1)^{|(G,\varepsilon)|_{>p}+1}
 (\1_{(a_{i_{0}}\ast G_{p},\varepsilon_{p})})\otimes\dots
\end{eqnarray*}
It follows
\begin{eqnarray*}
\Theta_{\Delta}\delta\1_{(G,\varepsilon)}
&=&
(-1)^{|(G,\varepsilon)|}
\Theta_{\Delta}(\1_{(G,\varepsilon)}\ast a_{i_{0}})\\
&=&
(-1)^{|(G_{p},\varepsilon_{p})|+1}
\Theta_{\Delta}(\1_{(a_{i_{0}}\ast G,\varepsilon)})
=
-
\,\1_{(G,\varepsilon)}.
\end{eqnarray*}
\item The proof for the case $G=\emptyset$ is similar.
\end{itemize}
\end{proof}

\begin{lemma}\label{lem:thetaaste}
Let
$\1_{(F\ast G,\varepsilon)}
\in \Hiru \tN *{\Delta\otimes [0,1]}$,
with $F\neq\emptyset$ and $G\neq\emptyset$.
\begin{enumerate}[1)]
\item Let $\ell\in\{1,\dots,n\}$. For any vertex $e\in\Delta_{\ell}$, we denote
$e_{0}\in \Delta_{\ell}\times\{{0}\}$
and
$e_{1}\in \Delta_{\ell}\times\{{1}\}$
the vertices of $\Delta\otimes [0,1]$ corresponding to $e$. Then we have
$$
\Theta_{\Delta}(\1_{(F \ast G,\varepsilon)})\ast e
=
\Theta_{\Delta}(\1_{(F \ast G,\varepsilon)}\ast e_{0})
+
\Theta_{\Delta}(\1_{(F \ast G,\varepsilon)}\ast e_{1}).
$$
\item For any virtual vertex $\tv_{\ell}$, $\ell<n$, we have
\begin{equation}\label{equa:thetavirtuel}
\Theta_{\Delta}(\1_{(F \ast G,\varepsilon)}\ast \tv_{\ell})=
(\Theta_{\Delta}(\1_{(F \ast G,\varepsilon)}))\ast \tv_{\ell}.
\end{equation}
\end{enumerate}
\end{lemma}

\begin{proof}
Let $\1_{(F\ast G,\varepsilon)}=
\1_{(F_{0},\varepsilon_{0})}\otimes\dots\otimes
\1_{(F_{p}\ast G_{p},\varepsilon_{p})}\otimes\dots\otimes
\1_{(G_{n-1},\varepsilon_{n-1})}\otimes \1_{G_{n}}$.

1) $\bullet$ Let $\ell<p$. Using 
the \defref{def:etunpointun}, we get
$\Theta_{\Delta}(\1_{(F \ast G,\varepsilon)}\ast e_{1})=0$
and the following equalities
\begin{eqnarray*}
\Theta_{\Delta}(\1_{(F \ast G,\varepsilon)}\ast e_{0})=&&\\
(-1)^{|(F\ast G,\varepsilon)|_{>\ell}}
\Theta_{\Delta} (
\1_{(F_{0},\varepsilon_{0})}
\otimes\dots\otimes
(\1_{(F_{\ell},\varepsilon_{\ell})}\ast e_{0})
\otimes
\dots\otimes
\1_{(F_{p}\ast G_{p},\varepsilon_{p})}
\otimes\dots\otimes
\1_{G_{n}}
) =&&\\
(-1)^{|(F\ast G,\varepsilon)|_{>\ell}+
|(F,\varepsilon)|_{<p}+1+|G_{p}|+\varepsilon_{p}}
\dots\otimes
(\1_{(F_{\ell},\varepsilon_{\ell})}\ast e)
\otimes
\dots\otimes
\1_{(F_{p}\cup G_{p},\varepsilon_{p})}
\otimes\dots =&&
\\
(-1)^{|(F,\varepsilon)|_{<p}+|G_{p}|+\varepsilon_{p}}
(
\1_{(F_{0},\varepsilon_{0})}
\otimes
\dots\otimes
\1_{(F_{p}\cup G_{p},\varepsilon_{p})}
\otimes\dots\otimes
\1_{G_{n}}
)\ast e
=
(\Theta_{\Delta}(\1_{(F \ast G,\varepsilon)}))\ast e.&&
\end{eqnarray*}

\medskip\noindent
$\bullet$ The argument is the same if $\ell>p$.

\medskip\noindent
$\bullet$ If $\ell=p$, these formulas become:
\begin{eqnarray*}
\Theta_{\Delta}(\1_{(F \ast G,\varepsilon)}\ast e_{0})
=
(-1)^{|(F\ast G,\varepsilon)|_{>p}}
\Theta_{\Delta}(
\1_{(F_{0},\varepsilon_{0})}
\otimes\dots\otimes
\1_{(F_{p}\ast G_{p},\varepsilon_{p})}\ast e_{0}
\otimes\dots\otimes
\1_{G_{n}}
).&&\\
=
(-1)^{|(F\ast G,\varepsilon)|_{\geq p}+\varepsilon_{p}}
\Theta_{\Delta}(
\1_{(F_{0},\varepsilon_{0})}
\otimes\dots\otimes
\1_{(F_{p}\ast G_{p}\ast e_{0},\varepsilon_{p})}
\otimes\dots\otimes
\1_{G_{n}}
).&&
\end{eqnarray*}
Write $F_{p}=[a_{i_{0}},\dots,a_{i_{r}}]$, $G_{p}=[b_{j_{0}},\dots,b_{j_{s}}]$. 
If the simplices are not compatible, both members of the previous equality  are zero.
The only cases giving
compatible  simplices are one of the following cases:
\begin{equation}\label{equa:ellegalp}
(*)\ 
a_{i_{t}}<e<a_{i_{t+1}} \text{ with } t\in\{0,\dots,r-1\},
\ \ (*) \  e<a_{i_{0}} \text{ and }e\in\Delta_{p}. 
\end{equation}
We treat the first case, the second being identical.
Notice:
\begin{eqnarray*}
F_{p}\ast G_{p}\ast e
&=&
[a_{i_{0}},\dots,a_{i_{r}}]\ast [b_{j_{0}},\dots,b_{j_{s}}]\ast e\\
&=&
(-1)^{s+1+r-t}
[a_{i_{0}},\dots,a_{i_{t}},e,a_{i_{t+1}},\dots, a_{i_{r}}]\ast [b_{j_{0}},\dots,b_{j_{s}}]
\end{eqnarray*}
and, setting
$F_{p(e)}=[a_{i_{0}},\dots,a_{i_{t}},e,a_{i_{t+1}},\dots, a_{i_{r}}]$,
$$
F_{p}(e)\cup G_{p}
=
[a_{i_{0}},\dots,a_{i_{t}},e,a_{i_{t+1}},\dots, a_{i_{r}}=b_{j_{0}},\dots,b_{j_{s}}]
=
(-1)^{s+r-t}
(F_{p}\cup G_{p}) \ast e.
$$
It follows
\begin{eqnarray*}
\Theta_{\Delta}(\1_{(F \ast G,\varepsilon)}\ast e_{0})
=&&\\
(-1)^{|(F\ast G,\varepsilon)|_{\geq p}+|(F,\varepsilon)|_{<p}+|G_{p}|+1}
\1_{(F_{0},\varepsilon_{0})}
\otimes\dots\otimes
\1_{((F_{p}\cup G_{p})\ast e,\varepsilon_{p})}
\otimes\dots\otimes
\1_{G_{n}}=&&\\
%
(-1)^{|(F,\varepsilon)|_{<p}+|G_{p}|+\varepsilon_{p}}
(\1_{(F_{0},\varepsilon_{0})}
\otimes\dots\otimes
(\1_{(F_{p}\cup G_{p},\varepsilon_{p})})
\otimes\dots\otimes
\1_{G_{n}})\ast e
=&&\\
(\Theta_{\Delta}(\1_{(F_{0},\varepsilon_{0})}
\otimes\dots\otimes
(\1_{(F_{p}\ast G_{p},\varepsilon_{p})})
\otimes\dots\otimes
\1_{G_{n}}))\ast e.&&
\end{eqnarray*}
In the case (\ref{equa:ellegalp}), we have
$\Theta_{\Delta}(\1_{(F \ast G,\varepsilon)}\ast e_{1})=0$, 
which gives the desired equality.

\medskip\noindent
Similar argument gives the result in the symmetric case of the previous one:
``$b_{j_{t}}<e<b_{j_{t+1}}$ with $t\in\{1,\dots,s-1\}$
or  $b_{j_{s}}<e$ with $e\in\Delta_{p}$.''

\medskip
2)
We now consider a virtual vertex  $\tv_{\ell}$.
If $\varepsilon_{\ell}=1$, equality (\ref{equa:thetavirtuel}) is verified since its two members are zero. 
We therefore assume $\varepsilon_{\ell}=0$. If $\ell<p$, we have:
\begin{eqnarray*}
\Theta_{\Delta}(\1_{(F \ast G,\varepsilon)}\ast \tv_{\ell})
=&&\\
(-1)^{|(F\ast G,\varepsilon)|_{>\ell}+
|(F,\varepsilon)|_{<p}+1+|(G_{p},\varepsilon_{p})|}
\dots\otimes
(\1_{(F_{\ell},0)}\ast \tv_{\ell})
\otimes
\dots\otimes
\1_{(F_{p}\cup G_{p},\varepsilon_{p})}
\otimes\dots =&&
\\
(-1)^{|(F,\varepsilon)|_{<p}+|(G_{p},\varepsilon_{p})|}
(
\1_{(F_{0},\varepsilon_{0})}
\otimes
\dots\otimes
\1_{(F_{p}\cup G_{p},\varepsilon_{p})}
\otimes\dots
\otimes
\1_{G_{n}}
)\ast \tv_{\ell}=&&\\
(\Theta_{\Delta}(\1_{(F \ast G,\varepsilon)}))\ast \tv_{\ell}.&&
\end{eqnarray*}
The argument is similar for $\ell\geq p$.
\end{proof}

\begin{proposition}\label{prop:homotopieRglobal}
Let $X$ be a filtered space and let  $ X\times \R$ endowed with the product filtration and a perversity~$\ov{p}$. 
Denoting also by $\ov{p}$ the induced perversity on $X$, there exists a linear map
$\Theta \colon \lau \tN*{\ov{p}}{X\times \R;R}\to \lau \tN {*-1} {\ov{p}}{X\times \R;R}$,
such that
\begin{equation}\label{equa:homotopieXR}
\Theta \circ \delta+\delta\circ \Theta= (I_{0}\circ\pr)^*-\id.
\end{equation}
\end{proposition}

\begin{proof}
Let $\sigma\colon\Delta\to X\times \R$ be a regular simplex and   $\omega\in  \Hiru \tN k{X\times \R;R}$. 
We define
$\hat{\sigma}\colon \cal L_\Delta =  \Delta\otimes [0,1]\to X\times \R$
by
$\hat{\sigma}(x,t)=(\pr(\sigma(x)),t\pr_{2}(\sigma(x)))$ and we set
$$\omega_{\Delta\otimes [0,1]}=
\sum_{(F\ast G,\varepsilon)\fa \widetilde{\Delta \otimes [0,1]}}
\omega_{\hat{\sigma}\circ \iota_{F\ast G}}(F\ast G,\varepsilon)\,\1_{(F\ast G,\varepsilon)},
$$
where $\iota_{F\ast G}\colon F\ast G \hookrightarrow \Delta \otimes [0,1]$ is the canonical injection.
By construction, $\omega_{\Delta \otimes [0,1]}\in\Hiru \tN k {\Delta \otimes [0,1]}$, and we can define
\begin{equation*}\label{equa:homotopiee}
\Theta(\omega)_\sigma=
\Theta_{\Delta }(\omega_{\Delta \otimes [0,1]})
\in \Hiru \tN {k-1}{\Delta }.
\end{equation*}
We need to verify $\Theta(\omega)\in \Hiru \tN*{X\times \R;R}$. For this purpose, we consider a regular face operator
$\delta_{\ell}\colon \nabla \to\Delta $
and
$\tau=\sigma\circ\delta_{\ell}$.
The property
$\omega_{\Delta \otimes [0,1]}\in \Hiru\tN *{\Delta \otimes [0,1]}$
becomes
$\delta_{\ell}^*\,\omega_{\Delta \otimes [0,1]}=\omega_{\nabla \otimes [0,1]}$.
The map
${\delta}^*_{\ell}\colon \Hiru \tN * {\Delta \otimes [0,1]}\to \Hiru \tN*{\nabla \otimes [0,1]}$
verifies
$$
{\delta}^*_{\ell}(\1_{(F*G,\varepsilon)})=
\left\{\begin{array}{cl}
\1_{(F*G,\varepsilon)}
&\text{if}\;
F*G\fa \nabla \otimes [0,1],\\
0&\text{if not,}
\end{array}\right.
$$
which implies the commutativity of the following diagram,
\begin{equation*}\label{equa:biendefini2}
\xymatrix{
\Hiru \tN *{\Delta \otimes [0,1]}
\ar[rr]^-{\Theta_{\Delta }}
\ar[d]_{{\delta}^*_{\ell}}
&&
\Hiru \tN{*-1}{\Delta \otimes [0,1]}
\ar[d]^{{\delta}^*_{\ell}}
\\
\Hiru \tN *{\nabla \otimes [0,1]}
\ar[rr]^-{\Theta_\nabla }
&&
\Hiru \tN {*-1}{\nabla \otimes [0,1]}.
}
\end{equation*}
We  deduce 
\begin{eqnarray*}
\delta^*_{\ell}(\Theta(\omega)_\sigma)
&=&
\delta^*_{\ell}\Theta_\Delta (\omega_{\Delta \otimes [0,1]})=
\Theta_\nabla (\delta^*_{\ell}\omega_{\Delta \otimes [0,1]})=
\Theta_\nabla (\omega_{\nabla \otimes [0,1]})
=
\Theta(\omega)_\tau .
\end{eqnarray*}
The map $\Theta\colon \Hiru \tN *{X\times \R;R}\to
\Hiru \tN {*-1}{X\times \R;R}$ is well defined.
From  \propref{prop:homotopieR}, we can deduce
\begin{equation}\label{equa:thetadelta}
(\Theta\delta+\delta\Theta)(\omega )_\sigma
=
\sum_{(F,\varepsilon)\fa \widetilde{\Delta \times\{0\}}}
\omega_{\hat{\sigma}\circ \iota_{F}}(F,\varepsilon)\,\1_{(F,\varepsilon)}
-
\sum_{(G,\varepsilon)\fa \widetilde{\Delta \times\{1\}}}
\omega_{\hat{\sigma}\circ \iota_{G}}(G,\varepsilon)\,\1_{(G,\varepsilon)}.
\end{equation}
Recall the canonical inclusions $\iota_{0},\,\iota_{1}\colon \Delta \to \Delta \otimes [0,1]$.
If $F$ is a face of $\Delta \times \{0\}$ identified to
$\nabla\fa\Delta $, one has
$\iota_{F}=\iota_{0}\circ\iota_{\nabla}$, from where
$$\hat{\sigma}\circ\iota_{F}=\hat{\sigma}\circ\iota_{0}\circ \iota_{\nabla}=I_{0}\circ\pr\circ\sigma\circ\iota_{\nabla}$$
and
$$
\sum_{(F,\varepsilon)\fa \widetilde{\Delta \times\{0\}}}
\omega_{\hat{\sigma}\circ \iota_{F}}(F,\varepsilon)\,\1_{(F,\varepsilon)}
= (I_{0}\circ\pr)^*
\left(
\sum_{(\nabla,\varepsilon)\fa \widetilde{\Delta }}
\omega_{{\sigma}\circ \iota_{\nabla}}(\nabla,\varepsilon)\,\1_{(\nabla,\varepsilon)}
\right)
= (I_{0}\circ\pr)^*(\omega ) _\sigma.
$$
Using $\iota_{G}=\iota_{1}\circ \iota_{\nabla}$
and
$\hat{\sigma}\circ\iota_{G}=\hat{\sigma}\circ\iota_{1}\circ \iota_{\nabla}=\sigma\circ\iota_{\nabla}$,
we show, as above, that the right side-hand of the second term of  (\ref{equa:thetadelta}) is  equal to
$\omega_\sigma $. In summary, we have shown,
$$(\Theta\delta+\delta\Theta)(\omega ) _\sigma= (I_{0}\circ\pr)^*(\omega ) _\sigma-\omega _\sigma ,$$
from which comes the equality (\ref{equa:homotopieXR}).

\medskip

As for the compatibility with the perverse  degrees, we prove
$\|\Theta(\omega)\|\leq \|\omega\|$,
for each  $\omega\in \Hiru \tN*{X\times \R;R}$, by arguing over the chosen basis. This inequality
follows directly from the definition of $\Theta$,
$$\|\Theta_{\Delta}(\1_{(F\ast G,\varepsilon)})\|_{\ell}
=\|\1_{(F\cup G,\varepsilon)}\|_{\ell}
\leq 
\|\1_{(F\ast G,\varepsilon)}\|_{\ell},
$$
for each $\ell\in\{1,\dots,n\}$. 
With (\ref{equa:homotopieXR}), we have constructed the homotopy\\
$\Theta \colon \lau \tN*{\ov{p}}{X\times \R;R}\to \lau \tN {*-1} {\ov{p}}{X\times \R;R}$.
\end{proof}

\begin{proof}[Proof of  \thmref{prop:isoproduitR}]
This is a direct consequence of the equalities $\pr\circ I_{0}=\pr\circ I_{1}=\id$ and 
$\lau \IH*{\ov{p}}{I_{0}\circ \pr}=\id
$, proved in (\ref{equa:homotopieXR}).
\end{proof}

The following result is used in \secref{15}.

\begin{proposition}\label{cor:SfoisX}
Let  $(X,\ov p)$ be a paracompact perverse space. Let $S^{\ell}$ be the unit sphere of $\R^{\ell+1}$. 
The canonical projection $p_{X}\colon {S}^\ell \times X\to X$, $(z,x)\mapsto x$, 
induces an isomorphism
$\lau \IH {k} {\ov{p}}{{S}^\ell\times X;R}\cong \lau \IH {k} {\ov{p}}{X;R}\oplus \lau \IH {k-\ell} {\ov{p}}{X;R}$.
\end{proposition}

\begin{proof}
It suffices to use an induction on $\ell$  with the decomposition 
${S}^{\ell} = {S}^{\ell} \backslash \{\mathrm{North pole}\} \cup  {S}^\ell \backslash \{\mathrm{South pole}\}$
(see \thmref{thm:MVcourte} and \thmref{prop:isoproduitR}).
\end{proof}


\section{Blown-up intersection cohomology of a cone. \thmref{prop:coneTW}.}\label{sec:cohomologiecone}

In this section, $ X $ is an $n$-dimensional \emph{compact} filtered space  and we represent the open cone as the quotient $ \rc  X = X \times [0,\infty[ /X \times \{0\}$, whose apex is denoted by $ \tw$. The formal dimension of $\rc X$ is $ n + 1 $ relatively to the conical filtration,
$(\rc X)_{i}=\rc X_{i-1}$ if $i\geq 1$ and $(\rc X)_{0}=\{\tw\}$.
The purpose of this section is to prove the following proposition, cf. also
\cite[Corollary 1.47]{CST1} and
\cite[Proposition 3.1.1]{MR2210257}.

\bt\label{prop:coneTW}
Let  $X$  be a compact filtered space.
Consider the open cone, $\rc  X = X \times [0,\infty[ /X \times \{0\}$, equipped with the conical filtration and a perversity $ \ov {p} $.
We also denote by $\ov{p}$ the perversity induced on $ X $.
The following properties are verified for any commutative ring, $R$.
\begin{enumerate}[\rm (a)]
\item The inclusion $\iota\colon X\to \rc X$, $x\mapsto [x,1]$,
induces an isomorphism,
$\lau \IH {k}{\ov{p}}{\rc X;R}\xrightarrow[]{\cong} \lau \IH{k}{\ov{p}}{X;R}$,
for each $k\leq \ov{p}(  \tw )$.
\item 
For each $k>  \ov{p}(  \tw )$, we have $\lau \IH {k}{\ov{p}}{\rc X;R}=0$.
\end{enumerate}
\et

\subsection{Simplices on a filtered space and its cone} 
First we link the complexes of $ X $ and $ \rc X$.
The formal dimension of the cone being different from that of the original space, we introduce some operations which increase or decrease the length of filtrations.
\begin{itemize}
\item If $\Delta=\Delta_{0}\ast\dots\ast \Delta_{n+1}$ is a regular simplex, of formal dimension   $n+1$, we define a regular simplex, of formal dimension  $n$, by
$\widehat{\Delta}=\Delta_{1}\ast\dots\ast\Delta_{n+1}$.
Its filtration is characterized by $\widehat{\Delta}_{i}=\Delta_{i+1}$, for each $i\in\{0,\dots,n\}$.

\item Let $\sigma\colon \Delta_{\sigma}=\Delta_{0}\ast\dots\ast\Delta_{n+1}\to \rc X$ be a regular simplex of $\rc X$.
Since $\sigma(\widehat{\Delta}_{\sigma})\subset X\times ]0,\infty[$, we define the restriction
$$\hat{\sigma}\colon\Delta_{\hat{\sigma}}=\widehat{\Delta}_{\sigma}\xrightarrow[]{\sigma}X\times ]0,\infty[.
$$
\item For each regular simplex of $\rc X$, 
$\sigma\colon \Delta_{\sigma}=\Delta_{0}\ast\dots\ast\Delta_{n+1}\to \rc X$,
the image of a point $x\in \Delta_{\sigma}$ can be written as,
$$\sigma(x)=[\sigma_{1}(x),\sigma_{2}(x)]\in {\rc X}=X\times [0,\infty[/X\times \{0\}.$$
Associated to the simplex $\sigma$, there is the following regular simplex of $\rc X$, 
$$\tc\sigma\colon \Delta_{\tc\sigma}=\left(\{\tp\}\ast\Delta_{0}\right)\ast\dots\ast \Delta_{n+1}\to \rc X,$$
defined by
$\tc\sigma((1-t)\tp+tx)=[\sigma_{1}(x),t\sigma_{2}(x)]$.
Moreover, if one considers $\hat\sigma\colon \widehat{\Delta}_{\sigma}\to X\times ]0,\infty[\hookrightarrow \rc X$ 
as a filtered simplex of the cone, then $\tc\hat{\sigma}$ is a face of $\tc \sigma$.
\end{itemize}

The \emph{truncation} 
of 
a cochain complex
is defined for all positive integers 
$s$ by
\begin{equation}\label{equa:troncation}
(\tau_{\leq s} C)^r=\left\{
\begin{array}{ccl}
C^r
&\text{if}&
r<s,\\
\cal Z C^s
&\text{if}&
r=s,\\
0
&\text{if}&
r>s,
\end{array}\right.
\end{equation}
where $\cal Z C^s$ means the $R$-module of cocycles whose  degree is $s$.

\parr{Construction of $f\colon \Hiru \tN*{X\times ]0,\infty[;R}\to \Hiru \tN*{\rc X;R}$}

Let $\sigma\colon \Delta_{\sigma}=\Delta_{0}\ast\dots\ast\Delta_{n+1}\to \rc X$
and 
$\omega\in \Hiru\tN *{X\times ]0,\infty[;R}$. We denote by
$\lambda_{\tc\Delta_{0}}$ the cocycle $\1_{(\emptyset,1)}+\sum_{e\in \cal V(\Delta_{0})}\1_{([e],0)}\in N^0(\tc\Delta_{0})$. We set
\begin{equation}\label{equa:lef}
f(\omega)_{\sigma}=\lambda_{\tc\Delta_{0}}\otimes \omega_{\hat{\sigma}}.
\end{equation}

\begin{proposition}\label{prop:lef}
Let $(\rc X,\ov{p})$ be a perverse space over the cone of the compact space  $X$ and let 
$(X\times ]0,\infty[,\ov{p})$ be the induced perverse space.
The  correspondence defined above  induces a  cochain map,
$$f\colon \tau_{\leq \ov{p}(\tw)} \lau \tN* {\ov{p}}{X\times ]0,\infty[;R}\to 
\tau_{\leq \ov{p}(\tw)}\lau \tN *{\ov{p}}{\rc X;R}.$$
\end{proposition}

\begin{proof}
First, we check that the application $ f $, defined locally at the level of simplices, extends globally to
$\Hiru \tN*{X\times ]0,\infty[;R}$.
For this, we must establish
$\delta^*_{k} f(\omega)_{\sigma}=f(\omega)_{\sigma\circ\delta_{k}},$
for each $\omega\in \Hiru \tN*{X\times ]0,\infty[;R}$, each regular simplex, 
$\sigma\colon \Delta_{\sigma}=\Delta_{0}\ast\dots\ast\Delta_{n+1}\to \rc X$, and any regular face operator,
 $\delta_{k}\colon \nabla\to \Delta_{\sigma}$, with $k\in\{0,\dots,\dim\Delta_{\sigma}\}$.
Let  $j_{0}$ denote the dimension of $\Delta_{0}$.
To determine the effect of $\delta_{k}$ on the operation $\sigma\mapsto \hat{\sigma}$,
 we must distinguish $ k> j_{0} $ of $ k \leq j_{0 } $.
 For the sake of convenience, we set $\delta_{s}=\id$ if $s<0$. 
From the construction of $\hat{\sigma}$, we have
 $$\widehat{\sigma\circ\delta_{k}}=\left\{
 \begin{array}{ccl}
 \hat{\sigma}\circ \delta_{k-j_{0}-1}
 &\text{if}&
 k>j_{0},\\
 \hat{\sigma}
  &\text{if}&
  k\leq j_{0},
  \end{array}\right.$$
which implies $\widehat{\sigma\circ\delta_{k}}=\hat{\sigma}\circ \delta_{k-j_{0}-1}$, with the previous convention.
  We conclude
  $$\delta^*_{k}f(\omega)_{\sigma}=\delta^*_{k}\left(\lambda_{\tc\Delta_{0}}\otimes \omega_{\hat{\sigma}}\right)=\left\{
 \begin{array}{ccl}
 \lambda_{\tc\Delta_{0}}\otimes \delta^*_{k-j_{0}-1}\omega_{\hat{\sigma}} 
&\text{if}&
 k>j_{0},\\
  \lambda_{\tc \nabla_{0}}\otimes \omega_{\hat{\sigma}}
  &\text{if}&
  k\leq j_{0}.
  \end{array}\right.$$
 It follows
  $\delta^*_{k}f(\omega)_{\sigma}=\lambda_{\tc \nabla_{0}}\otimes \omega_{\widehat{\sigma\circ\delta_{k}}}=
  f(\omega)_{\sigma\circ\delta_{k}}$.
  
  Since the 0-cochain $\lambda_{\tc\Delta_{0}}$ is a cocycle, the compatibility with the differentials is immediate from the equalities
$$\delta\left(f(\omega)_{\sigma}\right)=\delta\left(\lambda_{\tc\Delta_{0}}\otimes \omega_{\hat{\sigma}}\right)=
\lambda_{{\tc\Delta_{0}}}\otimes \delta \omega_{\hat{\sigma}}=f(\delta\,\omega)_{\sigma}.$$

The map  $f$ 
being compatible with the differentials, it remains to show that the image by $ f $ of
a $ \ov {p} $-allowable cochain,
 $\omega\in \Hiru \tN*{X\times ]0,\infty[;R}$, 
is a  $\ov{p}$-allowable cochain in $\Hiru \tN*{\rc X;R}$.
We choose $ \omega $ of degree less than or equal to $ \ov {p} (\tw) $
and refer to \defref{def:admissible} for the property of $ \ov {p} $-allowability.
For  the stratum reduced to $\tw$, the allowability
comes   directly  from
$\|f(\omega)_{\sigma}\|_{n+1}\leq |\omega_{\hat{\sigma}}|\leq \ov{p}(\tw).$
Now consider
a singular  stratum $ S $ of $ X $
and a regular simplex
$\sigma\colon \Delta_{\sigma}=\Delta_{0}\ast\dots\ast\Delta_{n+1}\to \rc X$,
such that $\sigma(\Delta_{\sigma})\cap (S\times ]0,\infty[)\neq \emptyset$.
Let $\ell=\codim_{\!X\times ]0,\infty[}(S\times ]0,\infty[)$ and notice the equivalence of the conditions
$\sigma(\Delta_{\sigma})\cap (S\times ]0,\infty[)\neq \emptyset$ and
$\hat{\sigma}(\Delta_{\hat{\sigma}})\cap (S\times ]0,\infty[)\neq \emptyset$. 
For such stratum, we have $\ell\in \{1,\dots,n\}$ and  
$\|f(\omega)_{\sigma}\|_{\ell}=\|\lambda_{\tc\Delta_{0}}\otimes \omega_{\hat{\sigma}}\|_{\ell}=
\|\omega_{\hat{\sigma}}\|_{\ell}$.
The result is a consequence of the inequality
$\|\omega_{\hat{\sigma}}\|_{\ell}\leq\|\omega\|_{S\times ]0,\infty[}\leq \ov{p}(S\times ]0,\infty[)$, 
arising from the $\ov{p}$-allowability of $\omega$.
\end{proof}

\parr{Construction of $g\colon  \Hiru \tN*{\rc X;R}\to \Hiru \tN*{ X\times ]0,\infty[;R}$}

Let $\omega\in\Hiru \tN *{\rc X;R}$ and 
$\tau\colon \Delta_{\tau}\to X\times ]0,\infty[$ a regular simplex.
We denote by $\tc\tau\colon \Delta_{c\tau}=\{\tp\}\ast \Delta_{\tau}\to \rc X$  the cone over $\tau$ 	defined above.
Notice $\widetilde{\Delta_{\tc\tau}}=\tc\{\tp\}\times \widetilde{\Delta_{\tau}}$.

\emph{Let  $\tv_{0}$ be
the apex of the cone over the  component of filtration degree 0 of a filtered simplex.}
The cone $\tc\{\tp\}$ having two vertices $\tp$ and $\tv_{0}$, the cochain $\omega_{\tc\tau}$ decomposes into
\begin{equation}\label{equa:leg}
\omega_{\tc\tau}=\1_{\tp}\otimes \gamma_{\tp}+\1_{\tv_{0}}\otimes \gamma_{\tv_{0}}+
\1_{\tp\ast\tv_{0}}\otimes \gamma'_{\tv_{0}},
\end{equation}
with $\gamma_{\tp},\gamma_{\tv_{0}},\gamma'_{\tv_{0}}\in \Hiru \tN *{\Delta_{\tau}}$. We set
$$g(\omega)_{\tau}=\gamma_{\tv_{0}}.$$

\begin{proposition}\label{prop:leg}
Let $(\rc X,\ov{p})$ be a perverse space with $X$ compact 
and   let $(X\times ]0,\infty[,\ov{p})$ be  the induced perverse space.
The correspondence defined above induces a cochain map,
$$g\colon \tau_{\leq \ov{p}(\tw)}\lau \tN *{\ov{p}}{\rc X;R}\to \tau_{\leq \ov{p}(\tw)} \lau \tN * {\ov{p}}{X\times ]0,\infty[;R}.$$
\end{proposition}

The proof follows the pattern of that of \propref{prop:lef};  we leave it to the reader.

\bigskip\noindent
\emph{Specify the compositions $f\circ g$ and $g\circ f$.}
\begin{enumerate}[\rm (a)]
\item Let $\omega\in\Hiru \tN *{X;R}$. Consider a regular simplex, $\tau\colon \Delta_{\tau}\to X\times]0,\infty[$ and its  associated map $\tc\tau\colon \{\tp\}\ast\Delta_{\tau}\to \rc X$. Following (\ref{equa:lef}), one has
$$f(\omega)_{\tc\tau}=\lambda_{\tc\{\tp\}}\otimes \omega_{\tau}
= \1_{\tp}\otimes \omega_{\tau}+
\1_{\tv_{0}}\otimes \omega_{\tau}.$$
It follows, according to (\ref{equa:leg}), $g(f(\omega))_{\tau}=\omega_{\tau}$ and $g\circ f=\id$.
\item Let $\omega\in\Hiru \tN *{\rc X;R}$. Consider a regular simplex
$\sigma\colon\Delta_{\sigma}=\Delta_{0}\ast\dots\ast \Delta_{n+1}\to \rc X$,
and its associated map $\tc\sigma\colon (\{\tp\}\ast\Delta_{0})\ast\dots\ast \Delta_{n+1}\to \rc X$.
The cochain $\omega_{\tc\sigma}$ decomposes into
\begin{equation}\label{equa:omegacsigma}
\omega_{\tc\sigma}=\sum_{F\fa \tc\Delta_{0}}\1_{F}\otimes \gamma_{F}+
\sum_{F\fa \tc\Delta_{0}}\1_{\tp\ast F}\otimes \gamma'_{F}+\1_{\tp}\otimes \gamma_{\emptyset}.
\end{equation}
Since the cochain $ \omega $ is globally defined and the simplex $ \tc \hat {\sigma} $ is a face of $ \tc \sigma $, we deduce
$\omega_{\tc\hat{\sigma}}=
\1_{\tv_{0}}\otimes \gamma_{\tv_{0}}+
\1_{\tp\ast\tv_{0}}\otimes \gamma'_{\tv_{0}}+
\1_{\tp}\otimes \gamma_{\emptyset}$.
It follows:
\begin{equation}\label{equa:fg}
f(g(\omega))_{\sigma}=\lambda_{\tc\Delta_{0}}\otimes g(\omega)_{\hat{\sigma}}=\lambda_{\tc\Delta_{0}}\otimes \gamma_{\tv_{0}}.
\end{equation}
\end{enumerate}

\parr{Construction of a homotopy $H\colon\Hiru \tN *{\rc X;R}\to \Hiru \tN {*-1}{\rc X;R}$}

If $\sigma\colon\Delta_{\sigma}=\Delta_{0}\ast\dots\ast \Delta_{n+1}\to \rc X$ is a regular simplex, we define a map
$H\colon \Hiru \tN{*}{\Delta_{\tc\sigma}}\to \Hiru \tN {*-1}{\Delta_{\sigma}}$, i.e.,
$$H\colon \Hiru N{*}{\tc(\tp\ast\Delta_{0})}\otimes \Hiru N*{\tc\Delta_{1}}\otimes\dots\otimes \Hiru N*{\Delta_{n+1}}
\to
\Hiru N{*-1}{\tc\Delta_{0}}\otimes \dots\otimes \Hiru N*{\Delta_{n+1}}.$$
We decompose $\omega_{\tc\sigma}\in  \Hiru\tN {*}{\Delta_{\tc\sigma}}$ 
as in the formula  (\ref{equa:omegacsigma}) and set:
\begin{equation}\label{equa:leH}
(H(\omega))_{\sigma}=
\sum_{\tv_{0}\neq F\fa \tc\Delta_{0}}(-1)^{|F|+1}
\1_{F}\otimes \gamma'_{F}+\lambda_{\Delta_{0}}\otimes \gamma'_{\tv_{0}},
\end{equation}
where $\tv_{0}$ is the apex of the cone over the component  filtration  of degree 0  and
$\lambda_{\Delta_{0}}$ the sum of 
0-cochains on $\Delta_{0}$.

\begin{proposition}\label{prop:leH}
Let $(\rc X,\ov{p})$ be a perverse space with $X$ compact.
\begin{enumerate}[\rm (a)]
\item The equality (\ref{equa:leH}) 
induces a linear map,
$H\colon\Hiru \tN *{\rc X;R}\to \Hiru \tN {*-1} {\rc X;R}$,
verifying
$$\delta\circ H+H\circ \delta=\id -f\circ g.$$
\item
Using the notation introduced in (\ref{equa:troncation}),  the application $ H $ induces a map,
$$H\colon \tau_{\leq \ov{p}(\tw)}\lau \tN *{\ov{p}}{\rc X;R}\to \tau_{\leq \ov{p}(\tw)} \lau \tN{*-1}{\ov{p}}{\rc X;R}.$$
\end{enumerate}
\end{proposition}

\begin{proof}
(a) We must establish the equality
$\delta_{k}^*H(\omega)_{\sigma}=H(\omega)_{\sigma\circ\delta_{k}},$
for each cochain $\omega\in  \Hiru \tN *{\rc X}$, each regular simplex 
$\sigma\colon\Delta_{\sigma}=\Delta_{0}\ast\dots\ast\Delta_{n+1}\to \rc X$
and each regular face operator
$\delta_{k}\colon D=D_{0}\ast\dots\ast D_{n+1}\to \Delta_{\sigma}$, with $k\in\{0,\dots,\dim\Delta_{\sigma}\}$. 
Denoting by $\delta^{\tp}_{*}$ 
the regular face operators of 
$\{\tp\}\ast \Delta_{\sigma}$, we can write
$\tc(\sigma\circ\delta_{k})=\tc\sigma \circ \delta^{\tp}_{k+1}$. 
If $k>\dim\Delta_{0}$, we set $k^{\circ}=k-\dim\Delta_{0}-1$ and
$\delta^{\circ}_{k^{\circ}}\colon D_{1}\ast\dots\ast D_{n+1}\to \Delta_{1}\ast\dots\ast \Delta_{n+1}$
the induced face. Following (\ref{equa:leH}), we have:
$$\delta_{k}^*H(\omega)_{\sigma}
=\left\{\begin{array}{cl}
\sum_{\tv_{0}\neq F\fa \tc D_{0}}(-1)^{|F|+1}
\1_{F}\otimes \gamma'_{F}+\lambda_{D_{0}}\otimes \gamma'_{\tv_{0}}
&
\text{if } k\leq\dim\Delta_{0},\\[.2cm]
\sum_{\tv_{0}\neq F\fa \tc\Delta_{0}}(-1)^{|F|+1}
\1_{F}\otimes \delta_{k^{\circ}}^{\circ,*}\gamma'_{F}+
\lambda_{\Delta_{0}}\otimes \delta_{k^{\circ}}^{\circ,*}\gamma'_{\tv_{0}}
&
\text{if } k>\dim\Delta_{0}.
\end{array}\right.
$$
Using the equality
$\omega_{\tc(\sigma\circ\delta_{k})}=\omega_{\tc\sigma\circ\delta^{\tp}_{k+1}}=
\delta_{k+1}^{\tp,*}\omega_{\tc\sigma}$ and
(\ref{equa:omegacsigma}), we get:
$$
\omega_{\tc(\sigma\circ\delta_{k})}
=
\left\{\begin{array}{cl}
(\sum_{F\fa \tc D_{0}}\1_{F}\otimes \gamma_{F}+
\1_{\tp\ast F}\otimes \gamma'_{F})
+\1_{\tp}\otimes \gamma_{\emptyset}
&
\text{if } k\leq \dim \Delta_{0},\\[.2cm]
(\sum_{F\fa \tc D_{0}}\1_{F}\otimes \delta_{k^{\circ}}^{\circ,*}\gamma_{F}+
\1_{\tp\ast F}\otimes \delta_{k^{\circ}}^{\circ,*}\gamma'_{F})
+\1_{\tp}\otimes \delta_{k^{\circ}}^{\circ,*}\gamma_{\emptyset}
&
\text{if } k>\dim\Delta_{0},
\end{array}\right.
$$
which gives
\begin{eqnarray*}
H(\omega)_{\sigma\circ\delta_{k}}
&=&
\left\{\begin{array}{cl}
(\sum_{\tv_{0}\neq F\fa \tc D_{0}}(-1)^{|F|+1}
\1_{F}\otimes \gamma'_{F})+
\lambda_{D_{0}}\otimes \gamma'_{\tv_{0}},
&
\text{if } k\leq \dim \Delta_{0},\\[.2cm]
(\sum_{\tv_{0}\neq F\fa \tc D_{0}}(-1)^{|F|+1}
\1_{F}\otimes \delta_{k^{\circ}}^{\circ,*}\gamma'_{F})+
\lambda_{D_{0}}\otimes \delta_{k^{\circ}}^{\circ,*}\gamma'_{\tv_{0}}
&
\text{if } k>\dim\Delta_{0},
\end{array}\right.\\[,2cm]
&=&
\delta_{k}^* H(\omega)_{\sigma}.
\end{eqnarray*}

\medskip
Let us study the behavior of $ H $ towards the differentials.
Let
$\sigma\colon\Delta_{\sigma}=\Delta_{0}\ast\dots\ast\Delta_{n+1}\to \rc X$ be a regular simplex and
  $\omega\in\tN^*(\rc X;R)$. Let us start from the equality (\ref{equa:omegacsigma})
and calculate the differential in
eliminating terms having a zero image by $ H $.
\begin{eqnarray}\label{previa}
(H(\delta \omega))_{\sigma}
&=&
H\left(
\sum_{e\in\cal V(\tc\Delta_{0})}
\1_{p\ast e}\otimes \gamma_{\emptyset}
+
\sum_{F\fa\tc\Delta_{0}}\1_{F\ast \tp}\otimes \gamma_{F}+ \right.\\\nonumber
&&
\left.
\sum_{\substack{F\fa\tc\Delta_{0} \\ e\in\cal V(\tc\Delta_{0})\phantom{-}}}
\1_{\tp\ast F\ast e}\otimes \gamma'_{F}+
\sum_{F\fa\tc\Delta_{0}}(-1)^{|F|+1}\1_{\tp\ast F}\otimes \delta \gamma'_{F}\right).
\label{eqna:homotopie}
\end{eqnarray}
Notice
\begin{eqnarray*}
\sum_{e\in\cal V(\tc\Delta_{0})}
\1_{p\ast e}\otimes \gamma_{\emptyset}
+
\sum_{F\fa\tc\Delta_{0}}\1_{F\ast \tp}\otimes \gamma_{F} =&&\\
\sum_{e\in\cal V(\Delta_{0})}
\1_{p\ast e}\otimes \gamma_{\emptyset}
+
\sum_{\tv_{0}\neq F\fa\tc\Delta_{0}}\1_{F\ast \tp}\otimes \gamma_{F}
+\1_{\tp\ast\tv_{0}}\otimes (\gamma_{\emptyset}-\gamma_{\tv_{0}}).&&
\end{eqnarray*}
Replacing in \eqref{previa} and developing the definition of  $H$, we get:
\begin{eqnarray*}
(H(\delta \omega))_{\sigma}
&=&
-\sum_{e\in\cal V(\Delta_{0})}\1_{e}\otimes \gamma_{\emptyset}
+\sum_{\tv_{0}\neq F\fa\tc\Delta_{0}}\1_{F}\otimes \gamma_{F}
+\lambda_{\Delta_{0}}\otimes (\gamma_{\emptyset}-\gamma_{\tv_{0}}-\delta\gamma'_{\tv_{0}})\\
&&
+\sum_{\substack{F\fa\tc\Delta_{0} \\ e\in\cal V(\tc\Delta_{0})\phantom{-}}}
(-1)^{|F|}\1_{F\ast e}\otimes \gamma'_{F}
+
\sum_{\tv_{0}\neq F\fa\tc\Delta_{0}}\1_{F}\otimes \delta \gamma'_{F}.
\end{eqnarray*}
On the other hand, the quantity $\delta\circ H$ can be written
\begin{eqnarray*}
\delta(H(\omega))_{\sigma}
&=&
\sum_{\substack{\tv_{0}\neq F\fa\tc\Delta_{0} \\ e\in\cal V(\tc\Delta_{0})\phantom{-}}}
(-1)^{|F|+1}\1_{F\ast e}\otimes \gamma'_{F}
-\sum_{\tv_{0}\neq F\fa\tc\Delta_{0}} \1_{F}\otimes \delta \gamma'_{F}\\
&&
+\sum_{e\in\cal V(\Delta_{0})}\1_{e\ast\tv_{0}}\otimes \gamma'_{\tv_{0}}
+\lambda_{\Delta_{0}}\otimes \delta\gamma'_{\tv_{0}}.
\end{eqnarray*}
Using (\ref{equa:fg}),
the sum of the two expressions can be reduced to:
\begin{equation}\label{equa:deltaH}
H(\delta\omega)_{\sigma}+\delta H(\omega)_{\sigma}
=
\sum_{F\fa \tc\Delta_{0}}\1_{F}\otimes \gamma_{F}
-\lambda_{\tc\Delta_{0}}\otimes \gamma_{\tv_{0}}=
\omega_{\sigma}-(f\circ g)(\omega)_{\sigma}.
\end{equation}

\medskip
(b) As in the proof of \propref{prop:lef}, we are reduced to consider a singular stratum  $T$ of $\rc X$,
a $\ov{p}$-allowable cochain $\omega\in \Hiru \tN *{\rc X}$, of degree  less than or equal  to $\ov{p}(\tw)$, 
and a regular simplex, $\sigma\colon\Delta_{\sigma}\to \rc X$, with $\sigma(\Delta_{\sigma})\cap T\neq\emptyset$.
Let $\ell=\codim_{\!\rc X} T\in \{1,\dots,n+1\}$.
Notice that $\tc\sigma(\Delta_{\tc\sigma})\cap T\neq\emptyset$. Thus, according to the definition of perverse degree
(cf. \defref{def:degrepervers}), we have
\begin{eqnarray*}
\ov{p}(T)
&\geq&
\|\omega\|_{T}\geq \|\omega_{\tc\sigma}\|_{\ell}
=
\max\{\|\1_{F}\otimes\gamma_{F}\|_{\ell},\,
\|\1_{\tp\ast F}\otimes \gamma'_{F}\|_{\ell},\,
\|\1_{\tp}\otimes \gamma_{\emptyset}\|_{\ell}\mid F\fa\tc\Delta_{0}\}
\\
&\geq&
\max\{\|\1_{\tp\ast F}\otimes \gamma'_{F}\|_{\ell}\mid F\fa\tc\Delta_{0}\},
\end{eqnarray*}
where the equality uses (\ref{equa:degrepervers}) and (\ref{equa:omegacsigma}).
We develop this expression by distinguishing two cases.
\begin{itemize}
\item Let $\ell\neq n+1$. 
By definition, for any face $F\fa \rc\Delta_{0}$, we have the equality
$\|\1_{\tp\ast F}\otimes \gamma'_{F}\|_{\ell}=
\|\1_{F}\otimes \gamma'_{F}\|_{\ell}$,
if $F\neq \tv_{0}$, and
$\|\1_{\tp\ast \tv_{0}}\otimes \gamma'_{\tv_{0}}\|_{\ell}=\|\lambda_{\Delta_{0}}\otimes \gamma'_{\tv_{0}}\|_{\ell}$. 
It follows:
$$\ov{p}(T)\geq
\max\{\|\1_{F}\otimes\gamma'_{F}\|_{\ell},\,
\|\lambda_{\Delta_{0}}\otimes \gamma'_{\tv_{0}}\|_{\ell}\mid \tv_{0}\neq F\fa \tc\Delta_{0}\}
=\|H(\omega)_{\sigma}\|_{\ell},$$
where the last equality uses  (\ref{equa:degrepervers}) and (\ref{equa:leH}). We deduce
$\ov{p}(T)\geq \|H(\omega)\|_{T}$.
\item Let $\ell=n+1$. In this case we have
\begin{eqnarray*}
\|H(\omega)_{\sigma}\|_{n+1}
&=&
\max\{\|\1_{\tp\ast F}\otimes \gamma'_{F}\|_{n+1},\,
\|\lambda_{\Delta_{0}}\otimes \gamma'_{\tv_{0}}\|_{n+1}\mid \tv_{0}\neq F\fa\tc\Delta_{0}\}\\
&=&
\max\{|\gamma'_{F}|,\,|\gamma'_{\tv_{0}}|\mid F\fa\Delta_{0}\}
\leq |\omega|-1\leq \ov{p}(\tw),
\end{eqnarray*}
where the first equality uses (\ref{equa:degrepervers}) and (\ref{equa:leH}).
It follows
$\ov{p}(\tw)\geq \|H(\omega)\|_{\tw}$.
\end{itemize}
It has been shown $\|H(\omega)\|\leq \ov{p}$. The property
$\|\delta H(\omega)\|\leq \ov{p}$ 	is deduced using
(\ref{equa:deltaH}).
\end{proof}


The determination of the cohomology of a cone follows from the properties of  $ f, g $ and $ H $.

\begin{proof}[Proof of \thmref{prop:coneTW}]
(a) From Propositions \ref{prop:lef}, \ref{prop:leg} and \ref{prop:leH}, we deduce that
the map $g\colon \tau_{\leq \ov{p}(\tw)}\lau \tN * {\ov{p}}{\rc X;R}\to \tau_{\leq \ov{p}(\tw)}\lau \tN*{\ov{p}}{X\times ]0,\infty[;R}$
is a quasi-isomor\-phism. We know from \thmref{prop:isoproduitR}, that the inclusion $I_{1}\colon X\to  X\times ]0,\infty[$
induces a quasi-isomorphism. It remains to prove  $I_{1}^*\circ g =\iota^*$.
The involved stratifications on $X \times ]0,\infty[$ are different in these two quasi-isomorphisms but the cohomologies are the same  (see subsection \ref{SF1}).

For this, consider $\omega\in \tau_{\ov{p}(\tw)}\lau \tN * {\ov{p}}{\rc X;R}$
and
$\tau\colon \Delta_{\tau}\to X$ a regular simplex. 
By definition, we have
$(\iota^*\omega)_{\tau}=\omega_{\iota\circ\tau}$. Notice
$\iota\circ \tau
=\tc(I_{1}\circ\tau)\circ\delta_{0}$, where $\delta_{0}(x)=0 \cdot \tp+1 \cdot x$. It then follows
$$(\iota^*\omega)_{\tau}=
\omega_{\tc(I_{1}\circ\tau)\circ\delta_{0}}=\delta^*_{0}\,\omega_{\tc(I_{1}\circ\tau)}
=_{(\ref{equa:leg})}\gamma_{\tv_{0}}
=g(\omega)_{I_{1}\circ\tau}
=I_{1}^*(g(\omega))_{\tau}.$$

\medskip
To prove part (b), we consider a cocycle $\omega\in \lau \tN k {\ov{p}}{\rc X;R}$ with $k>\ov{p}(\tw)$.
Following  \propref{prop:leH}, we have
$\delta H(\omega)=\omega-f(g(\omega))$ and it suffices to establish the equality$f(g(\omega))=0$.

We prove it by contradiction, assuming that there is a regular simplex
$\sigma\colon \Delta_{\sigma}\to \rc X$, such that $f(g(\omega))_{\sigma}\neq 0$. Following (\ref{equa:fg}), 
this implies
$\gamma_{\tv_{0}}\neq 0$. 
We get a contradiction,
\begin{eqnarray*}
k> \ov{p}(\tw)
&\geq&
\|f(g(\omega))\|_{\tw} 
\geq
\|f(g(\omega))_{\tc\sigma}\|_{n+1}
=
\|\lambda_{\tc(\tp\ast\Delta_{0})}\otimes \gamma_{\tv_{0}}\|_{n+1}\\
&=&
\|\lambda_{\tp\ast\Delta_{0}}\otimes \gamma_{\tv_{0}}\|_{n+1}
=|\gamma_{\tv_{0}}|=k.
\end{eqnarray*}
\end{proof}

The proof of the topological  independence Theorem of \secref{15} is based on the following excision theorem.

\subsection{Relative cohomology}\label{RelCoho}
Let $X$ be a filtered space  and $Y \subset X$ a subspace endowed with the induced stratification. Consider  a perversity $\ov{p}$ on $X$. We also denote by $\ov p$ the induced perversity on $Y$.
The \emph{complex of relative $\ov{p}$-intersection cochains} is the direct sum
$\lau \tN * {\ov p}{X,Y;R} =  \lau \tN * {\ov p}{X;R}  \oplus \lau \tN {*-1} {\ov p}{Y;R} $, endowed with the differential $D(\alpha,\beta) = (d\alpha, \alpha - d\beta)$. 
Its homology is the
\emph{relative blown-up $\ov{p}$-intersection cohomology of the perverse pair  $(X,Y,\ov{p})$,}
denoted by 
$\lau \IH{*}{\ov{p}}{X,Y;R}$.

By definition, we have a long exact sequence associated to the perverse pair $(X,Y,\ov{p})$:
\begin{equation}\label{equa:suiterelative2}
\ldots\to
\lau \IH {i}{\ov{p}}{X;R}\stackrel{\i^*}{\to}
\lau \IH {i}{\ov{p}}{Y;R}\to
\lau \IH{i+1}{\ov{p}}{X,Y;R}\stackrel{\pr^*}{\to}
\lau \IH{i+1}{\ov{p}}{X;R}\to
\ldots,
\end{equation}
where $\pr \colon \lau \tN * {\ov p}{X,Y;R}  \to \lau  \tN * {\ov p}{X;R} $ is defined by $\pr(\alpha,\beta) = \alpha$ and $\i^*$ comes from the restriction map $\i \colon \lau \tN * {\ov p}{X;R} \rightarrow  \lau \tN * {\ov p}{Y;R} $.

\begin{proposition}\label{prop:Excisionhomologie}
Let  $(X,\ov p)$ be a paracompact perverse space. 
If $F$ is a closed subset of $X$ and $U$ an open subset of $X$ with $F\subset U$, 
then the natural inclusion  $(X\backslash F,U\backslash F)\hookrightarrow (X,U)$ induces the isomorphism 
$$  \lau \IH * {\ov{p}} {X,U ;R}  \cong \lau \IH * {\ov{p}}{ X\backslash F,U\backslash F ;R}.
$$
\end{proposition}

\begin{proof}
For the sake of simplicity we write the complexes $\bi {\cal A} *= \lau \tN {*,\mathcal U} {\ov{p}} {X;R}$, 
$\bi {\cal B} * = \lau \tN {*,\mathcal U_1} {\ov{p}} {X\menos F;R}$, $\bi {\cal C} * = \lau \tN {*,\mathcal U_2} {\ov{p}} {U;R}$ and $\bi {\cal E} *=\lau \tN {*} {\ov{p}} {U\backslash F;R}$, the terms of the Mayer-Vietoris exact sequence  of \thmref{thm:MVcourte}, applied to the cover $\cal U = \{ X\menos F,U\}$. Associated to this sequence we have the exact sequence
$$
0 \to (\bi {\cal A} * \oplus \bi {\cal C} {*-1},D) \xrightarrow[]{f}
 (\bi {\cal B} *\oplus \bi {\cal C} * \oplus \bi {\cal C} {*-1} \oplus \bi {\cal E} {*-1}, D') \xrightarrow[]{g} (\bi {\cal E} {*} \oplus \bi {\cal E} {*-1},D) \to 0,
$$
where 
$D(a,c)  = (da ,a -dc)$, $f(a,c) = (a,a,c,0)$, $D'(b,c_0,c_1,e)=(db,dc_0,c_0-dc_1,b-c_0-de)$, $g(b,c_0,c_1,e) =(b-c_0,e)$ and $D(e,e') = (de,e-de')$.
The right-hand complex is acyclic since $D(e,e')=0 \Rightarrow (e,e') = D(e',0)$.
 The chain map $h \colon (\bi {\cal B} * \oplus \bi {\cal E} {*-1},D) \to
 (\bi {\cal B} *\oplus \bi {\cal C} * \oplus \bi {\cal C} {*-1} \oplus \bi {\cal E} {*-1}, D') ,
 $
 defined by 
 $h (b,e) = (b,0,0,e)$, where $D(b,e) = (db , b -de)$, is a quasi-isomorphism, since
\begin{eqnarray*}
0=  D'(b,c_0,c_1,e) &\Rightarrow&  (b,c_0,c_1,e)  = h(b,e+c_1) + D'(0,c_1,0,0) \\
 h(b',e') =  D'(b,c_0,c_1,e) &\Rightarrow& (b', e')  = D(b,e+c_1).
 \end{eqnarray*}
Let $k\colon 
 (\bi {\cal A} * \oplus \bi {\cal C} {*-1},D)
 \to 
 (\bi {\cal B} * \oplus \bi {\cal E} {*-1},D)$,
defined by $k(a,c) =(a,c)$.
If $D(a,c)=0$, then $h(k(a,c)) -f(a,c) = (0,-a,-c,c) = D'(0,-c,0,0)$. Therefore, $k$  is a quasi-isomorphism.
 Applying \thmref{thm:Upetits}  and the definition of relative blown-up $\ov{p}$-intersection cohomology  we get the result.\end{proof}
 
 The following result comes directly from the  Mayer-Vietoris sequence and the long exact sequence associated to a pair.

\bp\label{cor:homologieconerel}
Let $X$  be a compact filtered space.
Let us consider the cone, $\rc X$, with apex $\tw$, endowed with the cone filtration.
Consider a perversity  $\ov p$ on $\rc X$ and denote also by $\ov p$ the induced perversity on $X$.
Then
$$
 \lau \IH j{\ov{p}}{\rc X, \rc X \backslash \{ \tw\} ;R}=\left\{
\begin{array}{cl}
\lau \IH {j-1}{\ov{p}} {X;R}&\text{ if } j\geq \ov{p}( \tw) +2,
\\[.2cm]
0&\text{ if }  j\leq  \ov{p}( \tw) +1.\\[.2cm]%
\end{array}\right.$$
\ep

\section{Comparison between intersection  cohomologies. \thmref{thm:TWGMcorps}.}\label{subsec:lesdeuxcohomologies}

If $ R $ is a field, the existence of
a quasi-isomorphism between the cochain complexes,
$\lau C {*} {D\ov{p}}{X;R}$ and $\lau \tN *{\ov{p}} {X;R}$,
has been established in {\cite[Theorem B]{CST1}} 
for GM-perversities in the framework of filtered face sets.
We generalize this result using  general perversities without the condition $D\ov p \leq \ov t$. In order to eliminate this last constraint we use tame intersection cohomology and homology instead of intersection cohomology and homology.

Recall that 
a perverse CS set $(X,\ov{p})$ is \emph{locally $(D\ov{p},R)$-torsion free} if, for each singular stratum $S$ and each $x\in S$ with link $L$, we have that 
$\lau \gT {D\ov{p}} { \ov{p}(S)}{L;R}=0$, i.e., the torsion $R$-submodule of $\lau \gH {D\ov{p}} { \ov{p}(S)}{L;R}$,  vanishes (see \cite{MR699009}).

 \begin{theorem}\label{thm:TWGMcorps}
 Let $ (X, \ov p) $ be a paracompact separable perverse CS set 
and  $ R $  a Dedekind ring.
The blown-up  
and the tame intersection complexes, $\lau \tN*{\ov{p}}{X;R}$ and $\lau \gC{*} {D\ov{p}}{X;R}$, are related by a quasi-isomorphism
if one of the following hypotheses is satisfied.
\begin{enumerate}[1)]
\item The ring $R$ is a field.
\item The space $X$ is a locally $(D\ov{p},R)$-torsion free pseudomanifold.
\end{enumerate}
Under any of these assumptions, there exists an isomorphism
$$\lau \IH *{\ov{p}}{X;R}\cong \lau \gH*{D\ov{p}}{X;R}.$$
\end{theorem}

This isomorphism   exists for the top perversity $ \ov {t} $, for any  CS set and any
Dedekind ring (see \cite[Remark 1.51]{CST1}). This Theorem has a particular form in the classical setting of stratified pseudomanifolds and GM-perversities.

 \begin{corollary}\label{cor:TW corps}
 Let $\ov p$ be a GM-perversity, with \emph{complementary perversity} $\ov q$, i.e., $\ov q(i) = i-2-\ov p (i)$ for $i\geq 2$. We consider   a Dedekind ring $R$. 
 Let $ X$ be  a paracompact separable classical pseudomanifold.
The blown-up  
and the intersection complexes, $\lau \tN*{\ov{p}}{X;R}$ and $\lau C{*} {\ov{q}}{X;R}$, are related by a quasi-isomorphism 
if  $X$ is  locally $(\ov{q},R)$-torsion free. So, we have an isomorphism
$\lau \IH *{\ov{p}}{X;R}\cong \lau H*{\ov{q}}{X;R}.$
\end{corollary}

The proof of \thmref{thm:TWGMcorps}, similar to that used in \cite[Théorème A]{CST3}, meets the scheme
  used in \cite[Theorem 10]{MR800845}, 
\cite[Lemma 1.4.1]{MR2210257}, \cite[Section 5.1]{LibroGreg}. 

\begin{proposition}\label{prop:supersuperbredon}
Let $\cal F_{X}$ be the category whose objects are (stratified homeomorphic to) open subsets
of a given paracompact and separable CS set $X$ and whose morphisms are  stratified homeomorphisms and inclusions.
Let  $\cal Ab_{*}$ be the category of graded abelian groups. Let $F^{*},\,G^{*}\colon \cal F_{X}\to \cal Ab$
be two functors and
 $\Phi\colon F^{*}\to G^{*}$ a natural transformation satisfying
 the conditions listed
below.
\begin{enumerate}[(i)]

\item $F^{*}$ and $G^{*}$ admit exact Mayer-Vietoris sequences and the natural transformation $\Phi$ 
 induces a commutative diagram between these sequences,

\item If $\{U_{\alpha}\}$ is a disjoint collection of open subsets of $X$  and $\Phi\colon F_{*}(U_{\alpha})\to G_{*}(U_{\alpha})$ is an isomorphism for each $\alpha$, then $\Phi\colon F^{*}(\bigsqcup_{\alpha}U_{\alpha})\to G^{*}(\bigsqcup_{\alpha}U_{\alpha})$  is an isomorphism.

\item If $L$ is a compact filtered space such that 
$X$  has an open subset  stratified homeomorphic
to $\R^i\times \rc L$ and, if
$\Phi\colon F^{*}(\R^i\times (\rc L\backslash \{\tv\}))\to G^{*}(\R^i\times (\rc L\backslash \{\tv\}))$
is an isomorphism, then so is
$\Phi\colon F^{*}(\R^i\times \rc L)\to G^{*}(\R^i\times \rc L)$. Here, $\tv$ is the apex of the cone $\rc L$.

\item If $U$ is an open subset of X contained within a single stratum and homeomorphic
to an Euclidean space, then $\Phi\colon F^{*}(U)\to G^{*}(U)$ is an isomorphism.
\end{enumerate}
Then $\Phi\colon F^{*}(X)\to G^{*}(X)$ is an isomorphism.
\end{proposition}

\begin{proof} 
Let $U$  be an open subset of $X$. We apply \lemref{lem:bredon}, taking for $P(U)$ the property
$$``\Phi\colon F^*(U)\to G^*(U) \text{ is an isomorphism.''}$$
Notice that the CS set $X$ is  a separable, locally compact and metrizable space \cite[Proposition 1.11]{CST3}. We proceed by induction on the dimension of $ X $, the result is immediate when $ \dim X=0$, that is, a discrete space. The inductive step  occurs in two stages.

$\bullet$ \emph{We first show $P(Y)$ for any open subset $Y$ of a fixed conical chart.}
We can suppose $Y=\R^m\times \rc L$. If $L=\emptyset$, we apply (iv). Let us suppose $L\ne \emptyset$.
We consider the basis $\cal V$ of open sets of $Y$, composed of
\begin{itemize}
\item open subsets  $V$ of $Y$, 
with $\tv \not\in V$, which are CS sets of strictly lower dimension than $ X $,
\item open subsets $V=B\times \rc_{\varepsilon}L$, where $B \subset \R^m$ is an open $m$-cube, $\varepsilon >0$ and
$\rc_{\epsilon}=(L\times [0,\varepsilon[)/(L\times \{0\})$.
\end{itemize}
This family is closed by finite intersections and it verifies the assumptions of \lemref{lem:bredon}.
For a), it suffices to apply the induction hypothesis together with (iii) and (iv).
The property b) derives from (i) and property c) from (ii).

$\bullet$ \emph{We prove $P(X)$} by
considering the open basis of $X$ consisting of open subsets of conical charts. This family  is stable  by finite intersections
and verifies the assumptions of \lemref{lem:bredon}.
The condition a) is a consequence of the first step. The two other conditions derive from (i) and (ii), as previously.
\end{proof}

\begin{lemma}\label{lem:bredon}
Let $ X $ be a locally compact topological space, metrizable and separable. We are given an open basis
of $ X $, $ \cal U = \{U_{\alpha} \} $, closed by finite intersections, and a statement $P(U) $ on  open subsets of $X$ satisfying the following three properties.
\begin{enumerate}[a)]
\item  The property $P(U_{\alpha}) $ is true for all $ \alpha $.

\item  If $ U $, $ V $ are open subsets of $ X $ for which  properties $P(U) $, $ P (V) $ and $ P (U \cap V) $ are true, then
$ P (U \cup V) $ is true.

\item If $ (U_{i})_{i \in I} $ is a family of open subsets of $ X $, pairwise disjoint, verifying the property $ P (U_{i}) $ for all $ i \in I $, then $ P (\bigsqcup_i U_{i}) $ is true.
\end{enumerate}
Then the property $ P (X) $ is true.
\end{lemma}

\begin{proof} 
Since the family $\cal U$ is closed by finite intersections, the properties a) and b) imply that the property $P$ is true for any finite union of elements of $\cal U$.

Since the space $X$ is separable, metrizable and locally compact, then its Alexandroff's compactification $\hat{X}=X\sqcup\{\infty\}$ 
is  metrizable (cf. \cite[Exercise 23C]{MR0264581}) and there exists a proper map, $f\colon X\to [0,\infty[$, defined by $f(x)=1/d(\infty,x)$.
We associate to $f$ a countable family of compact subsets, $A_{n}=f^{-1}([n,n+1])$. 
Each $A_{n}$ possesses a finite cover consisting  of open subsets $U_{\alpha}\in\cal U$, included in $f^{-1}(]n-1/2,n+3/2[)$;
we denote by $U_{n}$ the union of the  elements of this cover.
Since the open subset $ U_ {n} $ is a finite union of elements of $ \cal U$, then the
property $ P (U_ {n}) $ is true.

Let $U_{\mathrm{even}}=\bigsqcup_{n}U_{2n}$
and let
$U_{\mathrm{odd}}=\bigsqcup_{n}U_{2n+1}$.
The hypothesis  c) implies that $P(U_{\mathrm{even}})$ and  $P(U_{\mathrm{odd}})$ are true.
Furthermore, since the intersection $U_{2n}\cap U_{2n+1}$ is a finite union of elements $U_{\alpha}$ of $\cal U$,
then the property $P(U_{2n}\cap U_{2n+1})$ is true.
From  $U_{\mathrm{even}}\cap U_{\mathrm{odd}}=\bigsqcup_{n} U_{2n}\cap U_{2n+1}$
and from the hypothesis c), 
we deduce that the property
$P(U_{\mathrm{even}}\cup U_{\mathrm{odd}})$ is true.
The conclusion then follows from
$U_{\mathrm{even}}\cup U_{\mathrm{odd}}=\bigcup_{n}U_{n}\supset \bigcup_{n}A_{n}=X$.
\end{proof}

\bigskip

Let  $X$ be  a filtered space.
We construct the map
$$\chi\colon\Hiru \tN *{X;R}\to \Hiru \gC* {X;R}$$ as follows.
If $\omega\in\Hiru \tN *{X;R}$ and if $\sigma\colon \Delta=\Delta_{0}\ast\dots\ast\Delta_{n}\to X$
is a regular simplex, 
we set:
$$\chi(\omega)(\sigma)=
 \omega_{\sigma}(\tDelta).
 $$
\begin{proposition}\label{prop:TWGM}
Let $(X,\ov{p})$ be a perverse space.
The operator $\chi$ induces a chain map,
$\chi\colon \lau \tN * {\ov{p}}{X;R}\to  \lau \gC*{D\ov{p}}{X;R}$.
\end{proposition}

\begin{proof}
Recall that $\lau \gC*{D\ov{p}}{X;R} = \Hom (\lau \gC{D\ov{p}} * {X;R},R)$.
Since the simplices coming from $\lau \gC{D\ov{p}}*{X;R} $ are regular, then the definition of $\chi$ makes sense. This is the reason of considering tame intersection chains instead of intersection chains.

We need to prove $\chi \circ \delta = \gd^* \circ \chi$, where $\gd^*$ is the linear dual of $\gd$ (see \defref{defgd}).
Consider a cochain 
$\omega$ of $ \lau \tN * {\ov{p}}{X;R}$ and a
chain $\xi \in \lau \gC{D\ov{p}} * {X;R}$.
We prove 
$\chi(\delta\omega)(\xi) =(\gd^*\chi)(\omega)(\xi) =-(-1)^{|\omega|}\chi(\omega)(\gd\xi) $.
In fact, it suffices to prove
$\chi(\delta\omega)(\sigma) =-(-1)^{|\omega|}\chi(\omega)(\gd\sigma) $ for a regular simplex $\sigma \colon \Delta \to X$ $D\ov p$-allowable.
Let us suppose that: 
\be\label{cleim}
\omega_\sigma (\cal H_i)=0 \hbox{ for each hidden face $\cal H_i$}
\ee
(see \eqref{Hid}).
Then
\begin{eqnarray*}
\chi(\delta\omega)({\sigma})
=(\delta\omega)_{\sigma}(\tDelta)
=-(-1)^{|\omega|} \;  \omega_{\sigma}(\partial \tDelta)
=   -(-1)^{|\omega|} \;  \omega_{\sigma}(\widetilde{\gd\Delta}) =
 -(-1)^{|\omega|} \; \chi(\omega)({\gd\sigma}).\label{equa:chidelta}
\end{eqnarray*}
Let us prove the claim \eqref{cleim}. We suppose $ \omega_\sigma (\cal H_i)\ne 0 $ for some hidden face $\cal H_i$.
By definition of a hidden face, notice that $\Delta_{n-i} \ne \emptyset
$, thus there exists an $i$-dimensional    stratum $S$ of $X$ with $S \cap \im \sigma \ne \emptyset$.
Then, since $\sigma$ is $D\ov p$-allowable and $\omega$ is $\ov p$-allowable,
\begin{itemize}
\item  
$
0 \leq |\Delta|_{\leq n-i} = \|\sigma\|_S \leq \dim \sigma - \codim S + D\ov p (S)
=\dim \sigma - \ov p (S) -2$, and
\item $0 \leq |\Delta|_{>n-i}  = |\cal H_i|_{>n-i}  =  \|\1_{\cal H_i}\|_{\codim S} \leq_{(1)} \|\omega_\sigma\|_{\codim S}\leq \ov p(S)$,

\end{itemize}
where $\leq_{(1)} $ comes from  $ \omega_\sigma (\cal H_i)\ne 0 $.
We conclude that $\ov p (S) $ is finite. Adding up these two inequalities, we get
$\dim \sigma -1= \dim\Delta-1=|\Delta|_{\leq n-i} + |\Delta|_{>n-i} \leq  \|\omega_\sigma\|_S+ \|\sigma\|_S
\leq \dim\sigma-2.
$
This contradiction gives the claim \eqref{cleim} and ends the proof.
\end{proof}


\begin{proof}[Proof of \thmref{thm:TWGMcorps}]
By hypothesis, $X$ is a separable paracompact CS set.
Properties (i), (ii), (iv) of  \propref{prop:supersuperbredon} are verified. 
For the  item (iii), we consider a
filtered compact space $ L $, for which the chain map of \propref{prop:TWGM}
induces an isomorphism
$\chi^*\colon \lau \IH * {\ov{p}}{\R^i\times L\times ]0,\infty[;R}\xrightarrow[]{\cong} 
\lau \gH *  {\ov{q}}{\R^i\times L\times ]0,\infty[;R}$.
Since the space $\R^i\times L\times ]0,\infty[$ is an open subset of 
$\R^i\times\rc L$, 
the following diagram commutes.
\begin{equation}\label{equa:dualitefin}
\xymatrix{
\lau \IH *{\ov{p}}{\R^i\times L\times ]0,\infty[;R)}
\ar[r]^{\chi^*}&
\lau \gH*{D
\ov{p}} {\R^i\times L\times ]0,\infty[;R}
\\
\lau \IH *{\ov{p}}{\R^i\times\rc L;R} 
\ar[u]\ar[r]^-{\chi^*_{\R^i\times \rc L}}&
\lau \gH  *{D\ov{p}}{\R^i\times \rc L;R)}\ar[u]
}
\end{equation}
Recall the blown-up intersection cohomology computation (see \thmref{prop:isoproduitR} and \thmref{prop:coneTW}):
\begin{eqnarray*}
\lau \IH k{\ov{p}}{\R^i\times\rc L;R} &=& \lau \IH k {\ov{p}}{\rc L;R}
\\ &=&
\left\{
\begin{array}{ccl}
\lau \IH k {\ov{p}}{L;R} =\lau \IH k{\ov{p}}{\R^i\times  L\times ]0,\infty[ ;R} 
&\text{if}&
k\leq \ov{p}(\tw),\\
0
&\text{if}&
k>\ov{p}(\tw).
\end{array}\right.
\end{eqnarray*}
The tame intersection cohomology is determined in \cite[Chapter 7]{LibroGreg}:
\begin{eqnarray*}
\lau \gH k{D
\ov{p}} {\R^i\times \rc L;R} &= &\lau \gH k{D\ov{p}}{\rc L;R}
\\ &=&
\left\{
\begin{array}{ccl}
\lau \gH k {D\ov{p}}{L;R}=  \lau \gH k{D
\ov{p}} {\R^i\times L \times ]0,\infty[;R}
&\text{if}&
k\leq \ov{p}(\tw),\\[,2cm]
\ext(\lau \gH{D\ov{p}} {k-1} {L;R},R)
&\text{if}&
k=  \ov{p}(\tw) +1,\\[,2cm]
0
&\text{if}&
k>  \ov{p}(\tw) +1.
\end{array}\right.
\end{eqnarray*}
The involved stratifications on $  L $ are different in these
computations but the cohomologies are the same  (see subsections \ref{SF1}, \ref{SF}).

In case 1) we have 
$\ext(\lau \gH{D\ov{p}} {\ov{p}(\tw)} {L;R},R)=0$ since  $R$ is a field.
In case 2) the link $L$ is a compact pseudomanifold.
According to  \cite[Section 5.7]{LibroGreg}, the tame intersection homology of $ L $ is finitely generated.
Since the torsion module of  $\lau \gH {D\ov{p}}  {\ov{p}(\tw)}{L;R}$ is zero by hypothesis, we deduce
$\ext(\lau \gH {D\ov{p}}  {\ov{p}(\tw)}{L;R},R)=0$.

So, in both cases, the vertical maps of \eqref{equa:dualitefin} are isomorphisms.
Since $\chi^*$ is an isomorphism then $\chi^*_{\R^i\times \rc L}$ is also an isomorphism.
\end{proof}

For some particular perversities the blown-up intersection cohomology  is described in terms of  the ordinary cohomology, denoted by $\Hiru H * {-;R}$.

\bp\label{PartPer}
Let $X$ be a paracompact separable CS set. Then

\begin{enumerate}[(a)]
\item $\lau \IH * {\ov 0} {X;R} \cong \Hiru H * {X;R}$ if $X$ is normal.

\item $\lau \IH * {\ov p} {X;R} \cong \Hiru H * {X, \Sigma;R}$ if $\ov p <0$.

\item $\lau \IH * {\ov p} {X;R} \cong \Hiru H * {X\menos\Sigma;R}$ if $\ov p >\ov t$.

\end{enumerate}
These isomorphisms preserve the cup product.
\ep
\bpr
The complex of singular chains  (resp. relative chains) with coefficients in $R$, is denoted by  $\hiru S * {X;R}$ (resp.$\hiru S * {X,\Sigma;R}$). It computes the ordinary homology $\hiru H * {X;R}$ (resp. $\hiru H * {X,\Sigma;R}$).

(a)
  Let
$
F\colon  \Hiru S*{X;R} \to  \lau \tN * {\ov 0} {X;R}
$
be
the differential operator 
defined by $F (\omega)_\sigma = \mu^*(\sigma^*(\omega)),$
for any regular filtered simplex $\sigma\colon \Delta\to X$. 	
The operator $\mu^*$ is associated to the $\mu$-amalgamation of $\Delta$ (see \ref{defmu*}.1). 
Applying  \propref{CMAmal} repeatedly, we conclude that $F$ is a chain map commuting with the cup product.
By definition of $\mu^*$ we have $||F (\omega)_\sigma ||_\ell \leq ||\sigma^*(\omega)||_0 =0$. The operator $F$ is therefore well defined.

We consider the statement: ``The operator $F$  is a quasi-isomorphism'' and we prove it by using \propref{prop:supersuperbredon}. We verify the four properties (i)-(iv) of this Proposition.
\begin{itemize}
\item[(i)]
It is well known that the functor $\Hiru S * {-;R}$ verifies the Mayer-Vietoris property.
It has been proven in  \thmref{thm:MVcourte} that     the functor  $\lau \tN*{\ov{0}}{-;R}$ also 
verifies Mayer-Vietoris property. One easily checks that $F$ induces a commutative diagram between these sequences.

\item[(ii),(iv)] Straightforward.

\item[(iii)] It is well known that
$H^*(\R^i\times \rc L;R)=H^0(\R^i\times \rc L;R)=R$ (constant cochains).
From Theorems \ref{prop:isoproduitR}, \ref{prop:coneTW}, we have
$\lau \IH* {\ov{0}}{\R^i\times \rc L;R}$ =$\lau \IH  0 {\ov{0}}{L;R}.$
Let us consider a basis point $x_0 \in  L \backslash \Sigma$, which is connected from   \cite[Lemma 2.6.3]{LibroGreg}. Also, for any point $x_1 \in L$ there exists a regular simplex $\sigma \colon [0,1] \to L$ going from $x_1$ to $x_0$.
If $\omega\in \lau \tN 0 {\ov{0}} {L;R}$ is a cocycle, the cochain $\omega_{\sigma}$ takes the same value on $x_1$ and on $x_0$. This implies $\lau \IH 0 {\ov{0}}{L;R}=R$.

\end{itemize}

(b)
 Consider the previous operator 
$
F\colon  \Hiru S*{X;R} \to  \lau \tN * {} {X;R}
$
defined by $F (\omega)_\sigma = \mu^*(\sigma^*(\omega)),$
for any regular filtered simplex $\sigma\colon \Delta\to X$. 	
We know that it is a chain map commuting with the cup product.
If $\omega$ vanishes on $\Sigma$ then $\|F (\omega)_\sigma\|_\ell =-\infty$ for any $\ell \in \{1, \ldots,n\}$. So, $
F\colon  \Hiru S*{X,\Sigma;R} \to  \lau \tN * {\ov p} {X;R}
$
is well defined.
We consider the statement: ``The operator $F$  is a quasi-isomorphism'' and we prove it 
as in the previous case. In fact, the only item to prove is (iii).
It comes from
$
\Hiru  H* {\R^i \times \rc L, \R^i \times \rc \Sigma_L;R} = 0=\lau \IH * {\ov p}{\R^i \times \rc L;R},
$
induced by  Theorems \  \ref{prop:isoproduitR}, \ref{prop:coneTW}.

(c) Consider the natural restriction $\gamma \colon \lau \tN * {\ov p} {X;R} \to \lau \tN * {\ov p} {X\menos \Sigma;R}  =\Hiru S * {X\menos \Sigma;R}$. 
It is a chain map commuting with the cup product. 	
We consider the statement: ``The operator $\gamma$  is a quasi-isomorphim'' we proceed as in the previous case. In fact, the only property to prove is the item (iii) of  \propref{prop:supersuperbredon}.
We know that $
\Hiru H * {\R^i \times \rc L  \menos \Sigma_{\R^i \times \rc L  };R} 
=
\Hiru H * {\R^i \times  (L\backslash\Sigma_L) \times ]0,\infty[;R} = \Hiru H * {L\backslash \Sigma_L;R}
=
\Hiru H * {\R^i \times (\rc L \menos \{ \tv \}) \menos \Sigma_{\R^i \times (\rc L \menos \{ \tv \}) };R} 
 $. On the other hand, we have 
$\lau  \IH * {\ov p} {\R^i \times \rc L;R} =
\lau \IH * {\ov p}  { L;R}
=
\lau \IH * {\ov p}  { \R^i \times (\rc L \menos \{ \tv \});R}
$
(cf. Theorems \ref{prop:isoproduitR}, \ref{prop:coneTW}). It suffices to apply the hypothesis of (iii).
\epr

\section{Topological invariance. \thmref{inv}.}\label{15}

We prove in this section that the blown-up intersection cohomology is a topological  invariant, working with CS sets and GM-perversities.
We follow  the procedure of King \cite{MR800845} (see also \cite{LibroGreg}).

\subsection{Intrinsic filtration}
A key ingredient in the King's proof of the topological invariance is the  intrinsic filtration of a CS set $X$. It was introduced in  \cite{MR800845}, crediting to Sullivan, see also  \cite{MR0478169}. We refer the reader to \cite{LibroGreg} where there is an exhaustive study of this notion.

Let $X$ be an $n$-dimensional filtered space. Two points $x,y\in X$  are \emph{equivalent} if there exists a homeomorphism $h \colon (U,x) \to (V,y)$, where $U,V$ are neighborhoods of $x,y$ respectively.
The local structure of a CS set implies that two points of the same stratum are necessarily equivalent.
Hence the equivalence
classes are unions of components of strata of X. 

Let $Y_i$ be the union of the equivalence
classes which only contain components of strata of dimension $\leq i$. The filtration $Y_0 \subset Y_1 \subset \cdots \subset Y_n=X$ is the \emph{intrinsic filtration}
 of $X$, denoted by  $X^*$. It is proved that $X^*$ is an $n$-dimensional  CS set \cite{MR800845}.
 Since the filtration $X^*$ coarsens the original filtration of $X$ then the identity map $\nu \colon X \to X^*$ is a stratified map, called the \emph{intrinsic aggregation} (see \defref{def:applistratifieeforte}).
 Recall that, given a perversity $\ov p$, the map  $\nu$ induces a morphism $\nu^* \colon \lau \IH * {\ov p}{X^*;R} \to
\lau \IH * {\ov p}{X;R} $ (see \thmref{MorCoho}).

\bp\label{prop:local}
Let $\ov p$ be a GM-perversity. 
Let $X$ be a paracompact CS set with no codimension
one strata.
  Let $S$ be a stratum of $X$ and $(U,\varphi)$ a conical chart of a point $x\in S$. 
 If the intrinsic aggregation  $\nu\colon X\to X^*$ induces the isomorphism
 $\nu_{*}\colon  \lau \IH{*} {\ov{p}}{(U\backslash S)^*;R} \xrightarrow[]{\cong} \lau \IH {*} {\ov{p}}{U\backslash S;R}$
then it also induces the  isomorphism,
 $$\nu_{*}\colon  \lau \IH{*} {\ov{p}}{U^*;R} \xrightarrow[]{\cong} \lau \IH {*} {\ov{p}}{U;R}.$$
\ep

\begin{proof}We analyze the local structure of $X$ and $X^*$. Without loss of generality, we can suppose
$U = \R^k \times \rc W$, where $W$ is a compact filtered space (possibly empty) and  $S \cap U= \R^k \times \{ \tw\}$, $\tw$ being the apex of the cone $\rc W$.
Following \cite[Lemma 2 and Proposition 1]{MR800845}, 
there exists a  homeomorphism
\begin{equation}\label{equa:homeo}
h\colon ( \R^k \times \rc W)^*\xrightarrow[]{\cong} \R^m\times \rc L,
\end{equation}
which is also a stratified map, where $L$ is a compact filtered space  (possibly empty)
and $m\geq k$. 
Moreover,  $h$ verifies
\begin{equation}\label{equa:hetcone}
h(\R^k\times \{\tw\} ) \subset \R^m\times \{\tv\} \text{ and }
h^{-1}(\R^m\times \{\tv\})=\R^k\times \rc A,
\end{equation}
where $A$ is an $(m-k-1)$ sphere and  $\tv$ is the apex of the cone $\rc L$.
Now, the hypothesis and the conclusion of the Proposition become
\begin{equation}\label{equa:hyp}
h \colon 
\lau \IH * {\ov{p}}{\R^m  \times \rc L \backslash h(\R^k  \times \{\tw\});R}
\xrightarrow[]{\cong}
\lau \IH *{\ov{p}}{\R^k  \times \rc W \backslash (\R^k  \times \{\tw\});R}
\end{equation}
and
\begin{equation}\label{equa:conc}
h\colon 
\lau \IH * {\ov{p}}{\R^m  \times \rc L ;R}
\xrightarrow[]{\cong}
\lau \IH * {\ov{p}}{\R^k  \times \rc W;R }.
\end{equation}
The existence of the homeomorphism (\ref{equa:homeo}) implies $k+s=m+t$, and $s\geq t$, since $m\geq k$, where $s = \dim W$  and  $t=\dim L$.  

The result is clear when  $s=-1$. So, we can suppose  $s\geq 0$ which implies that the stratum $\R^k\times \{\tw\}$ is singular. Since $X$ has no codimension one strata, we indeed have  $s \geq 1$.
When $t=-1$ then $L=\emptyset$ , $s =\dim A$ and therefore
\begin{eqnarray*}
\lau \IH * {\ov{p}}{\R^k  \times \rc W;R } 
&\stackrel{(\ref{equa:hetcone})}{\cong}&
\lau \IH * {\ov{p}}{ \R^k \times \rc A ;R}
\cong_{(2)}
\lau \IH * {\ov{p}
}{ \rc A;R }
\cong_{(3)} 
\left\{
\begin{array}{ll}
R & \hbox{if } *=0\\
0 & \hbox{if not}
\end{array}
\right. \\
&\cong&
\lau \IH * {\ov{p}}{\R^m  \times \rc L;R } .
\end{eqnarray*}
Here, $\cong_{(2)}$ is 
\thmref{prop:isoproduitR} and  $\cong_{(3)}$ comes from  
\begin{itemize}
\item \thmref{prop:coneTW}, and
\item $0 \leq \ov p(s+1) \leq \ov t (s+1) = s-1 < \dim A$ since $\ov p$ is a GM-perversity and $s \geq 1$.
\end{itemize}
So, we can suppose  $t\geq 0$ which implies that the stratum $\R^m\times \{\tv\}$ is singular. 

We establish (\ref{equa:conc}) in two different cases.

\smallskip
$\bullet$ \emph{First case}: $i >  \ov p(s+1)$.  

 \thmref{prop:coneTW} gives
$\lau \IH i {\ov{p}}{\R^k \times \rc W;R}=0$ and then we have to prove that  
\newline $\lau \IH i {\ov{p}} {\R^m \times \rc L;R}=0$. 
Since the stratum  $\R^m \times \{ \tv\}$  is singular and $\ov p$ is a GM-perversity then 
$
i> \ov{p}(s+1)
 \geq \ov{p}(t+1).$
 Now, \thmref{prop:coneTW} applied to
$\lau \IH i{\ov{p}} {\R^m \times \rc L;R}$ gives the result.

\smallskip
$\bullet$ \emph{Second case} : $i \leq \ov{p}(s+1)$.  
We have the isomorphisms
\begin{eqnarray}\label{equa:cas3debut}
\lau \IH i {\ov{p}}{\R^m\times \rc L\backslash h(\R^k\times \{\tw\});R}
\cong_{(1)}
\lau \IH {i}{\ov{p}}{\R^k\times\rc W\backslash (\R^k\times \{\tw\});R}  \cong \\
\lau \IH{i}{\ov{p}}{\R^k\times (\rc W\backslash \{\tw\});R} 
\cong 
\lau \IH{i}{\ov{p}}{\R^k\times ]0,1[\times  W;R} \cong_{(2)}
\lau \IH {i}{\ov{p}}{W;R},\nonumber&&
\end{eqnarray}
where $\cong_{(1)}$ is the hypothesis (\ref{equa:hyp}) and $\cong_{(2)}$ is \thmref{prop:isoproduitR}.
Write $h(\R^k\times\{\tw\})=B\times \{\tv\}\subset \R^m\times \{\tv\}$, with  $B$ closed.
We obtain a new sequence of isomorphisms
\begin{eqnarray}\label{equa:tropdeL}
\lau \IH {i}{\ov{p}}{\R^m\times \rc L\backslash h(\R^k\times \{\tw\}),\R^m\times \rc L\backslash \R^m\times\{\tv\};R}
\cong &&\\
\lau \IH {i}{\ov{p}}{(\R^m\times \rc L )\backslash (B\times\{\tv\}) ,\R^m\times (\rc L\backslash \{\tv\});R}
\cong_{(1)}\nonumber &&\\
\lau  \IH {i} {\ov{p}}{(\R^m\backslash B)\times \rc L, (\R^m\backslash B)\times (\rc L\backslash \{\tv\});R}
\cong
\lau \IH {i} {\ov{p}}{(\R^m\backslash B)\times (\rc L,\rc L\backslash \{\tv\});R}
\cong_{(2)}\nonumber &&\\
\lau \IH {i}{\ov{p}}{\R^{k+1}\times A\times (\rc L,\rc L\backslash \{\tv\});R}\cong_{(3)}\nonumber &&\\
\lau \IH {i}{\ov{p}}{\rc L,\rc L\backslash \{\tv\};R}\oplus 
\lau \IH {i-m+1+k} {\ov{p}}{\rc L,\rc L\backslash \{\tv\};R}.\nonumber&&
\end{eqnarray}
The isomorphism $\cong_{(1)}$ is the excision of $B\times (\rc L\backslash \{\tv\})$ (see \propref{prop:Excisionhomologie}),
$\cong_{(2)}$ comes from (\ref{equa:hetcone}) and $\cong_{(3)}$ from \thmref{prop:isoproduitR} and \propref{cor:SfoisX}.
Notice that $\R^m \times \rc L$, $\rc L$ and $\rc L \backslash \{ \tw \}$ are $F_\sigma$-subsets, \cite[Exercise 3H]{MR0264581}, of $X^*$ and therefore paracompact spaces \cite[Theorem 20.12 a),b)]{MR0264581}.
The hypothesis on  $i$ implies
$i-m+1+k\leq \ov p(s+1) -s+1+t \leq \ov{p}(t+1)+1$, since $\ov p$ is a GM-perversity.
Then $\lau \IH {i-m+1+k} {\ov{p}}{\rc L,\rc L\backslash \{\tv\};R}$  vanishes (see \propref{cor:homologieconerel}).
We have proved that
$\lau \IH {i} {\ov{p}}{\R^m\times \rc L\backslash h(\R^k\times \{\tw\}),\R^m\times \rc L\backslash \R^m\times\{\tv\};R}
\cong
\lau \IH {i}{\ov{p}}{\rc L,\rc L\backslash \{\tv\};R}.$
Finally, from this  isomorphism, from the long exact sequence of a pair  (\ref{equa:suiterelative2}) and from (\ref{equa:cas3debut}), we get that
$\lau \IH {i} {\ov{p}}{W;R}
\cong
\lau \IH {i} {\ov{p}}{\rc L;R}.
$ Applying \thmref{prop:coneTW} and \thmref{prop:isoproduitR} we have
$
\lau \IH {i}{\ov{p}}{\R^m\times \rc L;R}
\cong
\lau \IH {i} {\ov{p}}{\rc L;R}
\cong 
\lau  \IH{i}{\ov{p}}{W;R}\cong 
\lau \IH {i}{\ov{p}}{\rc W;R} \cong 
\lau \IH {i}{\ov{p}}{\R^k\times \rc W;R}.$
\end{proof}

\begin{theorem}\label{inv}
Suppose  $X$ is a separable paracompact CS set with no  codimension one strata.
Let $\ov p$ be  a $GM$-perversity.
Then, the intrinsic aggregation $\nu \colon X \to X^*$ induces an isomorphism
$\lau \IH * {\ov{p}}{X;R}\cong \lau \IH * {\ov{p}}{X^*;R}.$
It follows
that $\lau \IH * {\ov{p}}{X;R}$ is independent (up to isomorphism) of the choice of a stratification of $X$ as
 CS set with no  codimension one strata.

In particular, if $X'$ is another CS set with no  codimension one strata which  is homeomorphic to $X$ (not necessarily stratified homeomorphic), then
$\lau \IH * {\ov{p}}{X;R} \cong \lau \IH * {\ov{p}}{X';R}.$
These $\ov p$-depending isomorphisms can be chosen to preserve the cup product.

\end{theorem}
\begin{proof}
We consider the statement: ``The intrinsic aggregation $\nu \colon X \to X^*$  induces  a quasi-isomorphism'' and we prove it by using \propref{prop:supersuperbredon}.
We verify the four conditions of this Proposition.
\begin{itemize}

\item[(i)] The functor   $\lau \tN * {\ov p} {-;R}$ satisfies Mayer-Vietoris property (see \thmref{thm:MVcourte}).
Notice that we know from \cite[Theorem 20.12 b)]{MR0264581}
 that $X^*$ is paracompact.

\item[(ii),(iv)] Straightforward.

\item[(iii)] See \propref{prop:local}.

\end{itemize}
Last statement comes from the fact that $\nu$ is a stratified map and \thmref{MorCoho}.
\end{proof}

\section{Decomposition of the ordinary cap product.} \label{TWIC}

In this last section, we show how the ordinary cap product factorizes through the cap product we have defined for the blown-up intersection cohomology. This result extends the factorization developed in \cite[Theorem A]{CST7}.

 \bp\label{Comparacion}
 Let $X$ be a normal  separable paracompact  CS set endowed with two perversities $\ov p, \ov p$ verifying $\ov 0 \leq \ov p$ and $\ov p + \ov p \leq \ov t$. Then, there exists a commutative diagram
 $$
\xymatrix{
\lau H  * {}  {X;R}
\otimes 
\lau H  {} m  {X;R}
\ar[r]^-{\cap} 
\ar@<-4ex>[d]^\Phi
 & 
\lau H  {} {m-*}  {X;R}
  \\
 \lau \IH  *  {\ov{p}}  {X;R}
\otimes 
\lau H  {\ov{q}} m  {X;R}
 \ar[r]^-{\cap}  
 \ar@<-4ex>[u]^\phi 
& 
\lau H  {\ov{p} + \ov p} {m-*}  {X;R}
\ar[u]^\phi, 
}
$$
that is, 
$
\phi (\Phi(\alpha) \cap \xi) = \alpha \cap \phi (\xi),
$ for each $\alpha \in \lau H  * {}   {X;R}$ and $\xi \in \lau H  {\ov{q}} m  {X;R}$.
 The map $\Phi$  preserves the $\cup$-product. Moreover, if $(\ov p,\ov p) =(\ov 0,\ov t)$ then $\Phi$ and $\phi$ can be chosen  isomorphisms.
 \ep

 \begin{proof} Consider the natural inclusions $\lau \tN * {\ov 0} {X;R} \hookrightarrow \lau \tN * {\ov p} {X;R}$,
 $\lau C  {\ov q} * {X;R} \hookrightarrow \lau C  {\ov t} * {X;R}$ and 
$\lau C  {\ov p + \ov q} * {X;R} \hookrightarrow \lau C  {\ov t} * {X;R}$.
So, it suffices to prove the statement for $(\ov p,\ov q) =(\ov 0,\ov t)$
(see \remref{67}).

 The complex of singular chains (resp. cochains), with coefficients on $R$, is denoted by  $\hiru S * {X;R}$ (resp. $\Hiru S * {X;R}$), it computes the ordinary homology $\hiru H * {X;R}$ (resp. ordinary cohomology $\Hiru H * {X;R}$).
We proceed in three steps.
 
 \begin{itemize}
\item {\em Construction of $\phi$}. We consider the inclusion
$
\iota\colon \lau C {\ov t} * {X;R}\hookrightarrow  \hiru S * {X;R}.
$
We prove  the statement  :``The map $\iota$  is a  quasi-isomorphim'' by using \cite[Theorem 5.1.4]{LibroGreg}.
 Let us verify the four conditions (1)-(4) of this Theorem.
\begin{enumerate}[(1)]
\item
It is well known that the functor $\hiru S * {-;R}$ verifies the Mayer-Vietoris property.
It has been proven in  \cite[Proposition 4.1]{CST3} that     the functor   $\lau C {\ov t} * {-;R}$ also 
verifies Mayer-Vietoris property.

\item Since the support of chains are compact.

\item  Since 
\begin{eqnarray*}
\hiru H * {\R^i \times \rc L;R} &= &\hiru H 0 {\R^i \times \rc L;R}= R
\stackrel{Normal}{=} 
\hiru H 0{L;R} \\
&=&
\hiru H  0 {\R^i \times ( \rc L \menos \{ \tv\})  ;R} \\ &\stackrel{Hyp. \ (3)}{=}&
\lau H {\ov t} 0 {\R^i \times   ( \rc L \menos \{ \tv\}) ;R}
\stackrel{\hbox{\tiny \cite[Prop. 5.4]{CST3}}}{=}
\lau H {\ov t} * {\R^i \times \rc L;R}
\end{eqnarray*}

\item Straightforward.
\end{enumerate}
The differential operator $\phi$ is the isomorphism $\iota_{*}$.
 
 \item[]
 
\item {\em Construction of $\Phi$}. Consider the operator $F$ constructed in the proof of \propref{PartPer} (a) and set $\Phi = F_*$.

\item {\em Diagram commutativity}. 
This is a local question. We can consider a cochain $\omega \in \lau \tN * {} {X}$ and a regular simplex $\sigma \colon \Delta \to X$. We have:
$
\phi(\Phi(\omega ) \cap \sigma) =
\sigma_*\mu_* ((\Phi(\omega))_\sigma \tcap \tDelta)
=\sigma_* \mu_{*}(\mu^*(\sigma^* \omega)\tcap \tDelta)
\stackrel{Prop. \ref{CMAmal} (2)}{=} \sigma_*(\sigma^*\omega \cap \Delta)=
\omega \cap \sigma
=\omega \cap \phi(\sigma).
$
\end{itemize}
 \end{proof}

The case $\ov q  = \ov 0$ corresponds to the decomposition of \cite[Theorem A]{CST7}.
If we work with non positive  perversities, we have a decomposition  of the cap product as follows.
 \bp\label{Comparacion2}
 Let $X$ be a separable paracompact  CS set endowed with two perversities $\ov p < \ov 0$ and $\ov q < \ov 0$.
Then, there exists a commutative diagram
$$
\xymatrix{
\lau H * {}  {X,\Sigma;R}
\otimes 
\lau H  {} m  {X\menos \Sigma;R}
\ar[r]^-{\cap} 
\ar@<-4ex>[d]^\Phi
 & 
\lau H  {} {m-*}  {X\menos \Sigma;R}
  \\
 \lau \IH  *  {\ov{p}}  {X;R}
\otimes 
\lau H  {\ov{q}} m  {X;R}
 \ar[r]^-{\cap}  
 \ar@<-4ex>[u]^\phi 
& 
\lau H  {\ov{p} + \ov q} {m-*}  {X;R},
\ar[u]^\phi 
}
$$
that is, 
$
\phi (\Phi(\alpha) \cap \xi) = \alpha \cap \phi (\xi),
$ for each $\alpha \in \lau H  * {}   {X,\Sigma;R}$ and $\xi \in \lau H  {\ov{q}} m  {X;R}$,
 where $\Phi$, $\phi$ are isomorphisms and the first one  preserves the $\cup$-product .
\ep
 \begin{proof} We proceed in three steps.
 
 \begin{itemize}
\item {\em Construction of $\phi$}. We prove that the inclusion 
$
\lau C {\ov p} * {X;R}  \stackrel{\iota}{\hookleftarrow}  \hiru S * {X \backslash
\Sigma,R}
$
is a quasi-isomorphim. We proceed as in the previous Proposition. In fact, the only item to prove is  the property (3) of \cite[Theorem 5.1.4]{LibroGreg}.
The regular part of $\R^i \times \rc L$ is $\R^i \times  (L\backslash\Sigma_L) \times ]0,\infty[$.
 We know that $\hiru H * {\R^i \times  (L\backslash\Sigma_L) \times ]0,\infty[;R} = \hiru H * { L\backslash \Sigma_L;R}$. On the other hand, we have
$\lau H {\ov p} * {\R^i \times \rc L;R} =
\hiru H * { L\backslash \Sigma_L;R}
$
(cf. \cite[Prop. 5.4]{CST3}). Since both isomorphisms are induced by the canonical projection then we get (3).
The map $\phi$ is the isomorphism $\iota_*$.

 \item {\em Construction of $\Phi$}. Consider the operator $F$ constructed in the proof of \propref{PartPer} (b) and set $\Phi = F_*$.

\item {\em Diagram commutativity}. As in the previous Proposition.
\end{itemize}
\end{proof}

\bibliographystyle{amsalpha}

\begin{thebibliography}{A}

\bibitem{MR2607414}
M.~Banagl -- \textit{ Rational generalized intersection homology
  theories },  {Homology, Homotopy Appl.} \textbf{12} (2010), no.~1,
  p.~157--185.

\bibitem{MR1143404}
J.-P. Brasselet, G.~Hector and  M.~Saralegi --
  \textit{ Th\'eor\`eme de de {R}ham pour les vari\'et\'es stratifi\'ees },
   {Ann. Global Anal. Geom.} \textbf{9} (1991), no.~3, p.~211--243.

\bibitem{MR1700700}
G.~Bredon --  {Topology and geometry}, Graduate Texts in
  Mathematics, vol. 139, Springer-Verlag, New York, 1997, Corrected third
  printing of the 1993 original.

\bibitem{CST1}
D.~Chataur, M.~Saralegi-Aranguren and
  D.~Tanré --\textit{ Intersection Cohomology. Simplicial blow-up and
  rational homotopy.},  {ArXiv Mathematics e-prints. no. 1205.7057} (2012), To appear
  in Mem. Amer. Math. Soc.

\bibitem{CST3}
\bysame , \textit{ {Homologie d'intersection. Perversit\'es g\'en\'erales et
  invariance topologique} },  {ArXiv Mathematics e-prints. no. 1602.03009} (2016).

\bibitem{CST4}
\bysame , \textit{ {Poincaré duality with cap products in intersection
  homology} },  {ArXiv Mathematics e-prints. no. 1603.08773 } (2016).

\bibitem{CST7}
\bysame , \textit{ Singular decompositions of a cap product },  {Proceedings AMS }\textbf{145}(2017),   no.~8, p.~3645--3656.

\bibitem{CST2}
\bysame , \textit{ Steenrod squares on intersection cohomology and a conjecture of
  {M} {G}oresky and {W} {P}ardon },  {Algebr. Geom. Topol.} \textbf{16}
  (2016), no.~4, p.~1851--1904.

\bibitem{MR2209151}
G.~Friedman -- \textit{ Superperverse intersection cohomology:
  stratification (in)dependence },  {Math. Z.} \textbf{252} (2006),
  no.~1, p.~49--70.

\bibitem{MR2276609}
\bysame , \textit{ Singular chain intersection homology for traditional and
  super-perversities },  {Trans. Amer. Math. Soc.} \textbf{359} (2007),
  no.~5, p.~1977--2019 (electronic).

\bibitem{MR2461258}
\bysame , \textit{ Intersection homology {K}\"unneth theorems },  {Math.
  Ann.} \textbf{343} (2009), no.~2, p.~371--395.

\bibitem{MR2721621}
\bysame , \textit{ Intersection homology with general perversities },  {Geom.
  Dedicata} \textbf{148} (2010), p.~103--135.

\bibitem{MR2796412}
\bysame , \textit{ An introduction to intersection homology with general perversity
  functions }, in  {Topology of stratified spaces}, Math. Sci. Res. Inst.
  Publ., vol.~58, Cambridge Univ. Press, Cambridge, 2011, p.~177--222.

\bibitem{LibroGreg}
\bysame , \textit{ Singular intersection homology }, 
 {Available at http://faculty.tcu.edu/gfriedman/IHbook.pdf} (2017).

\bibitem{MR3046315}
G.~Friedman and J.~E. McClure -- \textit{ Cup and
  cap products in intersection (co)homology },  {Adv. Math.} \textbf{240}
  (2013), p.~383--426.

\bibitem{MR572580}
M.~Goresky and R.~MacPherson -- \textit{
  Intersection homology theory },  {Topology} \textbf{19} (1980), no.~2,
  p.~135--162.

\bibitem{MR696691}
\bysame , \textit{ Intersection homology. {II} },  {Invent. Math.}
  \textbf{72} (1983), no.~1, p.~77--129.

\bibitem{MR699009}
M.~Goresky and  P.~Siegel--\textit{ Linking
  pairings on singular spaces },  {Comment. Math. Helv.} \textbf{58}
  (1983), no.~1, p.~96--110.

\bibitem{MR736299}
S.~Halperin--\textit{ Lectures on minimal models },  {M\'em.
  Soc. Math. France (N.S.)} (1983), no.~9-10.

\bibitem{MR0478169}
M.~Handel--\textit{ A resolution of stratification conjectures
  concerning {CS} sets },  {Topology} \textbf{17} (1978), no.~2,
  p.~167--175.

\bibitem{MR800845}
H.~C. King--\textit{ Topological invariance of intersection homology
  without sheaves },  {Topology Appl.} \textbf{20} (1985), no.~2,
  p.~149--160.

\bibitem{RobertSF}
R.~MacPherson--\textit{ Intersection homology and perverse
  sheaves },  {Unpublished AMS Colloquium Lectures, San Francisco}
  (1991).

\bibitem{MR1245833}
M.~Saralegi--\textit{ Homological properties of stratified spaces },
   {Illinois J. Math.} \textbf{38} (1994), no.~1, p.~47--70.

\bibitem{MR2210257}
M.~Saralegi-Aranguren--\textit{ de {R}ham intersection cohomology for
  general perversities },  {Illinois J. Math.} \textbf{49} (2005), no.~3,
  p.~737--758 (electronic).

\bibitem{MR0319207}
L.~C. Siebenmann--\textit{ Deformation of homeomorphisms on stratified
  sets. {I}, {II} },  {Comment. Math. Helv.} \textbf{47} (1972),
  p.~123--136; ibid. 47 (1972), 137--163.

\bibitem{MR0264581}
S.~Willard-- {General topology}, Addison-Wesley Publishing
  Co., Reading, Mass.-London-Don Mills, Ont., 1970.

\end{thebibliography}

\end{document}